\documentclass{article}

\usepackage{graphicx} 

\usepackage[a4paper, total={6in, 8in}]{geometry}
\usepackage{fancyhdr}
\usepackage{pdfpages}
\usepackage{hyperref}
\usepackage{caption}
\hypersetup{colorlinks = true, linkcolor = purple, citecolor= blue}
\usepackage[toc,page]{appendix}
\usepackage{dsfont}
\usepackage[utf8]{inputenc}
\usepackage{bbm}
\usepackage{amsmath}
\usepackage[normalem]{ulem}
\usepackage{amsfonts}
\usepackage{amssymb}
\usepackage{amsthm}
\usepackage{graphicx} 
\usepackage{subcaption}
\usepackage[toc,page]{appendix}
\newtheorem*{theorem*}{Theorem}
\newtheorem{theorem}{Theorem}[section]

\newtheorem{lemma}[theorem]{Lemma}
\newtheorem{definition}[theorem]{Definition}
\newtheorem{remark}[theorem]{Remark}
\newtheorem*{remark*}{Remark}
\newtheorem{proposition}[theorem]{Proposition}
\usepackage{enumitem}
\usepackage[section]{placeins}
\usepackage{float}
\usepackage[toc,page]{appendix}

\usepackage[utf8]{inputenc}
\usepackage{newunicodechar}
\newunicodechar{̀}{\`{}}

\DeclareUnicodeCharacter{2217}{*}

\usepackage[
  backend=biber,
  style=numeric,
  sortcites=true,
  defernumbers=true,
  giveninits=true,   
  eprint=false,
  url=false,         
  doi=false,         
  isbn=false         
]{biblatex}
\renewcommand{\epsilon}{\varepsilon}

\usepackage{todonotes}

\title{Observability properties of the singular Grushin equation}
\author{Roman Vanlaere \footnote{CEREMADE, Université Paris-Dauphine PSL, CNRS UMR 7534, 75016 Paris, France}\footnote{Current address: IMUS, Universidad de Sevilla, Spain, rvanlaere@us.es.}}
\date{2026}

\begin{document}

\maketitle

\begin{abstract}
    We study the observability properties of the Grushin equation with an inverse square potential, whose singularity occurs at the boundary of two-dimensional rectangular domains or in the interior of the domain in higher dimensions. In some specific configurations of the observation set, we obtain the exact minimal time of observability. The analysis we present relies on recent Carleman estimates obtained by K. Beauchard, J. Dardé, and S. Ervedoza. As a byproduct of these results, we observe, for the heat equation associated to the Laplace-Beltrami operator on almost-Riemannian manifolds, a dependence of the minimal time of observability on the dimension of the singularity. 
\end{abstract}

\tableofcontents

\section{Introduction}

The goal of this paper is to study the observability properties of the heat equation associated with Grushin-type operators perturbed by an inverse-square potential. Such properties were previously studied in \cite{cannarsa2014null,anh2016null}, where the authors proved only the existence of a minimal time of observability for certain configurations of the observation set. Our main contribution is to go further and give an explicit characterization of this minimal time in some specific geometric settings, highlighting its dependence on both the observation set and the strength of the singularity. This explicit characterization reveals that the minimal time for the Laplace-Beltrami operator arising in almost-Riemannian geometry depends on the dimension of the singular set. \\

We also study a generalization of these Grushin-type operators in the two-dimensional setting and provide a lower bound on the minimal observability time. This bound also depends on both the observation set and the strength of the singularity, and is expected to be optimal.

\subsection{Setting and main results for the classical operator}\label{section: introduction classical operator}

Let $\Omega = \Omega_x \times \Omega_y$ be a connected compact manifold with smooth boundary, with $\Omega_x \subset \mathbb{R}^{d_x}$ a bounded domain such that $d_x = 1$ or $d_x \geq 3$, and $\Omega_y$ a compact Riemannian manifold of dimension $d_y \geq 1$ possibly with boundary. The main results of this section hold under the following assumptions. 
\begin{enumerate}[label = {$\operatorname{H}_{\arabic*}$}, ref = $\operatorname{H}_{\arabic*}$]
    \item \label{assumption: d = 1} If $d_x = 1$, then $\Omega_x = (0,L)$ for some $L > 0$.
    \item \label{assumption: d > 3} If $d_x \geq 3$, then $0 \in \Omega_x$.
\end{enumerate}

\begin{remark}
    In this paper we do not address the case $d_x = 2$. This is first due to the functional framework (see the well-posedness Section \ref{section: well-posedness}). When $d_x = 2$ and $0_{ \mathbb{R}^{d_x}} \in  \Omega_x$, the Hardy constant (recalled in \eqref{eqn: definition hardy constant}) vanishes, and the functional framework of Section \ref{section: well-posedness} fails. Thus, for it to hold, one must take $0_{ \mathbb{R}^{d_x}} \in \partial \Omega_x$. Moreover, when $d_x = 2$ our spectral analysis strategy of Section \ref{section: spectral analysis classical d > 3} is no longer valid, as it requires both $0_{ \mathbb{R}^{d_x}} \in \Omega_x$ and the functional framework of Section \ref{section: well-posedness} (see, in particular, Propositions \ref{prop: first eigenvalue in the ball} and \ref{prop: first eigenvalue in high dimension}). Setting $d_x = 2$ and $0_{ \mathbb{R}^{d_x}} \in \Omega_x$ requires carefully checking the spectral analysis within the new functional setting. Thus, to keep the framework of this paper unified and avoid any confusion, we exclude $d_x=2$, but expect our results to extend to this setting when $0_{ \mathbb{R}^{d_x}} \in \Omega_x$. Notably, obtaining similar results as ours when $0_{ \mathbb{R}^{d_x}} \in \partial \Omega_x$, $d_x \geq 2$, remains open. 
\end{remark}

We shall use the notation $(x,y) \in \Omega$ with $x = (x_1,...,x_{d_x}) \in \Omega_x$ and $y \in \Omega_y$, and denote by $\Delta_y$ the Laplace-Beltrami operator on $\Omega_y$, and by $\Delta_x$ the usual Laplace operator on $\Omega_x$.\\

Let $T > 0$. Our first system of interest is the following classical Grushin-type equation with a singular potential, 
\begin{align} \label{stm : adjoint system classical}
\left\{ \begin{array}{lcll}
     \partial_t f - \Delta_xf - |x|^{2\gamma}\Delta_y f + \displaystyle\frac{\nu^2-\mathsf{H}}{|x|^2}f & = & 0, & (t,x,y) \in (0,T) \times \Omega,\\[6pt]
     f(t,x,y) & = & 0, & (t,x,y) \in (0,T) \times \partial \Omega, \\[6pt]
     f(0,x,y) & = & f_0(x,y), & (x,y) \in \Omega,
    \end{array} \right.
\end{align}
where $f_0 \in L^2(\Omega)$, $\mathsf{H}$ is the Hardy constant \eqref{eqn: definition hardy constant} given below, which depends on the geometric setting for $\Omega_x$ (see \textit{e.g.} \cite{hardy1952inequalities,cazacu2014controllability}),
\begin{align}\label{eqn: hardy inequality}
    \mathsf{H}\int_{\Omega_x} \frac{|u|^2}{|x|^2} \ dx \leq \int_{\Omega_x}  |\nabla_xu|^2 \ dx, \quad \text{for every } u \in H_0^1(\Omega_x),
\end{align}
where 
\begin{align}\label{eqn: definition hardy constant}
    \begin{array}{cccl}
      \mathsf{H}  &=& d_x^2/4 & \text{if } 0 \in \partial\Omega_x, \\[6pt]
      \mathsf{H}  &=& (d_x - 2)^2/4 & \text{if } d_x \geq 3 \text{ and } 0 \in \Omega_x,
    \end{array}
\end{align}
and $\gamma \geq 1$ and $\nu > 0$ are fixed parameters. The constant $\mathsf{H}$ given in \eqref{eqn: definition hardy constant} is the best constant such that \eqref{eqn: hardy inequality} holds. When $0 \in \Omega_x$, the formula $\mathsf{H} = (d_x - 2)^2/4$ also applies for $d_x = 2$, but gives $\mathsf{H} = 0$, and \eqref{eqn: hardy inequality} reduces to a trivial inequality. This is why when $0 \in \Omega_x$ we required $d_x \geq 3$ in \eqref{eqn: definition hardy constant}, so that $\mathsf{H} > 0$. Well-posedness of system \eqref{stm : adjoint system classical} is related to Hardy's inequalities \eqref{eqn: hardy inequality}.\\

System \eqref{stm : adjoint system classical} is motivated by the following, and we refer to Section \ref{section: additional comments} for extended discussions. If $d_x = 1$, and we assume without loss of generality that $\Omega_y \subset \mathbb{R}^{d_y}$ is a smooth bounded domain, a model almost-Riemannian structure on $\Omega$ is generated by the vector fields $X = \partial_x$, and $Y_i = x^{\gamma} \partial_{y_i}$, $1 \leq i \leq d_y$. The corresponding Laplace-Beltrami operator is given by 
\begin{align*}
    \Delta_\gamma f = -\partial_x^2 f - |x|^{2\gamma}\Delta_y f + \frac{\gamma d_y}{x}\partial_x f,
\end{align*}
on $L^2(\Omega, x^{-\gamma d_y}dxdy)$, with $\Delta_y$ the usual Laplace operator. The change of variable $f = |x|^{\gamma d_y/2}g$ leads to 
\begin{align*}
    L_\gamma g = -\partial_x^2 g - |x|^{2\gamma}\Delta_y g + \frac{\gamma d_y}{2} \left( \frac{\gamma d_y}{2} +1 \right) \frac{g}{x^2},
\end{align*}
on $L^2(\Omega)$ for the Lebesgue measure. In the case $d_x = 1$, the operator $-\partial_x^2  - x^{2\gamma}\Delta_y + (\nu^2 - 1/4)/x^2$ is therefore simply a generalization of $L_\gamma$, dissociating the effects of the degenerate and singular terms, and allowing to consider a broader range of operator similar to $L_\gamma$. We then study a natural generalization of these $(d_y+1)$-dimensional operators in higher dimensions for $d_x$.

\begin{definition}[Observability]
Let $T>0$ and $\omega \subset \Omega$ an open set. We say that system \eqref{stm : adjoint system classical} is observable from $\omega$ in time $T > 0$ if there exists $C>0$ such that, for every $f_0 \in L^2(\Omega)$, the associated solution satisfies
\begin{align}\label{eqn: observability inequality Omega}
\int_\Omega |f(T,x,y)|^2 \ dx \ dy \leq C \int_0^T \int_{\omega} |f(t,x,y)|^2 \ dx \ dy \ dt.
\end{align}
\end{definition}

We denote the minimal time of observability by 
\begin{align}\label{eqn: def T(omega)}
    T(\omega) &:= \inf \{ T > 0, \text{ such that the system is observable in time $T$ from $\omega$} \}.
\end{align}
Observe that $T(\omega)=\infty$ if and only if the system of interest is not observable independently of the final time $T>0$, and $T(\omega) = 0$ if and only if the system of interest is observable in any time. Our first two results are the following. 

\begin{theorem}\label{thm: main theorem control classical grushin rectangle vertical}
    Consider system \eqref{stm : adjoint system classical} with $\gamma = 1$ and $\nu > 0$, and assume either \ref{assumption: d = 1} or \ref{assumption: d > 3}. Set $\omega = \omega_x \times \Omega_y \subset \Omega$ to be an open set such that $0_{\mathbb{R}^{d_x}} \notin \overline{\omega}_x \subset \Omega_x$. Then, the minimal time of observability of system \eqref{stm : adjoint system classical} from $\omega$ satisfies 
    \begin{align}\label{eqn: lower bound classical}
        T(\omega) \geq \frac{\operatorname{dist}(0_{\mathbb{R}^{d_x}},\omega_x)^2}{4(1+\nu)}.
    \end{align}   
    If moreover $\omega_x$ satisfies that there exists a smooth domain $\mathcal{O} \subset\subset \Omega_x$ containing $0_{\mathbb{R}^{d_x}}$ such that $\omega_x \subset \Omega_x \setminus \mathcal{O}$ and $\partial \mathcal{O} \subset \partial \omega_x$ (see Figure \ref{fig: geometric_assumption}), then
    \begin{align}
        T(\omega) \leq \frac{\sup \{|x|^2, x \in \partial\mathcal{O}\}}{4(1+\nu)}.
    \end{align}
    We conjecture that the lower bound \eqref{eqn: lower bound classical} is sharp for any open set $\omega_x \subset \Omega_x$. Namely, for any open set $\omega_x \subset \Omega_x$ and $\omega = \omega_x \times \Omega_y$, we conjecture that 
    \begin{align*}
        T(\omega) = \frac{\operatorname{dist}(0_{\mathbb{R}^{d_x}},\omega_x)^2}{4(1+\nu)}.
    \end{align*}
\end{theorem}

\begin{remark}
    As claimed via the conjecture in Theorem \ref{thm: main theorem control classical grushin rectangle vertical}, we believe that the geometric assumption on $\omega_x$ is not necessary. It stems from our method for proving observability in sufficiently large time. Our approach relies on \textit{boundary} Carleman estimates, which first yield boundary observability. We then extend this to interior observability by passing through the equivalent (by duality) controlled systems and using cutoff arguments. It is this last step that creates the geometric restriction. We make this precise in Section \ref{section : sketch of proof carleman}, where we also discuss a strategy to remove the geometric assumption.
\end{remark}

\begin{theorem}\label{thm: main theorem control classical grushin rectangle vertical 2}
    Consider system \eqref{stm : adjoint system classical} with $\gamma > 1$ and $\nu > 0$, and assume either \ref{assumption: d = 1} or \ref{assumption: d > 3}. Set $\omega = \omega_x \times \Omega_y \subset \Omega$ to be an open set such that $0_{\mathbb{R}^{d_x}} \notin \overline{\omega_x} \subset \Omega_x$. Then, system \eqref{stm : adjoint system classical} is never observable from $\omega$, \textit{i.e.}
    \begin{align}
        T(\omega) = + \infty.
    \end{align}
\end{theorem}

\begin{figure}[H]
    \centering
    \includegraphics[width=0.35\linewidth]{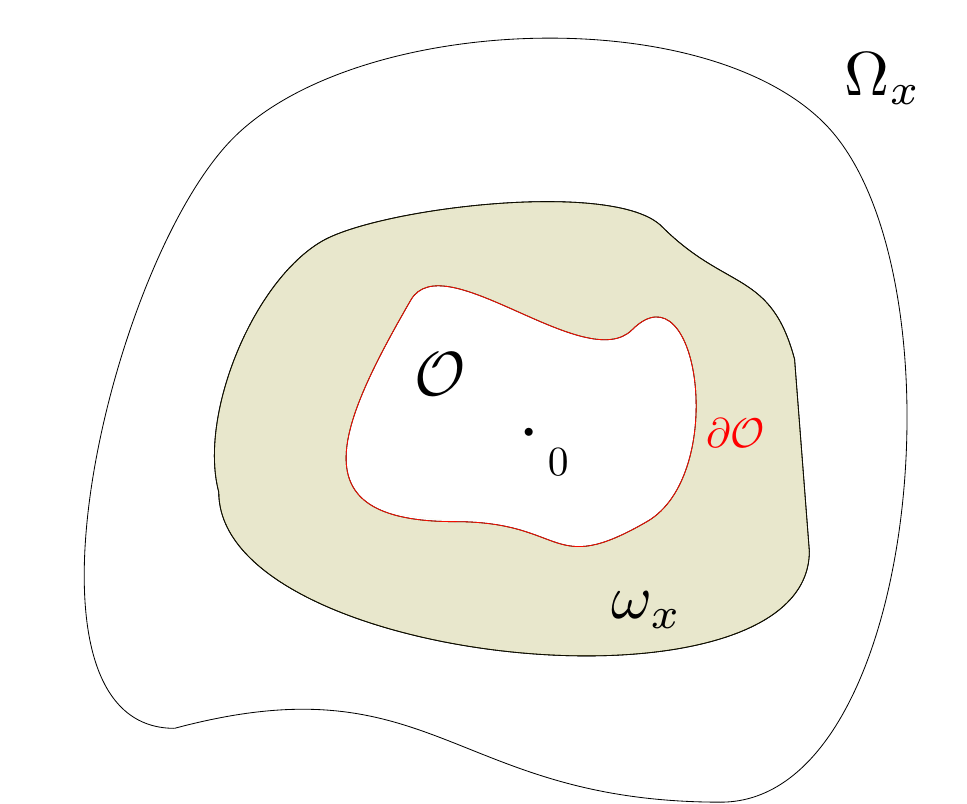}
    \caption{The geometric assumption on the observation set in Theorem \ref{thm: main theorem control classical grushin rectangle vertical}: the observation set is in green, whose boundary contains the boundary in red of a subdomain $\mathcal{O}$ containing $0$.}
    \label{fig: geometric_assumption}
\end{figure}

A direct corollary of Theorem \ref{thm: main theorem control classical grushin rectangle vertical} is that if $\gamma = 1$ and there exist $0<r<R$ such that $\omega = \{ r < |x| < R, x \in \Omega_x \} \times \Omega_y$, then
\begin{align*}
    T(\omega) = \frac{r^2}{4(1+\nu)}.
\end{align*}
Thus, the minimal time is obtained under the geometric restriction that $\omega_x$ forms a ring around $0$. We discuss this restriction in Section \ref{section : sketch of proof carleman}, and we conjecture that the lower bound in Theorem \ref{thm: main theorem control classical grushin rectangle vertical} is sharp for any open subsets $\omega_x \in \Omega_x$. \\

In dimension $d_x = 1$, this restriction is vacuous, and Theorem \ref{thm: main theorem control classical grushin rectangle vertical} explicitly characterizes the minimal time of observability previously shown to exist in \cite{cannarsa2014null}: for any $\omega = (a,b) \times \Omega_y$, $0 < a < b \leq L$, the minimal time of observability from $\omega$ of system \eqref{stm : adjoint system classical} posed on $(0,L) \times \Omega_y$ is
\begin{align}\label{eqn: minimal time obs d=1}
    T((a,b) \times \Omega_y) = \frac{a^2}{4(1+\nu)}.
\end{align}

The Laplace-Beltrami operator thus corresponds to $\nu^2 = (\gamma d_y + 1)^2/4$, for which the minimal time of observability that we obtain when $\gamma = 1$, from \eqref{eqn: minimal time obs d=1}, is 
\begin{align}\label{eqn: minimal time laplace beltrami}
    T((a,b) \times \Omega_y) = \frac{a^2}{6 + 2d_y}.
\end{align}

Another direct corollary, from Theorem \ref{thm: main theorem control classical grushin rectangle vertical 2}, is that if $\gamma > 1$ and $\left(\{0_{\mathbb{R}^{d_x}}\} \times \Omega_y \right) \cap \overline{\omega} = \emptyset$, then $T(\omega) = + \infty$. We briefly discuss the case $\gamma \in (0,1)$ in Section \ref{section : sketch of proof carleman}.\\

\subsection{Setting and main results for the generalized operator}

Let us now present, when $d_x = d_y = 1$, a generalization of the operators presented in the preceding section. \\

Let $\Omega = (0,L) \times (0,\pi)$. Let $\gamma \in \mathbb{N}^*$, and let $q$ be a function satisfying
\begin{enumerate}[label=$\operatorname{H}_q$,ref=$\operatorname{H}_q$]
    \item \label{assumption: q} $q \in C^{\gamma}([0,L])$ and
    \begin{align*}
        \partial_x^k q(0) = 0 \text{ for every } k \in \{0, ..., \gamma -1\}, \partial_x^\gamma q(0) > 0 \text{, and } q(x) > 0 \text{ for every } x > 0.
    \end{align*}
\end{enumerate}

We study the observability properties of the following system. Let $T > 0$, we consider 
\begin{align}\label{stm : adjoint system generalized}
    \left\{ \begin{array}{lcll}
     \displaystyle \partial_t f -\partial_x^2 f - q(x)^2\partial_y^2 f + \frac{\nu^2 - 1/4}{x^2} f &= &0, & (t,x,y) \in (0,T) \times \Omega,\\[6pt]
     f(t,x,y) & = & 0, & (t,x,y) \in (0,T) \times \partial \Omega, \\[6pt]
     f(0,x,y) & = & f_0(x,y), & (x,y) \in \Omega,
    \end{array} \right.
\end{align}
where $f_0 \in L^2(\Omega)$, $q$ satisfies \ref{assumption: q} for some $\gamma \geq 1$, and $\nu > 0$ is a fixed parameter.\\

To state our result, let us introduce the Agmon distance for a general function $q$ assumed to be continuous. For any $E \geq 0$, and any dimension $d_x \geq 1$, the Agmon distance is defined by
\begin{align}\label{eqn: definition Agmon distance}
    d_{\operatorname{agm},E} : (x,x') \in \mathbb{R}^{2d_x}  \mapsto \underset{c \in C^{1,\operatorname{pw}}([0,1];x,x')}{\inf} \int_0^1 \left( q(c(t))^2 - E \right)_+^{1/2} |c'(t)| \ dt,
\end{align}
where the $\inf$ is taken on the set of piecewise $C^1$ paths joining $x$ to $x'$.\\

When taking $E=0$, and $q$ satisfying $q(0) = 0$ and $q(x) \neq 0$ for every $x \neq 0$, we shall simply write the Agmon distance from $x$ to $0$ by $d_{\operatorname{agm}}(x)$. That is, by \eqref{eqn: definition Agmon distance}, if $d_x = 1$, $x \in [0,L]$ and $q$ satisfies \ref{assumption: q}, we have 
\begin{align*}
    d_{\operatorname{agm}}(x) = \int_0^{x} q(s) \ ds. 
\end{align*}

Observe that, letting $\delta > 0$ be such that the set $F_\delta = \{x \in \Omega_x, q(x)^2 \leq \delta\}$ is a connected neighborhood of $0$, which is possible by assumption, for any fixed $x \in \Omega_x$, we have
\begin{align*}
    \lim_{\delta \rightarrow 0^+} d_{\operatorname{agm},\delta} (x,F_\delta) = d_{\operatorname{agm}}(x).
\end{align*}

\begin{theorem}\label{thm: main theorem control generalized grushin rectangle vertical}
    Let $\nu > 0$. Assume that $q$ satisfies \ref{assumption: q} with $\gamma = 1$. For any $0 < a < b \leq L$, the minimal time of observability of system \eqref{stm : adjoint system generalized} from $\omega = (a,b) \times (0,\pi)$ satisfies 
    \begin{align}
        T(\omega) \geq \frac{d_{\operatorname{agm}}(a)}{2q'(0)(1+\nu)}.
    \end{align}   
\end{theorem}

Observe that if $q(x)= x$, then 
\begin{align*}
    \frac{d_{\operatorname{agm}}(a)}{2q'(0)(1+\nu)} = \frac{a^2}{4(1+\nu)},
\end{align*}
and we recover the minimal time of Theorem \ref{thm: main theorem control classical grushin rectangle vertical} when $d_x = 1$. The minimal time in Theorem \ref{thm: main theorem control classical grushin rectangle vertical} is in fact related to the Agmon distance associated to the radial function $x \mapsto |x|^2$. Indeed, in the high-dimensional setting,
\begin{align*}
   d_{\operatorname{agm}}(x)  = \int_0^1 s|x|^2 \ ds = \frac{|x|^2}{2}. 
\end{align*}

\subsection{Comments on the results and the literature}\label{section: comments results and literature}

By a duality argument (see \textit{e.g.} \cite[Proposition 2.48]{coron2007control}), it is well-known that observability from $\omega$ in time $T>0$ is equivalent to the null-controllability in time $T>0$ from $\omega$ of the same system with a control supported on $\omega$ as a source term. Therefore, our results can be translated in terms of null-controllability properties.\\

In the case $\nu^2 = \mathsf{H}$, the singular term of any of our systems vanishes, and we recover in systems \eqref{stm : adjoint system classical} and \eqref{stm : adjoint system generalized} the non-singular Grushin equation, with a boundary degeneracy if $d_x = 1$, and interior degeneracy if $d_x \geq 3$. In particular, in Theorem \ref{thm: main theorem control classical grushin rectangle vertical}, if $d_x = 1$ we recover an analogous result to \cite[Theorem 1.3]{beauchard2020minimal}, and if $d_x \geq 3$, to \cite[Theorem 1.1]{beauchard2020minimal} (they consider boundary observability in these theorems). Indeed, the minimal time of Theorem \ref{thm: main theorem control classical grushin rectangle vertical} is given, if $\nu^2 = \mathsf{H}$ and $\omega_x = \{ r < |x| < R\}$, by
\begin{align*}
    T(\omega_x \times \Omega_y) = 
    \left\{
    \begin{array}{cc}
        r^2/6 & \text{if } d_x = 1,  \\[6pt]
         r^2/2d_x& \text{if } d_x \geq 3. 
    \end{array}
    \right.
\end{align*}

The study of observability properties of the singular Grushin equation has not attracted as much attention in the literature as the non-singular case. Approximate controllability has been obtained in \cite{morancey2015approximate} when $d_x = d_y = 1$, with an interior singularity. Null-controllability has been studied when $d_x = d_y = 1$ in \cite{cannarsa2014null}, and when $d_x\geq 3$ under assumption \ref{assumption: d > 3} in \cite{anh2016null}. In \cite{cannarsa2014null,anh2016null}, they prove the existence of a minimal time, but they do not characterize it. Thus, Theorem \ref{thm: main theorem control classical grushin rectangle vertical} aims to complete their results, and succeeds to for $d_x = 1$ and in some particular settings for $d_x \geq 3$.\\

When $d_x = \gamma = 1$, the Laplace-Beltrami operator corresponds to $\nu^2  = (d_y + 1)^2/4$, for which we recall that the result we obtain in Theorem \ref{thm: main theorem control classical grushin rectangle vertical} is 
\begin{align*}
    T((a,b) \times \Omega_y) = \frac{a^2}{6+2d_y}.
\end{align*}
It is interesting to observe that the minimal time of observability for the Laplace–Beltrami operator depends on the dimension of $\Omega_y$, as for the singular equation in general it depends on the strength of the singular term. Indeed, both the non-singular ($\nu^2=\mathsf{H}$) and singular operators have the same Weyl law on the manifold $\Omega$, but exhibit different times of observability. Moreover, in the observation set, the exponential decay of a subsequence of eigenfunctions of these operators (roughly $\sim e^{-n\operatorname{dist(\omega,0)}^2}$, see Section \ref{section : agmon estimates}), as well as their localization from below (by applying Proposition \ref{prop: appendix cost of small time observability} to the eigenfunctions; see \textit{e.g.} \cite{LaurentLeauthaud23} for sharp estimates in the non-singular case), do not depend on the strength of the singularity. Hence, it seems that there is an additional abstract factor in play to explain this difference. It is for us and for now, simply due to influence of the singular potential on the speed of dissipation of the solutions of the Fourier components of our systems (see Section \ref{section: fourier decomposition} and Section \ref{section : sketch of proof carleman}). This observation allows us to conjecture a dependence of the minimal time of observability for the heat equation on almost-Riemannian manifolds, on the choice of a measure. We shall extend this remark in Section \ref{section: additional comments}.\\

As already said, when the singular term vanishes, we recover the classical non-singular Grushin equation. It has now been extensively studied in the two-dimensional setting $\Omega = (-L_-,L_+) \times \Omega_y$, where $\Omega_y$ is the one-dimensional torus or a finite interval. The pioneer work on the subject is \cite{beauchard2014null}, where the control is a vertical strip. It is shown that the minimal time depends on the strength of the degeneracy. If $\gamma < 1$, null-controllability occurs in any time, if $\gamma > 1$ it is never null-controllable, while in the critical case $\gamma = 1$, a minimal amount of time is required. This analysis is extended in higher dimensions in \cite{beauchard2014inverse}. The minimal time in the two dimensional setting is obtained, when the domain and the control zone are symmetric, and the control zone is a union of vertical strips, in \cite{beauchard20152d} using the transmutation method, and in \cite{allonsius2021analysis} by the moments method. Then, for the generalized equation, when the control is any vertical strip at the boundary, the minimal time is obtained in \cite{beauchard2020minimal}. Additionally to these results, when the control is the complementary of an horizontal strip and $\gamma = 1$, it has been obtained for the classical equation on the plane that it is never observable in \cite{LISSYprolate2025}, and in the compact setting for both the classical and generalized operators, in the series of papers \cite{duprez2020control,darde2023null,CRMATH_2017__355_12_1215_0}. In \cite{darde2023null}, a broader range of control zones is investigated. This shows that a geometric condition is necessary. Finally, in \cite{tamekue2022null}, the equation is studied on the Grushin sphere, and in \cite{vanlaere2025grushinlike} on two-dimensional manifolds via a reduction argument. \\

Concerning other null-controllability results for a class of degenerate parabolic equations, we can also cite \cite{beauchard2017heat}, that studies the heat equation on the Heisenberg group, or results concerning the Kolmogorov equation \cite{LEROUSSEAU20163193, beauchardKolmog2015, koenigfractionalheat, beauchardkolmogorov, DardeRoyerKolmogorov}. Concerning the heat equation in presence of a singular potential in the interior, we refer to \cite{vancostenoble2008null,ervedoza2008control}, and to \cite{cazacu2014controllability} when the singular potential is localized at the boundary.

\subsection{Structure of the paper}

The present work is organized as follows. We begin with two preliminary sections: 
\begin{enumerate}[label = \scalebox{0.5}{$\bullet$}]
    \item Section \ref{section: Well-posedness and Fourier decomposition} contains some preliminary remarks on our systems. Namely, we discuss the well-posedness of our systems in Section \ref{section: well-posedness}, and their Fourier decompositions in Section \ref{section: fourier decomposition}. The Fourier decompositions presented in Section \ref{section: fourier decomposition} are of importance for the proofs of our theorems. 
    \item We sketch the proofs of our results in Section \ref{section: sketch of proofs}. 
\end{enumerate}

Next, we have the sections containing the technicalities regarding the analysis of the Fourier components of our systems \eqref{stm : adjoint system classical} and \eqref{stm : adjoint system generalized}, as well as the proofs of our main results: 

\begin{enumerate}[label = \scalebox{0.5}{$\bullet$}]
    \item Section \ref{section: boundary observability via carleman estimates} presents some boundary Carleman estimates and the inferred costs of small-time boundary observability. The Carleman estimates are proved in Section \ref{section: carleman estimate classical}. Since the Carleman weights are inspired by the heat kernels on the half-line for the Fourier components of our operators, these kernels are computed in Section \ref{section: heat kernel computation}. We then derive from these the cost of small-time boundary observability in Section \ref{section: boundary observability}.
    \item Section \ref{section: spectral analysis} is devoted to the spectral analysis of the Fourier components. We establish general Agmon estimates in Section \ref{section : agmon estimates}. The spectrum of the classical operators is then analyzed. When $d_x = 1$, it is done in Section \ref{section: spectral analysis classical gamma = 1} for $\gamma = 1$, and in Section \ref{section: spectral analysis classical gamma > 1} for $\gamma > 1$. In dimension $d_x \geq 3$, and in any case for $\gamma \geq 1$, the analysis is performed in Section \ref{section: spectral analysis classical d > 3}. The generalized operators are studied in Section \ref{section: spectral analysis generalized}.
    \item Section  \ref{section: proofs} gathers the proofs of the theorems.
    The upper bound in Theorem \ref{thm: main theorem control classical grushin rectangle vertical} is proved in Section \ref{section: proof positive result}. The other results, concerning non-observability, are proved in Section \ref{section: proofs negative results}.
\end{enumerate}

Finally, Section \ref{section: additional comments} contains further remark complementing those presented of Section \ref{section: comments results and literature}, in particular regarding extensions to manifolds. We also list some questions that remain to be answered.\\

Appendix \ref{section: appendix The cost of small time observability from interior subsets} provides an alternative proof of Theorem \ref{thm: main theorem control classical grushin rectangle vertical} when $d_x=1$, and is also motivated by the discussions of Section \ref{section: sketch of proofs}.

\section{Well-posedness and Fourier decomposition}\label{section: Well-posedness and Fourier decomposition}

Since the treatments of systems \eqref{stm : adjoint system classical} and \eqref{stm : adjoint system generalized} are similar, we discuss in this section the following system, on $\Omega = \Omega_x \times \Omega_y$, for a sufficiently regular function $q$,
\begin{align} \label{stm : general system}
\left\{ \begin{array}{lcll}
     \partial_t f - \Delta_x f - q(x)^2\Delta_y f + \displaystyle\frac{\nu^2-\mathsf{H}}{|x|^2}f & = & 0, & (t,x,y) \in (0,T) \times \Omega,\\[6pt]
     f(t,x,y) & = & 0, & (t,x,y) \in (0,T) \times \partial \Omega, \\[6pt]
     f(0,x,y) & = & f_0(x,y), & (x,y) \in \Omega,
    \end{array} \right.
\end{align}
encompassing them both. Unless specified otherwise in the below discussions, $0_{\mathbb{R}^{d_x}}$ may belong to either the interior or the boundary of $\Omega_x$, as long as the corresponding Hardy constant $\mathsf{H}$ defined in \eqref{eqn: definition hardy constant} is positive, \textit{i.e.}, if $d_x = 2$, we impose $0 \in \partial \Omega_x$.

\subsection{Well-posedness} \label{section: well-posedness}

The well-posedness of system \eqref{stm : general system} is already treated for $q(x) = x$ and $d_x = d_y = 1$ in \cite[Section 2]{beauchard2014null} and \cite[Section 2]{cannarsa2014null}, and in a higher dimensional setting in \cite{anh2016null}. We recall it for the sake of self-completeness.  \\

Denote by $V := \overline{C_0^\infty(\Omega)}^{|\cdot|_V}$ the completion of $C_0^\infty(\Omega)$ under the norm $|\cdot|_V$ induced by the scalar product (thanks to Hardy's inequalities \eqref{eqn: hardy inequality})
\begin{align}
    (u,v)_V = \int_\Omega \nabla_x u \cdot \nabla_x v + q(x)^2\nabla_y u \cdot \nabla_y v + \frac{\nu^2 - \mathsf{H}}{|x|^2}uv \ dx \ dy, \quad \text{for every } u,v \in C_0^\infty(\Omega).
\end{align}

Denote by $W := \overline{C_0^\infty(\Omega)}^{|\cdot|_W}$ the completion of $C_0^\infty(\Omega)$ under the norm $|\cdot|_W$ induced by the scalar product
\begin{align}
    (u,v)_W = \int_\Omega \nabla_x u \cdot \nabla_x v + q(x)^2\nabla_y u \cdot \nabla_y v \ dx \ dy, \quad \text{for every } u,v \in C_0^\infty(\Omega).
\end{align}

We observe that thanks to Hardy inequalities \eqref{eqn: hardy inequality}, since $\nu > 0$, we have $H_0^1(\Omega) \subset V \subset W$, and so $V$ is dense in $L^2(\Omega)$. \\

Set $L_q^2(\Omega) := L^2(\Omega, q(x)dxdy)$. Following \cite[Lemma 1]{beauchard2014null}, every $f \in W$ admits weak derivatives $\partial_{x_i} f \in L^2(\Omega)$ and $\partial_{y_i} f \in L_q^2(\Omega)$. We consider the Friedrich extension of $G$ with minimal domain $C_0^\infty(\Omega)$, 
\begin{align}
    \begin{array}{lll}
        D(G_F) &=& \{f \in V, \ \exists c > 0 \text{ such that } |(f,g)_V| \leq c \|g\|_{L^2(\Omega)}, \text{ for every } g \in V\}, \\[6pt]
         (G_F f,g)_{L^2(\Omega)} &=& (f,g)_V. 
    \end{array}
\end{align}

Notice that this implies that $G_F = -\Delta_x^2 - q(x)^2\Delta_y^2 + (\nu^2 - \mathsf{H})/|x|^2$, so that we shall keep the notation $G$ for $G_F$. The operator $(G,D(G))$ is self-adjoint and positive on $L^2(\Omega)$. \\

The operator $(G,D(G))$ generates an analytic semigroup of contraction $(e^{-tG})_{t \geq 0}$ (see \textit{e.g.} \cite{zabczyk2020mathematical}). We define the weak solution of system \eqref{stm : general system} with a source term $h \in L^2((0,T) \times \Omega)$.

\begin{definition}
    A function $f \in C([0,T], L^2(\Omega)) \cap L^2((0,T), V)$ is solution of system \eqref{stm : general system} with a source term $h \in L^2((0,T) \times \Omega)$, if for every $g \in D(G)$, the mapping $t \in [0,T] \mapsto (f(t),g)_{L^2(\Omega)}$ is absolutely continuous, and for almost every $t \in [0,T]$,
    \begin{align}
        \frac{d}{dt}(f(t),g)_{L^2(\Omega)} = (f(t),Gg)_{L^2(\Omega)} + (h(t),g)_{L^2(\Omega)}.
    \end{align}
\end{definition}

From \cite{ball1977strongly}, this notion of solution is equivalent to the one given by the Duhamel formula 
\begin{align}\label{eqn: duhamel formula}
    f(t) = e^{-tG}f_0 + \int_0^t e^{-(t-s)G}h(s) \ ds, \quad t \in [0,T].
\end{align}

\begin{proposition}{\cite[Proposition 1]{cannarsa2014null}}
  For every $f_0 \in L^2(\Omega)$, $T > 0$, $h \in L^2(0,T;L^2(\Omega))$, there exists a unique weak solution of system \eqref{stm : adjoint system generalized} with a source term $h \in L^2((0,T) \times \Omega)$. Moreover, 
  \begin{align}
      \|f(t)\|_{L^2(\Omega)} \leq \|f_0\|_{L^2(\Omega)} + \sqrt{T}\|h\|_{L^2(0,T;L^2(\Omega))}, 
  \end{align}
  and $f(t) \in D(G)$, $f'(t) \in L^2(\Omega)$, for almost every $t \in (0,T)$.
\end{proposition}

\subsection{Fourier decomposition}\label{section: fourier decomposition}

For any of our results, we follow the usual strategy (first introduced in \cite{beauchard2014null}) of Fourier decomposition of system \eqref{stm : general system}. The first step is to take advantage of the tensorized structure of our systems, the domain, and the observation set $\omega$ that we decompose as $\omega = \omega_x \times \Omega_y$. Denoting by $\phi_n$ the eigenfunctions of $-\Delta_y$, $\phi_n$ being associated to $\xi_n^2$, $n \in \mathbb{N}^*$, we can write any solution of system \eqref{stm : adjoint system classical} as 
\begin{align}\label{eqn: fourier decomposition}
    f(t,x,y) = \displaystyle \sum_n f_n(t,x) \phi_n(y), \quad \text{with} \quad f_n(t,x) = \displaystyle \int_{\Omega_y} f(t,x,y) \phi_n(y) \ dy.
\end{align}
For every $n \geq 1$, the Fourier coefficients $f_n \in C([0,T], L^2(\Omega_x))\cap L^2(0,T; H_0^1(\Omega_x))$ are the unique weak solutions of (see \cite[Proposition 2]{cannarsa2014null} or \cite[Proposition 2.1]{anh2016null})
\begin{align}\label{stm : adjoint system generalized fourier}
    \left\{ \begin{array}{lcll}
     \partial_t f_n -\Delta_x f_n + \xi_n^2q(x)^2 f_n + \displaystyle\frac{\nu^2 - \mathsf{H}}{|x|^2} f_n & = & 0, & (t,x) \in (0,T) \times \Omega_x,\\[6pt]
     f_n(t,x) & = & 0, & t \in (0,T), \ x \in \partial\Omega_x, \\[6pt]
     f_n(0,x) & = & f_{n,0}(x), & x \in \Omega_x.
    \end{array} \right.
\end{align}

We denote by $G_\xi$ the operator 
\begin{align}\label{eqn: definition G_xi}
    \begin{array}{cll}
        D(G_\xi) &=&  \{f \in H_0^1(\Omega_x), \ G_\xi f \in L^2(\Omega_x)\}, \\[6pt]
        G_\xi &=& -\Delta_x + \xi^2q(x)^2 + \displaystyle\frac{\nu^2 - \mathsf{H}}{|x|^2}.
    \end{array}
\end{align}
The domain of $G_\xi$ will be justified again in Section \ref{section: spectral analysis}, in which we will recall \eqref{eqn: definition G_xi}. The operator $G_\xi$ is self-adjoint with compact resolvent (since its form domain is $H_0^1(\Omega_x)$ by Hardy inequalities \eqref{eqn: hardy inequality}). For any $\xi > 0$, we denote its eigenvalues by $\lambda_{\xi,k}$, $k \geq 0$, and the associated eigenfunctions by $u_{\xi,k}$. The eigenfunctions $u_{\xi,k}$ form an Hilbert basis of $L^2(\Omega_x)$.  We are therefore looking at an infinite family of one-dimensional systems. We shall write $G_n, \lambda_{n,k}$ and $u_{n,k}$ for $G_{\xi_n}$.\\

Recall that by Bessel-Parseval equality we have, for any $\omega_x \subset \Omega_x$, for any $f$ solution of system \eqref{stm : general system}, and decomposed as in \eqref{eqn: fourier decomposition},
\begin{align}
    \int_{\omega_x} \int_{\Omega_y} |f(t,x,y)|^2 \ dx \ dy = \sum_{n=1}^\infty \int_{\omega_x} |f_n(t,x)|^2 \ dx, \quad \text{for almost every } t \in (0,T).
\end{align}
Thus, the following holds.\\

System \eqref{stm : general system} is observable in time $T > 0$ from $\omega = \omega_x \times \Omega_y$ if and only if the Fourier systems \eqref{stm : adjoint system generalized fourier} are uniformly observable with respect to $n \geq 1$ in time $T > 0$ from $\omega_x$, that is, if and only if there exists $C > 0$ such that for every $n \geq 1$, for every $f_{0,n} \in L^2(0,L)$, the solution $f_n$ of system \eqref{stm : adjoint system generalized fourier} satisfies  
\begin{align}\label{eqn: observability inequality uniform fourier}
\int_{\Omega_x} |f_n(T,x)|^2 \ dx \leq C \int_0^T \int_{\omega_x} |f_n(t,x)|^2 \ dx \ dt.
\end{align}

\section{Sketch of proofs}\label{section: sketch of proofs}

Let us sketch here the proofs of our results, based on the previous Fourier decompositions. We start with the upper bound on the minimal time $T(\omega)$ defined in \eqref{eqn: def T(omega)}, and then discuss how to obtain lower bounds. 

\subsection{Upper bound on the minimal time of null-controllability: Carleman estimates}\label{section : sketch of proof carleman}

In this subsection, we discuss the proof of the upper bound in Theorem \ref{thm: main theorem control classical grushin rectangle vertical}. We thus focus on system \eqref{stm : adjoint system classical}, and its Fourier components given by systems \eqref{stm : adjoint system generalized fourier} with $q(x) = |x|$. \\

Consider an open subset $\omega_x$ satisfying the geometric assumption of Theorem \ref{thm: main theorem control classical grushin rectangle vertical} for some open smooth subdomain $\mathcal{O}$. Our proof is decomposed in two steps: 
\begin{enumerate}[label = (\roman*), ref = (\roman*)]
    \item \label{step i} Setting $\mathcal{O}_\epsilon := \{ x \in \Omega_x, \operatorname{dist}(x,\mathcal{O}) < \epsilon \}$, with $\epsilon$ sufficiently small such that $\partial\mathcal{O}_\epsilon \subset \omega_x$, we prove uniform boundary observability in any time 
    \begin{align}\label{eqn: upper bound time boundary obs}
        T \geq T(\epsilon) := \frac{\sup \{ |x|^2, x \in \partial \mathcal{O}_\epsilon \}}{4(1+\nu)},
    \end{align}
    from $\partial\mathcal{O}_\epsilon$, for systems \eqref{stm : adjoint system generalized fourier} posed on $\mathcal{O}_\epsilon$, \textit{i.e.} there exists $C > 0$ such that for every $n \geq 1$, any solution $f_n$ of system \eqref{stm : adjoint system generalized fourier} posed on $\mathcal{O}_\epsilon$ satisfies 
    \begin{align}\label{eqn: boundary observability inequality uniform fourier}
    \int_{\Omega_x} |f_n(T,x)|^2 \ dx \leq C \int_0^T \int_{\partial\mathcal{O}_\epsilon} |\partial_{\nu_x}f_n(t,x)|^2 \ ds(x) \ dt,
    \end{align}
    where $\partial_{\nu_x}$ stands for the outward normal derivative to the boundary.
    \item \label{step ii} This implies that system \eqref{stm : adjoint system classical} posed on $\mathcal{O}_\epsilon \times \Omega_y$ is observable in any time $T \geq T(\epsilon)$ from $\partial \mathcal{O}_\epsilon \times \Omega_y$ by the discussion of Section \ref{section: fourier decomposition}. Equivalently, system 
    \begin{align}\label{eqn: boundary control system posed on O_epsilon}
    \left\{ \begin{array}{lcll}
     \partial_t f - \Delta_xf - |x|^{2\gamma}\Delta_y f + \displaystyle\frac{\nu^2-\mathsf{H}}{|x|^2}f & = & 0, & (t,x,y) \in (0,T) \times \mathcal{O}_\epsilon \times \Omega_y,\\[6pt]
     f(t,x,y) & = & \mathbf{1}_{\partial \mathcal{O}_\epsilon \times\Omega_y}u(t,x,y), & (t,x,y) \in (0,T) \times \partial (\mathcal{O}_\epsilon \times\Omega_y), \\[6pt]
     f(0,x,y) & = & f_0(x,y), & (x,y) \in \mathcal{O}_\epsilon \times \Omega_y,
    \end{array} \right.
    \end{align}
    is null-controllable in any time $T \geq T(\epsilon)$. Classical extension arguments by cutoffs in the elliptic region then infer that system 
    \begin{align}\label{eqn: internal control system from O_epsilon posed on Omega} 
    \left\{ \begin{array}{lcll}
     \partial_t f - \Delta_xf - |x|^{2\gamma}\Delta_y f + \displaystyle\frac{\nu^2-\mathsf{H}}{|x|^2}f & = & \mathbf{1}_{(\mathcal{O}_\epsilon \setminus \mathcal{O})\times \Omega_y} u(t,x,y), & (t,x,y) \in (0,T) \times \Omega,\\[6pt]
     f(t,x,y) & = & 0, & (t,x,y) \in (0,T) \times \partial \Omega, \\[6pt]
     f(0,x,y) & = & f_0(x,y), & (x,y) \in \Omega,
    \end{array} \right.
    \end{align}
    is null-controllable in any time $T \geq T(\epsilon)$, with controls supported on $(\mathcal{O}_\epsilon \setminus \mathcal{O})\times \Omega_y \subset \omega$. As it is equivalent to internal observability from $(\mathcal{O}_\epsilon \setminus \mathcal{O} ) \times \Omega_y$, this yields the upper bound of Theorem \ref{thm: main theorem control classical grushin rectangle vertical} since 
    \begin{align*}
        \mathcal{O}_\epsilon \setminus \mathcal{O} \subset \omega_x \Rightarrow T(\omega) \leq T((\mathcal{O}_\epsilon \setminus \mathcal{O}) \times \Omega_y),
    \end{align*}
    and we can take $\epsilon > 0$ arbitrary small in step \ref{step i}.
\end{enumerate}

Step \ref{step ii} above is classical so we do not detail it. Our work is to focus on step \ref{step i}. Since systems \eqref{stm : adjoint system generalized fourier} are already known to be observable in any time $T > 0$, both from interior subsets or from the boundary, we only have to study the uniformity of the cost of observability in \eqref{eqn: boundary observability inequality uniform fourier} for large $n$. This is done by combining a cost of small time boundary observability of systems \eqref{stm : adjoint system generalized fourier} with the dissipation speed of their solutions, as follows. \\

We first obtain that any solution $f_n$ of system \eqref{stm : adjoint system generalized fourier}, posed on $\mathcal{O}_\epsilon$, satisfies the decay rate  
\begin{align}\label{eqn: dissipation sketch of proof}
    \|f_n(T)\|_{L^2(\mathcal{O}_\epsilon)}^2 \leq e^{-4\xi_n(1+\nu)(T-T_0)} \|f_n(T_0)\|_{L^2(\mathcal{O}_\epsilon)}^2, \quad 0 \leq T_0 < T.
\end{align}
This dissipation speed is given by an estimation of the first eigenvalues of $G_n$, whose study is provided in Section \ref{section: spectral analysis}, and more precisely, in Propositions \ref{prop: localization eigenvalues classical gamma = 1 d=1} for $d_x = 1$, and \ref{prop: first eigenvalue in the ball} for $d_x \geq 3$.\\

Then, we prove the existence of a constant $d(\epsilon) > 0$ such that, for any $T_0 > 0$, there exists $n_0 \in \mathbb{N}$ such that for any $n \geq n_0$ we have 
\begin{align}\label{eqn: cost of boundary observability sketch of proof}
    \|f_n(T_0)\|_{L^2(\mathcal{O}_\epsilon)}^2 \leq C(T_0)e^{d(\epsilon)\xi_n} \int_0^T \int_{\partial\mathcal{O}_\epsilon} |\partial_{\nu_x}f_n(t,x)|^2 \ ds(x) \ dt. 
\end{align}
Combining \eqref{eqn: dissipation sketch of proof} and \eqref{eqn: cost of boundary observability sketch of proof} shows that uniform observability holds for any $n$ sufficiently large as long as 
\begin{align}
    T \geq \frac{d(\epsilon)}{4(1+\nu)} + T_0,
\end{align}
with $T_0$ arbitrary small. The upper bound of Theorem \ref{thm: main theorem control classical grushin rectangle vertical} following steps \ref{step i} and \ref{step ii} is proved in Section \ref{section: proof positive result}.\\ 

We therefore need to obtain sharp estimates on the cost of boundary observability $d(\epsilon)$ in the asymptotic $n \rightarrow + \infty$. To achieve this, we rely on Carleman estimates inspired from \cite{beauchard2020minimal}, and given in Proposition \ref{prop: boundary carleman classical}, from which we estimate $d(\epsilon)$ in Proposition \ref{prop: cost of boundary observability}. In fact, the Carleman weight we use is exactly the one of \cite[eq. (2.14) and proof of Lemma 1.2]{beauchard2020minimal}, fully adapted to the presence of our singular coefficient (see Section \ref{section: carleman estimate classical} for details). \\

However, when $d_x = 1$ and $q(x)$ is not exactly $|x|$, we cannot simply use the Carleman weight of \cite[eq. (3.18) in Proposition 3.3]{beauchard2020minimal}. Indeed, it contains a term $B(L-x)$ whose role is to ensure that the space derivative of the weight keeps a constant non-vanishing sign in the computations of Carleman estimates. In the presence of our singular coefficient, this same term produces more singular contributions near the origin, which cannot be easily absorbed in the estimates. The case of a general coefficient $q(x)$ thus need further investigation. \\

We emphasize that for $d_x \geq 3$, the Carleman estimates given in \cite{ervedoza2008control} are not well-suited here. They give observability constants of the form $\exp(C(T)\xi_n^{4/3})$, while we are looking for something of the form $\exp(C(T)\xi_n)$. However, when considering system \eqref{stm : adjoint system classical} with $\gamma \in (0,1/2)$, they are sufficient to obtain observability in any time $T > 0$. For the case $\gamma \in (1/2,1)$, a refinement is needed. \\

We also propose in Appendix \ref{section: appendix The cost of small time observability from interior subsets}, Proposition \ref{prop: appendix cost of small time observability}, to estimate from above the cost of small time observability from interior subsets for the solutions of systems \eqref{stm : adjoint system generalized fourier} with $q(x) = |x|$, and in dimension $d_x = 1$. That is, to obtain an estimate of the form of \eqref{eqn: cost of observability sketch of proof} below, with $d$ well-estimated, which is similar to \eqref{eqn: cost of boundary observability sketch of proof} but for subsets distributed in the interior of $\Omega_x$. \\

Appendix \ref{section: appendix The cost of small time observability from interior subsets} thus provides an alternative proof of the upper bound in Theorem \ref{thm: main theorem control classical grushin rectangle vertical}, at least in the case $d_x = 1$, by following the same strategy as for the boundary observability case above, and from which one obtains the exact minimal time of observability \eqref{eqn: minimal time obs d=1} without ever relying on the equivalent control systems. \\

The computation are performed in dimension $d_x = 1$ for simplicity of the presentation, and because by decomposition in spherical harmonics in the setting of $\Omega_x = B(0,L)$ and $\omega_x = \{a < |x| < b \}$, the problem can be reduced anyway to the one-dimensional case (see \cite[Section 3.3]{vancostenoble2008null}). Since this is much harder and tedious to present when $\omega_x$ is not of the form $\omega_x = \{r < |x| < R\}$, we preferred to use in the main part of the paper the strategy of steps \ref{step i} and \ref{step ii} to prove the upper bound of Theorem \ref{thm: main theorem control classical grushin rectangle vertical}.\\

Moreover, Appendix \ref{section: appendix The cost of small time observability from interior subsets} is motivated by the discussion below. \\

We now discuss a strategy, used in the upcoming work \cite{DardeTrabut2025} for boundary observability, to relax when $\Omega_x = B(0,L)$ the geometric restriction on $\omega_x$ in Theorem \ref{thm: main theorem control classical grushin rectangle vertical}. In this setting, we assume $\omega_x = \{a \le |x| \le b\}$. Assume now that for any $n \ge n_0$ and any time $T_0 > 0$ (observe that $n_0$ is independent of $T_0$), we have a cost of small time observability 
\begin{align}\label{eqn: cost of observability sketch of proof}
    \|f_n(T_0)\|_{L^2(\Omega_x)}^2 \leq C(T_0)e^{d\xi_n} \int_0^{T_0} \int_{\omega_x} |f_n(t,x)|^2 \ dx \ dt. 
\end{align}
Using a decomposition of the solution in spherical harmonics (see \textit{e.g.} \cite{vancostenoble2008null}), the observation on $\omega_x$ becomes in the new spherical variables $(r,\theta) \in (0,L) \times \mathbb{S}^{d_x -1}$, an observation on $\{a \le r \le b\} \times \mathbb{S}^{d_x -1}$. Then, using Miller’s strategy \cite{miller2010direct}, which is a Lebeau–Robbiano type strategy from the dual viewpoint of observability, combined with spectral inequalities for spherical harmonics \cite{jerison1999nodal}, one can show from \eqref{eqn: cost of observability sketch of proof} that the cost of observing from $\{a \le r \le b\} \times \Gamma$, $\Gamma \subset \mathbb{S}^{d_x - 1}$, in any time $T_0 > 0$ has the same dependence on $\xi_n$ as in \eqref{eqn: cost of observability sketch of proof}. Then, following the proof of boundary observability given above, this will give an upper bound on the minimal time of observability from any open subset $\tilde{\omega}_x \subset B(0,L)$, by taking $b  = a + \epsilon$, and $\epsilon$ and $\Gamma$ arbitrary small so that $\{ a < |x| < a+ \epsilon\} \times \Gamma \subset \tilde{\omega}_x$. This upper bound will be precise as long as the cost obtained in \eqref{eqn: cost of observability sketch of proof} is precise.\\

This strategy is one of the reasons we conjectured that the lower bound in Theorem \ref{thm: main theorem control classical grushin rectangle vertical} is sharp for any open subset $\omega_x \subset \Omega_x$. From the control system viewpoint, this approach has been used for the Grushin equation in \cite{beauchard2014null,beauchard2014inverse}.\\

However, we stress again that to employ Miller's strategy \cite{miller2010direct}, \eqref{eqn: cost of observability sketch of proof} must hold for every $T_0 > 0$ and $n \geq n_0$, with $n_0$ independent of $T_0$. This independence of $n_0$ on $T_0$ is precisely what we miss. 

\subsection{Lower Bound on the minimal time of null-controllability: Agmon estimates}\label{section : sketch proof observability}

To obtain a lower bound on $T(\omega)$, that is prove the non-observability part of Theorem \ref{thm: main theorem control classical grushin rectangle vertical}, and Theorems \ref{thm: main theorem control classical grushin rectangle vertical 2} and \ref{thm: main theorem control generalized grushin rectangle vertical}, we  disprove the existence of a constant $C > 0$ such that the observability inequality \eqref{eqn: observability inequality uniform fourier} holds uniformly in $n \geq 1$. This is done in a classical manner by testing the inequality \eqref{eqn: observability inequality uniform fourier} against the solution associated to an eigenfunction of $G_n$, and taking $n \rightarrow + \infty$. We therefore need, at least for $n$ large, to extract a sequence of eigenvalues $(\lambda_{n,k_n})_n$ of $G_n$ introduced in \eqref{eqn: definition G_xi}, and obtain an upper bound on the norm of the eigenfunction $u_{n,k_n}$ in the observation set $\omega_x$, which is done by means of Agmon estimates.\\ 

The sequence of eigenvalues will be obtained in Section \ref{section: spectral analysis}. More precisely, for the classical operators ($q(x) = |x|^{\gamma})$, they will correspond to the first eigenvalues for each $n$ sufficiently large. We first perform the spectral analysis in one dimension in Section \ref{section: spectral analysis classical gamma = 1} when $\gamma = 1$, and Section \ref{section: spectral analysis classical gamma > 1} when $\gamma >1$. We then discuss in Section \ref{section: spectral analysis classical d > 3} how this extends in higher dimensions under assumption \ref{assumption: d > 3}. In Section \ref{section: spectral analysis generalized}, by a perturbative argument, we study the generalized operator. The Agmon estimate is given in Section \ref{section : agmon estimates}. \\

Since all the non-observability results are proved in a similar manner (only the knowledge of the sequence of eigenvalues changes), we prove them all in a single Section \ref{section: proofs negative results}.

\section{Boundary observability via Carleman estimates}\label{section: boundary observability via carleman estimates}

In this section, we prove some global Carleman estimates for our operators. These Carleman estimates follow from those obtained in \cite[Section 2]{beauchard2020minimal}. Prior to that, we propose a computation of the heat kernel of the operators $G_\xi$ on the whole half-line. This is not necessary for our proofs, but it may have interest on its own. Indeed, the Carleman estimates obtained in \cite{beauchard2020minimal} are inspired from the heat kernels of the Fourier components of the classical non-singular Grushin operators on the whole real line (see \cite[Remark 2.4]{beauchard2020minimal}). Our computations of the heat kernels of the singular operators show that both kernels are very much look-alike. This intuitively explains why we obtain similar results of observability for the singular equation as for the non-singular one. The main difference between the results comes from the effect of the singular potential on the speed of dissipation of the associated semigroup. But in terms of propagation of information, both dynamics behave roughly the same.   

\subsection{Heat kernel for the classical operator on the half-line}\label{section: heat kernel computation}

Let us first introduce our operator on the half-line. In this case, $\mathsf{H} = 1/4$, and we fix $\nu > 0$. For $\xi > 0$, we define
\begin{align}\label{eqn: definition grushin singulier classical on R^+}
    \begin{array}{crl}
    D(\mathcal{G}_\xi) &:=& \left\{ f \in H_0^1(0,+\infty), \quad \mathcal{G}_\xi f \in L^2(0,+\infty) \right\}, \\[8pt]
    \mathcal{G}_\xi  &=& -\partial_x^2 + \xi^2 x^2 + \displaystyle\frac{\nu^2 - 1/4}{x^2}.     
    \end{array}
\end{align}
This operator is self-adjoint with compact resolvent. Compactness of the resolvent follows from Hardy inequality, combined with \cite[Theorem 3.1]{berezin2012schrodinger}. Following \cite[Section 3.2]{romankummer2025} (where the operator is actually $\mathcal{G}_1$, but we can perform the change of variable $r = \sqrt{\xi}x$) we have the below result.  

\begin{proposition}\label{prop: eigv and eigenf of Gn on R+}
    The $L^2(0,+\infty)$-normalized eigenfunctions of $\mathcal{G}_\xi $ are  
    \begin{align}
        \Phi_{\xi,k}(x) = \sqrt{\frac{2\xi^{1+\nu}}{\Gamma(1+\nu)}}\sqrt{\frac{k!}{(1+\nu)_k}}e^{-\xi x^2/2}x^{1/2 + \nu}L_k^{(\nu)}(\xi x^2),
    \end{align}
    where $(a)_j = a(a+1)...(a+j-1)$ is the Pochammer symbol (or rising factorial), and $L_k^{(\nu)}$ are the generalized Laguerre polynomials. The associated eigenvalues are given by 
    \begin{align}
        \xi \mu_k = \xi \left(4k + 2(1+\nu) \right), \quad k \geq 0,
    \end{align}
    where we set $\mu_k = 4k + 2(1+ \nu)$, which are the eigenvalues of $\mathcal{G}_1$.
\end{proposition}

From Proposition \ref{prop: eigv and eigenf of Gn on R+}, we are able to compute the heat kernel of $\mathcal{G}_\xi $. 

\begin{proposition}\label{prop: heat kernel}
    The Heat kernel associated to $\mathcal{G}_\xi $ is given by 
    \begin{align}
        k_\xi(t,x,x') = \frac{\xi \sqrt{xx'}}{\sinh(2\xi t)}I_\nu\left( \frac{\xi xx'}{\sinh(2\xi t)}\right)e^{-\frac{\xi}{2}(x^2+x'^2)\coth(2\xi t)},
    \end{align}
    where $I_\nu$ is the modified Bessel function of order $\nu$.
\end{proposition}

\begin{proof}
    Since the family $\{\Phi_{\xi,k}\}_k$ is complete and orthonormal in $L^2(0,+\infty)$, we use Mehler formula for the heat kernel. That is, we write it using its eigenfunction expansion
    \begin{align} \label{eqn: mehler formula}
        k_\xi(t,x,x') = \sum_{k\geq 0} e^{-\xi \mu_kt}\Phi_{\xi,k}(x)\Phi_{\xi,k}(x'),
    \end{align}
    and we compute \eqref{eqn: mehler formula} to prove the proposition. \\
    
    From Proposition \ref{prop: eigv and eigenf of Gn on R+}, using that $(1+\nu)_k \Gamma(1+\nu) = \Gamma(1+\nu + k)$, we have 
    \begin{align*}
        k_\xi(t,x,x') = 2\xi^{1+\nu}(xx')^{\frac{1}{2}+\nu}e^{-2\xi (1+\nu)t}e^{-\frac{\xi}{2}(x^2 + x'^2)} S_\xi (t,x,x'), 
    \end{align*}
    where we set
    \begin{align}
        S_\xi (t,x,x') = \sum_{k\geq 0} \frac{k!}{\Gamma(1+\nu + k)} e^{-4\xi kt}L_k^{(\nu)}(\xi x^2)L_k^{(\nu)}(\xi x'^2).
    \end{align}
    Recall the following summation formula from \cite[Eq. (20) page 189]{bateman1953higher},
    \begin{align}
        \sum_{k \geq 0} \frac{k!}{\Gamma(k+\nu +1)}L_k^{(\nu)}(x)L_k^{(\nu)}(y)z^k = \frac{1}{(1-z)(xyz)^{\nu/2}}e^{-z\frac{x+y}{1-z}}I_\nu\left( 2\frac{\sqrt{xyz}}{1-z}\right), \quad |z| < 1,
    \end{align}
    where $I_\nu$ is the modified Bessel function of order $\nu$. Using this formula with $z = e^{-4\xi t}$, replacing $x,y$ respectively by $\xi x^2,\xi x'^2$, we get 
    \begin{align}
        S_\xi (t,x,x') = \frac{1}{(1-e^{-4\xi t})(\xi xx')^\nu e^{-2\nu \xi t}}e^{\displaystyle -\xi (x^2+x'^2)\frac{ e^{-4\xi t}}{\displaystyle 1-e^{-4\xi t}}}I_\nu\left( 2\xi xx'\frac{e^{-2\xi t}}{1-e^{-4\xi t}}\right).
    \end{align}
    Thus, 
    \begin{align}
        k_\xi(t,x,x') = 2\xi(xx')^{\frac{1}{2}}\frac{e^{-2\xi t}}{1-e^{-4\xi t}}e^{\displaystyle -\xi (x^2+x'^2) \left(\frac{1}{2} +  \frac{e^{-4\xi t}}{1-e^{-4\xi t}}\right)}I_\nu\left( 2\xi xx'\frac{e^{-2\xi t}}{1-e^{-4\xi t}}\right).
    \end{align}
    Now, the statement follows from the identities 
    \begin{align*}
    \begin{array}{lcl}
        \displaystyle \frac{e^{-2\xi t}}{1-e^{-4\xi t}} &=& \displaystyle \frac{1}{2\sinh(2\xi t)}, \\[12pt]
       \displaystyle \frac{1}{2} + \frac{e^{-4\xi t}}{1-e^{-4\xi t}} &=& \frac{1}{2}\coth(2\xi t).
    \end{array}
    \end{align*}
\end{proof}

\subsection{A global Carleman estimate for the classical operator}\label{section: carleman estimate classical}

In this section, we obtain global parabolic Carleman estimates for systems \eqref{stm : adjoint system generalized fourier} with $q(x) = |x|$, using a weight inspired from the heat kernel that we computed in Section \ref{section: heat kernel computation}. We choose as a Carleman weight the function 
\begin{align}\label{eqn: carleman weight grushin classic}
    \varphi_\xi (t,x) = \frac{\xi}{2}(L^2 - |x|^2)\coth(2\xi t).
\end{align}
One may observe that this is exactly the Carleman weight of \cite[Section 2.1]{beauchard2020minimal}, and we obtain, modulo the singular term and the presence of a source term, the same global Carleman estimate. The proof follows the classical one for Carleman estimates, with ideas from the proof of \cite[Proposition 2.1]{beauchard2020minimal}. \\

Let us make another observation as to why it allows us to treat the singular term. The space-dependent part of $\varphi_\xi$, namely $(L^2 - |x|^2)$, is actually the space-dependent part of the Carleman weights used to deal with the heat equation in the presence of a singular potential (see \textit{e.g.} \cite{vancostenoble2008null, ervedoza2008control} among others). Indeed, the term $|x|^2$ is well-suited, in the integration by parts computations, to deal with the singular potential $1/|x|^2$, as follows. The usual Carleman machinery requires obtaining energy estimates for the conjugated operator $P_{\varphi_\xi} g := e^{-\varphi_\xi} P_\xi \left( e^{\varphi_\xi} g \right)$, by decomposing it and integrating by parts. With this weight, this produces a term $x \cdot \nabla_x g$ that acts on $(\nu^2 - \mathsf{H}) g/|x|^2$, and the integration by parts eventually reproduces the same singular term, which is then handled using Hardy's inequality to obtain observability. \\

We replace $\xi_n$ in system \eqref{stm : adjoint system generalized fourier} by a generic parameter $\xi > 0$, and we handle a source term. Namely, letting $T > 0$, we consider the following system
\begin{align}\label{stm : classical grushin nonhomog generic parameter}
    \left\{
    \begin{array}{lcll}
        \displaystyle \partial_t f - \Delta_x f + \xi^2|x|^2 + \frac{\nu^2 - \mathsf{H}}{|x|^2}f  &=& f_\xi ,& (t,x) \in (0,T) \times \Omega_x,  \\[6pt]
         f(t,x) &=& 0,& (t,x) \in (0,T) \times \partial\Omega_x, \\[6pt]
         f(0,x) &=& f_0(x), & x \in \Omega_x, 
    \end{array} \right.
\end{align}
with $f_0 \in L^2(\Omega_x)$, $f_\xi  \in L^2((0,T) \times \Omega_x)$, and $\nu > 0$. \\

\begin{proposition}[Boundary Carleman estimate]\label{prop: boundary carleman classical}
    Let $\varphi_\xi $ be defined by \eqref{eqn: carleman weight grushin classic}. Set $ L := \sup \{|x|, x \in \Omega_x \}$. For any solution $f$ of system \eqref{stm : classical grushin nonhomog generic parameter}, the function $g = fe^{-\varphi_\xi }$ satisfies 
    \begin{align}
        \int_{\Omega_x} |\nabla_x g(T,x)|^2 - &\frac{\xi^2L^2}{\sinh(2\xi t)^2}|g(T,x)|^2 + \frac{\nu^2 - \mathsf{H}}{|x|^2}|g(T,x)|^2 \ dx \notag \\
        &\leq \int_0^T \int_{\Omega_x} |f_\xi |^2e^{-2\varphi_\xi } \ dx \ dt + \xi L\int_0^T \frac{\sinh(4\xi t)}{\sinh(2\xi T)^2} \int_{\Gamma_+} |\nabla_x g \cdot \eta|^2 \ dS \ dt,
    \end{align}
    where $\Gamma_+ = \{x \in \partial \Omega_x, x \cdot \eta > 0\}$, and $\eta$ is the normal outward pointing unit vector at the boundary.
\end{proposition}

Let us make some preliminary computations. Setting 
\begin{align*}
    f(t,x) = g(t,x)e^{\varphi_\xi (t,x)}, \quad (t,x) \in (0,T) \times \Omega_x,
\end{align*}
a few computations show that 
\begin{align}
    \partial_t f &= \partial_t ge^{\varphi_\xi (t,x)} - \frac{\xi^2(L^2 - |x|^2)}{\sinh(2\xi t)^2}g(t,x)e^{\varphi_\xi (t,x)}, \\
    \Delta_x f &=  \left( \Delta_x g(t,x) - 2\xi \coth(2\xi t) x \cdot \nabla_x g(t,x) - d_x \xi \coth(2\xi t)g(t,x) + \xi^2|x|^2\coth(2\xi t)^2g(t,x) \right) e^{\varphi_\xi (t,x)}.
\end{align}
Thus, using that 
\begin{align}\label{eqn: hyperbolic sin cot identity}
    \frac{1}{\sinh(2\xi t)^2} - \coth(2\xi t)^2 = -1, 
\end{align}
we get, writing $P_\xi = \partial_t + G_\xi$, that
\begin{align}\label{eqn: def P_phi,xi}
    P_{\varphi_\xi } g := e^{-\varphi_\xi } P_\xi \left( ge^{\varphi_\xi } \right) = \partial_t g - \Delta_x g + \xi \coth(2\xi t) \left( 2x \cdot \nabla_x g + d_xg \right) - \frac{\xi^2L^2}{\sinh(2\xi t)^2}g + \frac{\nu^2 - \mathsf{H}}{|x|^2}g.
\end{align}

Hence, if $f$ is solution of system \eqref{stm : classical grushin nonhomog generic parameter}, we get that $g$ must solve
\begin{align}\label{stm: carleman conjugated solution}
    \left\{
    \begin{array}{lcll}
       \displaystyle \partial_t g - \Delta_x g + \xi \coth(2\xi t) \left( 2x \cdot \nabla_x g + d_xg \right) - \frac{\xi^2L^2}{\sinh(2\xi t)^2}g + \frac{\nu^2 - \mathsf{H}}{|x|^2}g  &=& e^{-\varphi_\xi }f_\xi ,& (0,T) \times \Omega_x,  \\[6pt]
         g(t,x) &=& 0,& (0,T) \times \partial\Omega_x, \\[6pt]
         \underset{t \rightarrow 0^+}{\lim} \|g(t)\|_{L^2(\Omega_x)} &=& 0.
    \end{array} \right.
\end{align}

\begin{proof}[Proof of Proposition \ref{prop: boundary carleman classical}]
    We will follow part of the proof of \cite[Proposition 2.1]{beauchard2020minimal}, and use their notations. As shown by system \eqref{stm: carleman conjugated solution}, $g = fe^{-\varphi_\xi }$ satisfies 
    \begin{align*}
        P_{\varphi_\xi }g = f_\xi e^{-\varphi_\xi}, 
    \end{align*}
    where $P_{\varphi_\xi}$ is defined in \eqref{eqn: def P_phi,xi}. We set $\theta(t) = \xi \coth(2\xi t)$, so that 
    \begin{align*}
        P_{\varphi_\xi }g =  \partial_t g - \Delta_x g + \theta(t) \left( 2x \cdot \nabla_x g + d_x g \right) + \frac{L^2}{2}\theta'(t)g + \frac{\nu^2 - \mathsf{H}}{|x|^2}g.
    \end{align*}
    We set 
    \begin{align}\left\{
    \begin{array}{lll}
        \mathsf{P}_1 g &=& \displaystyle -\Delta_x g + \frac{L^2}{2}\theta'(t) g + \frac{\nu^2 - \mathsf{H}}{|x|^2}g, \\[10pt]
        \mathsf{P}_2 g &=& \partial_t g + 2\theta(t)x \cdot \nabla_x g + d_x \theta(t) g.
    \end{array}\right. 
    \end{align}
    We have
    \begin{align*}
        \int_{\Omega_x} \left| P_{\varphi_\xi }g\right|^2 = \int_{\Omega_x} \left| f_\xi \right|^2e^{-2\varphi_\xi } \geq 2 \int_{\Omega_x} \mathsf{P}_1g \mathsf{P}_2g.
    \end{align*}
    We want to compute the cross-product on the right-hand side,
    \begin{align*}
        \int_{\Omega_x} \mathsf{P}_1g \mathsf{P}_2 g = \sum_{i=1}^3 \sum_{j=1}^3 I_{ij},
    \end{align*}
    where $I_{ij}$ is the integral on $\Omega_x$ of the product between the $i$-th term of $\mathsf{P}_1g$ and the $j$-th term of $\mathsf{P}_2g$. We have, using the fact that $g \in H_0^1(\Omega_x)$, and denoting by $\eta$ the outward pointing normal vector to $\partial \Omega_x$,
    \begin{align*}
        \begin{array}{lll}
          I_{11} &=& \displaystyle \frac{1}{2} \int_{\Omega_x} \partial_t (|\nabla_xg|^2),  \\[8pt]
          I_{12} &=& \displaystyle -2\theta(t) \left(\frac{1}{2} \int_{\partial \Omega_x} (x \cdot \eta ) |\nabla_x g \cdot \eta|^2  + \left( \frac{d_x}{2} - 1 \right)\int_{\Omega_x} |\nabla_x g|^2\right),\\[8pt]
          I_{13} &=& \displaystyle d_x\theta(t)\int_{\Omega_x} |\nabla_x g|^2, \\[8pt]
          I_{21} &=& \displaystyle \frac{L^2}{4} \theta'(t) \int_{\Omega_x} \partial_t (|g|^2), \\[8pt]
          I_{22} &=& \displaystyle -d_x\frac{L^2}{2}\theta'(t)\theta(t) \int_{\Omega_x} |g|^2, \\[8pt]
          I_{23} &=& \displaystyle d_x\frac{L^2}{2}\theta'(t)\theta(t) \int_{\Omega_x} |g|^2, \\[8pt]
          I_{31} &=& \displaystyle \frac{1}{2} \int_{\Omega_x} \frac{\nu^2 - \mathsf{H}}{|x|^2}\partial_t(|g|^2), \\[8pt]
          I_{32} &=& \displaystyle (2 - d_x ) \theta(t) \int_{\Omega_x} \frac{\nu^2 - \mathsf{H}}{|x|^2} |g|^2, \\[8pt]
          I_{33} &=& \displaystyle d_x\theta(t) \int_{\Omega_x} \frac{\nu^2 - \mathsf{H}}{|x|^2} |g|^2.
        \end{array}
    \end{align*}
    
    Let us set 
    \begin{align*}
        D(t) = \int_{\Omega_x} |\nabla_x g|^2 + \frac{L^2}{2}\theta'(t) |g|^2 + \frac{\nu^2 - \mathsf{H}}{|x|^2} |g|^2,
    \end{align*}
    and observe that $\theta''(t) = -4\theta'(t)\theta(t)$. We now sum all the $I_{ij}$. \\
    
    First, we have
    \begin{align*}
        \frac{1}{2}D'(t) &= \frac{1}{2}\frac{d}{dt}\left[ \int_{\Omega_x} |\nabla_x g|^2 + \frac{L^2}{2}\theta'(t) |g|^2 + \frac{\nu^2 - \mathsf{H}}{|x|^2} |g|^2 \right] \\
        &= \frac{1}{2} \int_{\Omega_x} \partial_t(|\nabla_x g|^2) + \frac{L^2}{2}\theta''(t) |g|^2 + \frac{L^2}{2} \theta'(t) \partial_t(|g|^2) + \frac{\nu^2 - \mathsf{H}}{|x|^2} \partial_t (|g|^2) \\
        &= \int_{\Omega_x} \frac{1}{2} \partial_t(|\nabla_x g|^2) - L^2\theta'(t)\theta(t) |g|^2 + \frac{L^2}{4} \theta'(t) \partial_t(|g|^2) + \frac{\nu^2 - \mathsf{H}}{2|x|^2} \partial_t (|g|^2) \\
        &= I_{11} + I_{21} + I_{31} - \int_{\Omega_x} L^2\theta'(t)\theta(t) |g|^2 
    \end{align*}
    Second, remark that $I_{22} + I_{23} = 0$. And finally, 
    \begin{align*}
        I_{12} + I_{13} + I_{32} + I_{33} &= \displaystyle -2\theta(t) \left(\frac{1}{2} \int_{\partial \Omega_x} (x \cdot \eta ) |\nabla_x g \cdot \eta|^2  + \left( \frac{d_x}{2} - 1 \right)\int_{\Omega_x} |\nabla_x g|^2\right) \\
        &+ \displaystyle d_x\theta(t)\int_{\Omega_x} |\nabla_x g|^2 + \displaystyle (2 - d_x ) \theta(t) \int_{\Omega_x} \frac{\nu^2 - \mathsf{H}}{|x|^2} |g|^2 +  \displaystyle d_x\theta(t) \int_{\Omega_x} \frac{\nu^2 - \mathsf{H}}{|x|^2} |g|^2 \\
        &=  \displaystyle  -\theta(t)\int_{\partial \Omega_x} (x \cdot \eta ) |\nabla_x g \cdot \eta|^2  - \theta(t) \left( d_x - 2 \right)\int_{\Omega_x} |\nabla_x g|^2  \\
        &+ \displaystyle d_x\theta(t)\int_{\Omega_x} |\nabla_x g|^2 + \displaystyle (2 - d_x ) \theta(t) \int_{\Omega_x} \frac{\nu^2 - \mathsf{H}}{|x|^2} |g|^2 +  \displaystyle d_x\theta(t) \int_{\Omega_x} \frac{\nu^2 - \mathsf{H}}{|x|^2} |g|^2 \\
        &= \displaystyle  -\theta(t)\int_{\partial \Omega_x} (x \cdot \eta ) |\nabla_x g \cdot \eta|^2  + 2\theta(t) \int_{\Omega_x} |\nabla_x g|^2 + \displaystyle 2 \theta(t) \int_{\Omega_x} \frac{\nu^2 - \mathsf{H}}{|x|^2} |g|^2 \\
        &= \displaystyle  -\theta(t)\int_{\partial \Omega_x} (x \cdot \eta ) |\nabla_x g \cdot \eta|^2  + 2\theta(t) D(t) - \int_{\Omega_x} L^2 \theta'(t) \theta(t) |g|^2.
    \end{align*}
    
    Summing everything gives
    \begin{align*}
        \sum_{i=1}^3 \sum_{j=1}^3 I_{ij} = \frac{1}{2}D'(t) + \int_{\Omega_x} L^2\theta'(t)\theta(t) |g|^2  -\theta(t)\int_{\partial \Omega_x} (x \cdot \eta ) |\nabla_x g \cdot \eta|^2  + 2\theta(t) D(t) - \int_{\Omega_x} L^2 \theta'(t) \theta(t) |g|^2,
    \end{align*}
    so that
    \begin{align}
       \int_{\Omega_x} \mathsf{P}_1g \mathsf{P}_2 g = \frac{1}{2} D'(t) + 2\theta(t)D(t) - \theta(t) \int_{\partial \Omega_x} (x \cdot \eta ) |\nabla_x g \cdot \eta|^2.
    \end{align}
    It follows that 
    \begin{align}
        \int_{\Omega_x} \left| f_\xi \right|^2e^{-2\varphi_\xi }  + 2\theta(t) \int_{\Gamma_+} (x \cdot \eta ) |\nabla_x g \cdot \eta|^2 \geq D'(t) + 4\theta(t)D(t). 
    \end{align}
    We can therefore conclude as in the proof of \cite[Proposition 2.1]{beauchard2020minimal} by integrating on $(0,T)$ the above inequality to obtain the Carleman estimate. 
\end{proof}

\subsection{Boundary observability}\label{section: boundary observability}

We now obtain a cost of boundary observability of systems \eqref{stm : adjoint system generalized fourier} with $q(x) = |x|$ from Proposition \ref{prop: boundary carleman classical}. Recall that we replaced $\xi_n$ by a generic parameter $\xi > 0$ in system \eqref{stm : adjoint system generalized fourier}, so that we look at solutions of system \eqref{stm : classical grushin nonhomog generic parameter} with $f_\xi = 0$.

\begin{proposition}\label{prop: cost of boundary observability}
    For every $T > 0$, there exists $\xi_0 > 0$ and a constant $C > 0$, such that for every $\xi \geq \xi_0$, the solution $f$ of system \eqref{stm : classical grushin nonhomog generic parameter} with $f_\xi = 0$ satisfies
    \begin{align}
        \int_{\Omega_x} |f(T,x)|^2\ dx \leq C\xi e^{L^2/2T}e^{\xi L^2} \int_0^T \int_{\partial \Omega_x} |\partial_{\nu_x} f(t,x)|^2 \ dS \ dt,
    \end{align}
    where $L = \sup \{|x|, \ x \in \partial\Omega_x\}$.
\end{proposition}

\begin{proof}
    Let $T > 0$, and $f$ a solution of system \eqref{stm : classical grushin nonhomog generic parameter} with $f_\xi = 0$, and with $\xi > 0$ sufficiently large. We apply the Carleman estimate of Proposition \ref{prop: boundary carleman classical} which gives, for $g = fe^{-\varphi_\xi}$ with $\varphi_\xi$ defined in \eqref{eqn: carleman weight grushin classic},
    \begin{align}\label{eqn: carleman proof lemma}
        \int_{\Omega_x} |\nabla_x g(T,x)|^2 - &\frac{\xi^2L^2}{\sinh(2\xi t)^2}|g(T,x)|^2 + \frac{\nu^2 - \mathsf{H}}{|x|^2}|g(T,x)|^2 \ dx \notag \\
        &\leq \xi L\int_0^T \frac{\sinh(4\xi t)}{\sinh(2\xi T)^2} \int_{\Gamma_+} |\nabla_x g \cdot \eta|^2 \ dS \ dt.
    \end{align}
    We first bound from below the left-hand term of \eqref{eqn: carleman proof lemma}. If $\nu^2 - \mathsf{H} \geq 0$, we have, using Poincaré's inequality, 
    \begin{align*}
        &\int_{\Omega_x} |\nabla_x g_\xi (T,x)|^2 - \frac{\xi^2L^2}{\sinh(2\xi T)^2}|g_\xi (T,x)|^2 + \frac{\nu^2 - \mathsf{H}}{|x|^2}|g_\xi (T,x)|^2 \ dx \\
        &\geq \int_{\Omega_x} |\nabla_x g_\xi (T,x)|^2 - \frac{\xi^2L^2}{\sinh(2\xi T)^2}|g_\xi (T,x)|^2 \ dx \\
        &\geq \int_{\Omega_x} C(\Omega_x)|g_\xi (T,x)|^2 - \frac{\xi^2L^2}{\sinh(2\xi T)^2}|g_\xi (T,x)|^2 \ dx.
    \end{align*}
    Now, using the fact that for any $T > 0$, the function $\xi \mapsto \xi^2L^2/\sinh(2\xi T)^2$ is decreasing and converges to $0$ as $\xi \rightarrow + \infty$, we get that there exists some $C' > 0$, such that for every $\xi > 0$ large enough depending on $T$ and $L$, 
    \begin{align}\label{eqn: obs via carl proof 1}
        &\int_{\Omega_x} |\nabla_x g_\xi (T,x)|^2 - \frac{\xi^2L^2}{\sinh(2\xi T)^2}|g_\xi (T,x)|^2 + \frac{\nu^2 - \mathsf{H}}{|x|^2}|g_\xi (T,x)|^2 \ dx \geq C'\int_{\Omega_x} |g_\xi (T,x)|^2\ dx.
    \end{align}
    On the other hand, if $\nu^2 - \mathsf{H} \leq 0$, we use Hardy inequality \eqref{eqn: hardy inequality}. We have 
    \begin{align}
        \int_{\Omega_x} \frac{\nu^2 - \mathsf{H}}{|x|^2}|g_\xi (T,x)|^2 \ dx \geq \left( \frac{\nu^2}{\mathsf{H}}- 1 \right) \int_{\Omega_x} |\nabla_x g_\xi (T,x)|^2 \ dx. 
    \end{align}
    Thus, 
    \begin{align*}
        &\int_{\Omega_x} |\nabla_x g_\xi (T,x)|^2 - \frac{\xi^2L^2}{\sinh(2\xi T)^2}|g_\xi (T,x)|^2 + \frac{\nu^2 - \mathsf{H}}{|x|^2}|g_\xi (T,x)|^2 \ dx \\
        &\geq \int_{\Omega_x}  \frac{\nu^2}{\mathsf{H}}|\nabla_x g_\xi (T,x)|^2 - \frac{\xi^2L^2}{\sinh(2\xi T)^2}|g_\xi (T,x)|^2 \ dx.
    \end{align*}
    Since $\nu > 0$, we can proceed as for the case $\nu^2 - \mathsf{H} \geq 0$. There exists some $C > 0$, such that for every $\xi > 0$ large enough, we have 
    \begin{align}\label{eqn: obs via carl proof 2}
        \int_{\Omega_x} |\nabla_x g_\xi (T,x)|^2 - \frac{\xi^2L^2}{\sinh(2\xi T)^2}|g_\xi (T,x)|^2 + \frac{\nu^2 - \mathsf{H}}{|x|^2}|g_\xi (T,x)|^2 \ dx\geq C \int_{\Omega_x} |g_\xi (T,x)|^2 \ dx.
    \end{align}
    We now bound from above the right-hand side of \eqref{eqn: carleman proof lemma}.
    Using the facts that 
    \begin{align*}
    \begin{array}{llll}
         \displaystyle \frac{\sinh(4\xi t)}{\sinh(2\xi T)^2} &\leq& 2 \coth(2\xi T), & \text{for every } t \in (0,T), \\[8pt]
         \coth(2\xi T)  &\leq& 2, & \text{for $\xi$ large enough,} 
    \end{array}        
    \end{align*}
    we deduce that 
    \begin{align}\label{eqn: obs via carl proof 3}
         \xi L\int_0^T \frac{\sinh(4\xi t)}{\sinh(2\xi T)^2} \int_{\Gamma_+} |\nabla_x g \cdot \eta|^2 \ dS \ dt \leq 4\xi L \int_0^T \int_{\partial \Omega_x} |\nabla_x g \cdot \eta|^2 \ dS \ dt.
    \end{align}
    Combining \eqref{eqn: carleman proof lemma} with \eqref{eqn: obs via carl proof 1}, \eqref{eqn: obs via carl proof 2}, \eqref{eqn: obs via carl proof 3}, replacing $g$ by its definition and using that $f$ vanishes on $\partial \Omega_x$, we obtain 
    \begin{align}
        \int_{\Omega_x} |f(T,x)|^2\ dx \leq C\xi e^{L^2\xi\coth(2\xi T)} \int_0^T \int_{\partial \Omega_x} |\partial_{\nu_x} f(t,x)|^2 \ dS \ dt.
    \end{align}
    This concludes the proof using the inequality $\coth(s) \leq 1 + 1/s$, $s > 0$.
\end{proof}

\section{Spectral analysis}\label{section: spectral analysis}

This section is devoted to the spectral analysis of the Fourier components of our singular operators, introduced in \eqref{eqn: definition G_xi}. We consider again in this section a generic parameter $\xi \in (0,+\infty)$ instead of $\xi_n$. Let us first reintroduce properly the operators with a general coefficient $q$. \\

Let $\nu \geq 0$ and $\xi > 0$. Recall   
\begin{align*}
    G_\xi = -\partial_x^2 + \xi^2q(x)^2 + \frac{\nu^2 - \mathsf{H}}{|x|^2},      
\end{align*}
defined on $C_c^\infty(\Omega_x)$. We consider its Friedrich extension, and we keep the notation $G_\xi$. That is, we first set $H_{0,\nu}^1(\Omega_x)$ to be the completion of $C_c^\infty(\Omega_x)$ with respect to the norm $\|\cdot\|_\nu$, where 
\begin{align}
    \|f\|_\nu := \left( \int_{\Omega_x} |\nabla f(x)|^2 + \left( \xi^2q(x)^2 + \frac{\nu^2 - \mathsf{H}}{x^2} \right)f(x)^2 \ dx\right)^{1/2}.
\end{align}
Thanks to Hardy's inequalities \eqref{eqn: hardy inequality}, when $\nu > 0$, $\|\cdot\|_\nu$  and the usual $H_0^1(\Omega_x)$-norm are equivalent. Thus, we have $H_{0,\nu}^1(\Omega_x) = H_0^1(\Omega_x)$. As shown in \cite[Section 5]{vazquez2000hardy}, in the critical case $\nu = 0$ we have the strict inclusion 
\begin{align*}
   H_0^1(\Omega_x) \subsetneq H_{0,\nu}^1(\Omega_x).
\end{align*}

We consider the operator
\begin{align}\label{eqn: definition G_xi bis}
    \begin{array}{ccl}
       D(G_\xi) &=& \left\{f \in H_{0,\nu}^1(\Omega_x), \quad G_\xi f \in L^2(\Omega_x) \right\},\\[6pt]
        G_\xi &=& \displaystyle -\Delta_x^2 + \xi^2 q(x)^2 + \frac{\nu^2 - \mathsf{H}}{x^2}.
    \end{array}
\end{align}
Observe that since we take $\nu > 0$, \eqref{eqn: definition G_xi bis} is \eqref{eqn: definition G_xi}, but we recall it for clarity. We do not specify in the notation $D(G_\xi)$ the dependence on $\nu$, nor on $\gamma$ and the dimension, as there shall not be any confusion. This operator is positive-definite, self-adjoint with compact resolvent, and from Sturm-Liouville theory, its eigenvalues are simple. Hence, its eigenvalues form an increasing sequence of real numbers, 
\begin{align}
  0 \leq \lambda_{\xi,0} < \lambda_{\xi,1} < ... < \lambda_{\xi,k} < ...
\end{align}

We first discuss Agmon estimates in Section \ref{section : agmon estimates}. Then, we treat the classical operators ($q(x) = |x|^{\gamma}$) when $d_x = 1$ in Section \ref{section: spectral analysis classical gamma = 1} when $\gamma = 1$, and Section \ref{section: spectral analysis classical gamma > 1} when $\gamma > 1$. Then, we analyze the classical operators when $d_x \geq 3$ in Section \ref{section: spectral analysis classical d > 3}. Finally, we discuss the generalized operator in Section \ref{section: spectral analysis generalized}. 

\subsection{Agmon estimates}\label{section : agmon estimates}

In this section, we give upper estimates on the $L^2$-norm in the observation set $\omega_x$ of some eigenfunctions of the operators introduced above in \eqref{eqn: definition G_xi bis}. This will be of use to obtain a lower bound on the minimal time of observability in Section \ref{section: proofs negative results}.\\

We consider here the generalized operator 
\begin{align}
    G_\xi = -\Delta_x^2  + \xi^2q(x)^2  + \frac{\nu^2 - \mathsf{H}}{|x|^2},
\end{align}
where $q \in C^0(\overline{\Omega_x})$ satisfies 
\begin{enumerate}[label = ({$\operatorname{H}'_q$}), ref = ($\operatorname{H}'_q$)]
    \item \label{assumption: weak assumption q}$q(0) = 0$, $q(x) \neq 0$ for every $x \neq 0$.
\end{enumerate}

Observe that assumption \ref{assumption: weak assumption q} is weaker than \ref{assumption: q}. We denote the eigenfunctions of $G_\xi$, normalized in $L^2$-norm, by $u_{\xi,k}$, associated to $\lambda_{\xi,k}$. For every $\xi > 0$, we extract one eigenvalue of $G_\xi$ that we denote $\lambda_\xi$, and the corresponding normalized eigenfunction is denoted by $u_\xi$. \\

Introduce the set 
\begin{align}\label{eqn: definition F delta}
    F_\delta := \left\{ x \in \Omega_x, \quad q(x)^2 \leq \delta \right\}.
\end{align}
By assumption \ref{assumption: weak assumption q}, there exists $\delta_0$, such that for every $0 < \delta < \delta_0$, the set $F_\delta$ is a connected neighborhood of $0$ in $\Omega_x$.

\begin{proposition}[Agmon estimate]\label{prop: agmon decay grushin general}
Under the setting described above, let $\omega_x \subset \Omega_x$ be such that $\overline{\omega_x} \cap \{0\} = \emptyset$. Assume that, as $\xi \rightarrow + \infty$, $ \lambda_\xi= o(\xi^2)$. Then, there exists $\delta_0 > 0$, such that for every $\delta \in (0,\delta_0)$, there exists $\xi_0 > 0$ and $C > 0$, such that for every $\xi \geq \xi_0$,
\begin{align}
    \int_{\omega_x}  u_\xi(x)^2 \ dx \leq C e^{-2\xi(1-\delta)d_{\operatorname{agm},\delta}(\omega_x,F_\delta)},
\end{align}
where $d_{\operatorname{agm},\delta}$ is defined by \eqref{eqn: definition Agmon distance} with $E = \delta$, and $F_\delta$ is defined by \eqref{eqn: definition F delta} and satisfies $F_\delta \cap \overline{\omega} = \emptyset$.
\end{proposition}

\begin{proof} Let $u_\xi  = \theta_\xi e^{-\xi \psi }$ for some sufficiently regular function $\psi $. We follow the proof of \cite[Proposition 4.1]{vanlaere2025grushinlike}, in higher dimensions. We first compute
\begin{align*}
    \Delta_x u_\xi   = \left( \Delta_x \theta_\xi   -2\xi  \nabla_x\theta_\xi  \cdot \nabla_x \psi  - \xi \theta_\xi \Delta_x \psi  + \xi^2 |\nabla_x \psi |^2 \right)e^{-\xi \psi }.
\end{align*}
Since $(G_\xi  - \lambda_\xi )u_\xi   = 0$, we obtain 
\begin{align}\label{eqn: nabla theta_n}
    -\Delta_x \theta_\xi   + 2\xi  \nabla_x\theta_\xi  \cdot \nabla_x \psi  + \left(\xi \Delta_x \psi  - \xi^2 |\nabla_x \psi |^2 + \xi^2 q(x)^2 + \frac{\nu^2 - \mathsf{H}}{|x|^2} - \lambda_\xi  \right)\theta_\xi  = 0.
\end{align}
An integration by parts shows that
\begin{align}\label{eqn: ipp proof agmon}
     \int_{\Omega_x} \left( \nabla_x\theta_\xi  \cdot \nabla_x \psi  \right) \theta_\xi  = - \frac{1}{2} \int_{\Omega_x} \theta_\xi ^2 \Delta_x \psi .
\end{align}

Therefore, multiplying \eqref{eqn: nabla theta_n} by $\theta_\xi $ and integrating by parts, using \eqref{eqn: ipp proof agmon}, we get
\begin{align*}
    \int_{\Omega_x} |\nabla_x \theta_\xi |^2 + \left( \xi^2 q(x)^2 + \frac{\nu^2 - \mathsf{H}}{|x|^2} - \xi^2 |\nabla_x \psi |^2 - \lambda_\xi   \right)\theta_\xi ^2 = 0.
\end{align*}

Dividing by $\xi^2 $, we obtain 
\begin{align}\label{eqn: proof agmon = 0 divided by xi^2}
    \frac{1}{\xi^2 }\int_{\Omega_x} |\nabla_x \theta_\xi |^2 + \int_{\Omega_x}\left( q(x)^2 + \frac{\nu^2 - \mathsf{H}}{\xi^2 |x|^2} - |\nabla_x \psi |^2 - \frac{\lambda_\xi }{\xi^2 }  \right)\theta_\xi ^2 = 0.
\end{align}

Let us deal with the singular term. If $\nu^2 - \mathsf{H} \geq 0$, then we directly obtain from \eqref{eqn: proof agmon = 0 divided by xi^2} that
\begin{align*}
    \int_{\Omega_x}\left( q(x)^2 - |\nabla_x \psi |^2 - \frac{\lambda_\xi }{\xi^2}  \right)\theta_\xi ^2 \leq 0.
\end{align*}
Now, if $\nu^2 - \mathsf{H} < 0$, by Hardy inequality \eqref{eqn: hardy inequality}, we have, since $\theta_\xi  \in H_0^1\left(\Omega_x \right)$,
    \begin{align}
        \int_{\Omega_x} \frac{\nu^2 - \mathsf{H}}{\xi^2 |x|^2}\theta_\xi ^2 \ dx \geq \frac{\nu^2 - \mathsf{H}}{\xi^2 \mathsf{H}} \int_{\Omega_x} |\nabla_x \theta_\xi|^2 \ dx. 
    \end{align}
Thus, from \eqref{eqn: proof agmon = 0 divided by xi^2}, we again directly obtain 
\begin{align*}
    \int_{\Omega_x}\left( q(x)^2 - |\nabla_x \psi |^2 - \frac{\lambda_\xi }{\xi^2 }  \right)\theta_\xi ^2 \leq 0.
\end{align*}

We can now conclude as in the proof of \cite[Proposition 4.1]{vanlaere2025grushinlike}.\\

Let $\delta_0 > 0$ sufficiently small such that, for every $\delta \in (0, \delta_0)$, the set $F_\delta$ introduced in \eqref{eqn: definition F delta} is a connected neighborhood of zero in $\Omega_x$, which is decreasing for the inclusion as $\delta \rightarrow 0^+$, and moreover satisfies $F_\delta \cap \overline{\omega_x} = \emptyset$ since $\overline{\omega_x} \cap \{0\} = \emptyset$. We choose such a $\delta \in (0,\delta_0)$. \\

As we assumed that $\lambda_\xi  = o(\xi^2 )$, there exists $\xi_0 > 0$, such that for every $\xi \geq \xi_0$, we have 
\begin{align*}
    \left\{ x \in \Omega_x, \quad q(x)^2 \leq \frac{\lambda_\xi}{\xi^2 } \right\} \subset F_\delta.
\end{align*}

Hence, for every $\xi > \xi_0$, we have
\begin{align}\label{eqn: negarive agmon delta}
    \int_{\Omega_x}\left( q(x)^2  - |\nabla_x \psi |^2 - \delta  \right)\theta_\xi ^2\leq 0.
\end{align}
Choose $\psi (x) := (1-\delta)d_{\operatorname{agm},\delta}(x,F_\delta)$, where $d_{\operatorname{agm}, \delta}$ is the degenerated distance \eqref{eqn: definition Agmon distance} induced by the degenerated Agmon metric $(q^2 - \delta)_+ dx^2$, where $dx^2$ is the standard Euclidean metric on $\Omega_x$. We note that we have 
\begin{align}
    |\nabla \psi (x)|^2 \leq (1-\delta)^2(q^2 - \delta)_+.
\end{align}
Therefore, on the set $F_\delta$ we have that $q^2 - |\nabla \psi |^2 - \delta = q^2 - \delta \leq 0$, and on the complementary of $F_\delta$ we have 
\begin{align*}
    q^2 - |\nabla \psi |^2 - \delta &\geq (1-(1-\delta)^2)(q^2 - \delta)_+ \\
    &= (2\delta - \delta^2)(q^2 - \delta)_+ \\
    &\geq 0.
\end{align*}
From \eqref{eqn: negarive agmon delta}, it follows that 
\begin{align*}
    (2\delta - \delta^2) \int_{\Omega_x \setminus F_\delta} (q^2 - \delta) \theta_\xi ^2 \leq - \int_{F_\delta} (q^2 - \delta) \theta_\xi ^2.
\end{align*}

Replacing $\theta_\xi $ by its expression $u_\xi  (x)e^{\xi(1-\delta)d_{\operatorname{agm},\delta}(x,F_\delta)}$, using the fact that $\overline{\omega_x} \cap F_\delta = \emptyset$, and that $d_{\operatorname{agm},\delta}(x,F_\delta) = 0$ on $F_\delta$, we get, on one hand
\begin{align*}
    - \int_{F_\delta}  (q^2 - \delta) \theta_\xi ^2 &=  - \int_{F_\delta}  (q^2 - \delta) u_\xi ^2  \\
    &\leq \sup_{F_\delta}|q^2 - \delta|,
\end{align*}
since $u_\xi  $ is normalized in $L^2$-norm, and on the other hand,
\begin{align*}
 (2\delta - \delta^2) \int_{\Omega_x \setminus F_\delta} (q^2 - \delta) \theta_\xi ^2  &=  (2\delta - \delta^2) \int_{\Omega_x \setminus F_\delta} (q^2 - \delta) u_\xi ^2   e^{2\xi (1-\delta)d_{\operatorname{agm},\delta}(x,F_\delta)} \\
 &\geq (2\delta - \delta^2) \int_{\omega_x} (q^2 - \delta) u_\xi ^2   e^{2\xi (1-\delta)d_{\operatorname{agm},\delta}(x,F_\delta)} \\
 &\geq (2\delta - \delta^2) \min_{\omega_x} (q^2 - \delta) e^{2\xi (1-\delta)d_{\operatorname{agm},\delta}(\omega_x,F_\delta)} \int_{\omega_x}  u_\xi ^2 ,
\end{align*}
where we stress that $\min_{\omega_x} (q^2 - \delta) > 0$. Therefore, we finally get
\begin{align*}
   \int_{\omega_x}  u_\xi ^2  \leq \frac{\sup_{F_\delta}|q^2 - \delta|}{(2\delta - \delta^2) \min_{\omega_x} (q^2 - \delta)} e^{-2\xi (1-\delta)d_{\operatorname{agm},\delta}(\omega_x,F_\delta)},
\end{align*}
which proves the proposition. 
\end{proof}

\subsection{Analysis of the classical operator in the case $\gamma = 1$ in one dimension}\label{section: spectral analysis classical gamma = 1}

Let us assume that $d_x = 1$. Thus, $\Omega_x = (0,L)$, and $\mathsf{H} = 1/4$. The results of this subsection follow from \cite{romankummer2025}. We thus consider the operator 
\begin{align}
    G_\xi = -\partial_x^2 + \xi^2x^{2} + \frac{\nu^2 - 1/4}{x^2},
\end{align}
and we denote its eigenfunctions by $u_{\xi,k}$, not necessarily normalized in norm, associated to the eigenvalues $\lambda_{\xi,k}$.\\

The Kummer equation is defined by, for $a,b,z \in \mathbb{C}$,
\begin{align}\label{eqn: Kummer equation}
    z w''(z) + \left( b - z \right) w'(z) - a w(z) = 0.
\end{align}
The Kummer function is a solution of \eqref{eqn: Kummer equation}, defined by, as long as $b$ is not a negative integer,
\begin{align}\label{eqn: definition kummer function}
    M (a,b,z) &= \sum_{k \geq 0} \frac{(a)_k}{(b)_k k!}z^k, 
\end{align}
where $(\cdot)_k = \cdot(\cdot+1)...(\cdot+k-1)$, $(\cdot)_0 = 1$, is the Pochhammer symbol. \\

We showed in \cite[Section 2]{romankummer2025}, that the eigenfunctions of $G_\xi$ are expressed in terms of the above Kummer functions. In \cite{romankummer2025}, we made the analysis on $(0,1)$, but by dilatation of the operator, it is easily translatable on $(0,L)$. 

\begin{proposition}{\cite[Theorem 1.1]{romankummer2025}}\label{prop: exact form eigenfunctions singular + cond eigenv}
    The (non-normalized) eigenfunctions $u_{\xi,k}$ of $G_\xi$, associated to $\lambda_{\xi,k}$, $k \geq 0$, are  
    \begin{align}
    u_{\xi,k}(x) = A e^{- \xi x^2/2}x^{\frac{1}{2} + \nu}M \left(-\frac{\lambda_{\xi,k} - 2\xi (1 + \nu)}{4\xi}, 1 + \nu, \xi x^2 \right), \quad A \in \mathbb{R},
    \end{align}
where $\lambda_{\xi,k} \geq 0$ is an eigenvalue if and only if
\begin{align}\label{eqn : condition eigenvalue zero kummer}
    M \left(-\frac{\lambda_{\xi,k} - 2\xi (1 + \nu)}{4\xi},1 + \nu, \xi L^2 \right) = 0.
\end{align}
\end{proposition}

The Kummer function \eqref{eqn: definition kummer function}, from which the eigenfunctions are described, is the only sufficiently regular solution of the Kummer equation \eqref{eqn: Kummer equation} at $x = 0$ to belong to $D(G_\xi)$. The above Proposition \ref{prop: exact form eigenfunctions singular + cond eigenv} will be of use in Section \ref{section: spectral analysis generalized}, where we treat the generalized operator. \\

The transformation $z = x/L$ sends $G_\xi$ to the operator 
\begin{align}\label{eqn: classical grushin operator dilated}
    -\partial_z^2 + \xi^2L^4z^2 + \frac{\nu^2 - 1/4}{z^2}
\end{align}
on $L^2(0,1)$, for which the eigenvalues are given by $L^2\lambda_{\xi,k}$. We obtained a precise description of the eigenvalues of \eqref{eqn: classical grushin operator dilated} in \cite{romankummer2025}, for any $k \leq \lfloor \tau \xi L^2/4\rfloor$, where $\tau \in (0,1)$ is fixed, and uniformly with respect to $\xi > 0$ sufficiently large. Namely, from \cite[Theorems 1.3 and 1.4]{romankummer2025}, we derive the following. 

\begin{proposition}\label{prop: localization eigenvalues classical gamma = 1 d=1} 
    Let $\nu > 0$. For every $\tau \in (0,1)$, there exists $\xi_0 > 0$, $C_1,C_2 > 0$, such that for every $\xi \geq \xi_0$, for every $k \leq \left\lfloor\frac{\tau \xi L^2 }{4} \right\rfloor$,  
        \begin{align}
           4k + 2(1+\nu) < \frac{\lambda_{\xi,k}}{\xi} \leq 4k + 2(1+\nu) + C_1e^{-C_2 \xi}.
        \end{align}
\end{proposition}

\subsection{Analysis of the classical operator in the case $\gamma > 1$ in one dimension}\label{section: spectral analysis classical gamma > 1}

Let us still assume that $d_x = 1$. Thus, $\Omega_x = (0,L)$, and $\mathsf{H} = 1/4$. Here, we consider the operator 
\begin{align}
    G_\xi = -\partial_x^2 + \xi^2 x^{2\gamma} + \frac{\nu^2 - 1/4}{x^2},
\end{align}
and we denote its eigenvalues by $\lambda_{\xi,k}$.\\

The case $\gamma = 1$ is also covered in this subsection but in a less precise way than Proposition \ref{prop: localization eigenvalues classical gamma = 1 d=1}. We simply obtain a rough estimate on the first eigenvalue of our operators in Fourier for $\xi > 0$ large. This is sufficient for the proof of Theorem \ref{thm: main theorem control classical grushin rectangle vertical 2}. \\

Let us introduce some notations for the quadratic forms associated to our operators. Set 
\begin{align}
\mathcal{Q}_{\xi,L}(u,v)  = \displaystyle\int_0^L u'(x)v'(x) + \xi^2 x^{2\gamma}u(x)v(x) + \frac{\nu^2  - 1/4}{x^2}u(x)v(x)\ dx,
\end{align}
where $\mathcal{Q}_{\xi,L}$ is defined, for any $\nu > 0$, on $H_0^1(0,L)^2$. Observe that $\mathcal{Q}_{\xi,L}$ is the quadratic form associated with $G_\xi$ in $L^2(0,L)$, and that $\mathcal{Q}_{1,\infty}$ is the one associated with $\mathcal{G}_\gamma :=-\partial_x^2 + x^{2\gamma} + (\nu^2 - 1/4)/x^2$ on $L^2(0,+\infty)$, analogously defined as $\mathcal{G}$ introduced in Section \ref{section: heat kernel computation}. Moreover, by Hardy inequality \eqref{eqn: hardy inequality}, since $\nu > 0$, recall that they are positive-definite for any fixed $\nu > 0$.\\

Let us denote the first eigenvalue of $\mathcal{G}_\gamma$ by $\overline{\mu}$. Recall that by the min-max formula, we have
    \begin{align}
    \begin{array}{ccl}
        \lambda_{\xi,0} &=& \min_{u \in H_0^1(0,L), \|u\|=1} \mathcal{Q}_{\xi,L}(u,u), \\[6pt]
        \overline{\mu} &=& \inf_{u \in H_0^1(0,+\infty), \|u\|=1} \mathcal{Q}_{1,\infty}(u,u).
    \end{array}        
    \end{align}

\begin{lemma}{\cite[Lemma 3.5]{prandi2018quantum}}\label{energy cutoff lemma}
    Let $\chi$ be a real-valued Lipschitz function with compact support in $(0,+\infty)$. Let $u\in H^1(0,+\infty)$. Then, we have 
    \begin{align}
        \mathcal{Q}_{1,\infty} (\chi u,\chi u) = \mathcal{Q}_{1,\infty} (u,\chi^2 u) + \langle u , |\chi' |^2 u \rangle.
    \end{align}
\end{lemma}

We now follow \cite[Proposition 3.3]{vanlaere2025grushinlike}.

\begin{proposition}\label{prop: eigenvalue classical gamma geq 1}
For every $\xi > 0$, we have
    \begin{align}\label{lower bound mu_n}
         \lambda_{\xi,0} \geq \overline{\mu}\xi^{\frac{2}{1+\gamma}}.
    \end{align}
Moreover, for every $\epsilon > 0$ sufficiently small, there exists $\xi_0 > 0$, such that for every $\xi \geq \xi_0$, we have
\begin{align}\label{upper bound mu_n}
    \lambda_{\xi,0} \leq (1 + \epsilon)(\overline{\mu} + \epsilon)\xi^{\frac{2}{1+\gamma}}.
\end{align}
\end{proposition}

\begin{proof} We start with the lower bound.
Let $\tau_\xi := \xi^{1/(1 + \gamma)}$. Making the change of variable $s = \tau_\xi x$, and setting $v(s) = u(\tau_{\xi}^{-1}y)\tau_{\xi}^{-1/2}$, we obtain
\begin{align}\label{eqn: minmax classical dilated}
    \lambda_{\xi,0} =  \tau_{\xi}^2 \min \{ \mathcal{Q}_{1,\tau_\xi L}(v,v), v \in H_0^1(0, \tau_\xi L), \|v\| = 1\}.
\end{align}

Thus, we have
 \begin{align}
     \lambda_{\xi,0} \geq \tau_{\xi}^2 \overline{\mu}.
 \end{align}

Now we treat the upper bound. Let $0 < \delta < 1$. From \cite[Lemma 3.2]{vanlaere2025grushinlike}, there exist a constant $C > 0$ independent of $\delta$ and $\xi$, and two smooth functions $\chi_{\xi,1}$, $\chi_{\xi,2}$ such that
\begin{itemize}
    \item $0 \leq \chi_{\xi,i}(x) \leq 1$, for $i=1,2$, 
    \item $\chi_{\xi,1}(x) = 1$ on $[0,\delta L\tau_\xi]$ and $\chi_{\xi,1}(x) = 0$ on $(\tau_\xi L, +\infty)$,
    \item $\chi_{\xi,2}(x) = 1$ on $(\tau_\xi L, +\infty)$ and $\chi_{\xi,2}(x) = 0$ on $[0,\delta L\tau_\xi]$,
    \item $\chi_{\xi,1}(x)^2$ + $\chi_{\xi,2}(x)^2 = 1$, for every $x> 0$ and $\xi > 0$, 
    \item $\underset{x \in (0,+\infty)}{\sup} |\chi'_{\xi,i}| \leq \frac{C}{(1-\delta)\tau_\xi}$, for $i=1,2$. 
\end{itemize}
Let $v \in H_0^1(0,+\infty)\cap L^2((0,+\infty),|s|^{2\gamma}ds)$, $\|v\|_{L^2(0,+\infty)}=1 $, be the minimizer of $\mathcal{Q}_{1,\infty}(v,v)$. From Lemma \ref{energy cutoff lemma}, we derive that
\begin{align}\label{eqn: lower bound quadratic form Q_infty proof spectral}
    \mathcal{Q}_{1,\infty}(v,v) = \sum_{i=1}^2 \mathcal{Q}_{1,\infty}(\chi_{\xi,i}v,\chi_{\xi,i}v) - \sum_{i=1}^2 \int_\mathbb{R} |\chi'_{\xi,i}|^2|v|^2 \geq \mathcal{Q}_{1,\infty} (\chi_{\xi,1}v,\chi_{\xi,1}v) - c(\xi)\|v\|_{L^2(\mathbb{R})},
\end{align}
where $c(\xi) := 2\left( \frac{C}{(1-\delta)\tau_\xi} \right)^2$.\\

Henceforth, by the definition \eqref{eqn: minmax classical dilated} of $\lambda_{\xi,0}$ after dilatation, using \eqref{eqn: lower bound quadratic form Q_infty proof spectral}, and with the minimum taken on all non-zero elements, we have  
\begin{align*}
    \lambda_{\xi,0} &\leq \tau_{\xi}^2 \frac{\mathcal{Q}_{1,\tau_\xi L}(\chi_{\xi,1}v,\chi_{\xi,1}v)}{\|\chi_{\xi,1}v\|^2_{L^2(0,\tau_\xi L)}} \\
    &= \tau_{\xi}^2 \frac{\mathcal{Q}_{1,\infty} (\chi_{\xi,1}v,\chi_{\xi,1}v)}{\|\chi_{\xi,1}v\|^2_{L^2(0,+\infty)}} \\
    &\leq \frac{\tau_{\xi}^2}{\|\chi_{\xi,1}v\|^2_{L^2(0,+\infty)}} \left[ \mathcal{Q}_{1,\infty} (v,v) + c(\xi) \right], \quad \text{because $\|v\|_{L^2(\mathbb{R})} = 1$},\\
    &= \frac{\tau_{\xi}^2}{\|\chi_{\xi,1}v\|^2_{L^2(0,+\infty)}} \left[ \overline{\mu} + c(\xi) \right].
\end{align*}
This proves the upper bound since $\|\chi_{\xi,1}v\|_{L^2(0,+\infty)} \rightarrow 1$ and $c(\xi) \rightarrow 0$, as $\xi \rightarrow + \infty$. 
\end{proof}

\subsection{Analysis of the classical operator in higher dimensions}\label{section: spectral analysis classical d > 3}

Let us now assume that $d_x \geq 3$, and that we are under the setting of assumption \ref{assumption: d > 3}. That is,  $0 \in \Omega_x$ and $\mathsf{H} = (d_x - 2)^2/4$. We consider the operators 
\begin{align}\label{eqn: definition G_xi fourier high dimension q(x) = x gamma}
    G_\xi = -\Delta_x + \xi^2 |x|^{2\gamma} + \frac{\nu^2 - \mathsf{H}}{|x|^2},
\end{align}
and we denote its eigenvalues by $\lambda_{\xi,k}$.

\subsubsection{The case of a ball}

Let us first assume that $\Omega_x = B(0,L)$ for some $L > 0$, with $d_x \geq 3$. We want to estimate the first eigenvalue of $G_\xi$ in this case. The idea is to reduce the problem to a radial one on $(0,L)$, and recover the setting of the preceding subsections. In such case, we will have a Dirichlet boundary condition at $r = 0$ and $r = L$, where $r = |x|$. In other word, we recover the eigenvalues and eigenfunctions of the operator $G_\xi$ when $d_x = 1$, thus the analysis of Sections \ref{section: spectral analysis classical gamma = 1} and \ref{section: spectral analysis classical gamma > 1} above. We only sketch the proof and omit some details. \\

First, we write $x = r\theta$, $r \in (0,L)$ and $\theta \in \mathbb{S}^{d_x  -1}$. In these coordinates, the Laplacian writes 
\begin{align}
    -\Delta_x = - \partial_r^2 - \frac{d_x - 1}{r}\partial_r - \frac{1}{r^2}\Delta_{\mathbb{S}^{d_x - 1}},
\end{align}
that acts on $L^2((0,L) \times \mathbb{S}^{d_x - 1}, r^{d_x-1}drd\theta)$, where $d\theta$ is the surface measure on the sphere, and $\Delta_{\mathbb{S}^{d_x - 1}}$ is the Laplace-Beltrami operator on the sphere. Our operator thus becomes 
\begin{align}\label{eqn: definition G_xi fourier coordonnees spheriques}
    \begin{array}{ccl}
    D(G_\xi) &=& \{ f \in H^1 \left((0,L) \times \mathbb{S}^{d_x - 1}, r^{d_x - 1}drd\theta \right), \quad f(L,\theta) = 0 \text{ for } \theta \in \mathbb{S}^{d_x - 1}, \quad G_\xi f \in L^2(\mathbb{R})\},  \\[8pt]
    G_\xi &=& \displaystyle -\partial_r^2 - \frac{d_x - 1}{r}\partial_r - \frac{1}{r^2}\Delta_{\mathbb{S}^{d_x - 1}} + \xi^2 r^{2\gamma} + \frac{\nu^2 - \mathsf{H}}{r^2},
    \end{array}   
\end{align}
We keep the notation $G_\xi$ here for simplicity as no confusion shall arise. \\

We decompose any $f \in L^2((0,L) \times \mathbb{S}^{d_x - 1}, r^{d_x-1}drd\theta)$ in a basis of eigenfunctions of $-\Delta_{\mathbb{S}^{d_x - 1}}$. Recall (see \textit{e.g.} \cite[Chapter 3.C]{berger2006spectre}) that its eigenvalues are given by $\ell(d_x + \ell - 2)$, $\ell \geq 0$, and the corresponding eigenspaces $V_\ell$ are of dimension $m_\ell$, spanned by the eigenfunctions $Y_{\ell,m}$, $1 \leq m \leq m_\ell$,
\begin{align*}
    -\Delta_{\mathbb{S}^{d_x - 1}}Y_{\ell,m} = \ell(d_x + \ell - 2)Y_{\ell,m}, \quad \text{for every } \ell \geq 0 \quad \text{and} \quad 1 \leq m \leq m_\ell.
\end{align*}
We write
\begin{align}
    f(r,\theta) = \sum_{\ell \geq 0} \sum_{m = 1}^{m_\ell} F_{\ell,m}(r)Y_{\ell,m}(\theta).
\end{align}

The first eigenvalue of $G_\xi$ is obtained by considering $\ell = 0$. That is, we solve 
\begin{align}
    \left(- \partial_r^2 - \frac{d_x - 1}{r}\partial_r + \xi^2r^{2\gamma} + \frac{\nu^2 - \mathsf{H}}{r^2} \right)F = \lambda F,
\end{align}
in $L^2((0,L), r^{d_x - 1}dr)$. Making the change of variable $F = r^{-(d_x-1)/2}\mathsf{F}$, $\mathsf{F}$ solves
\begin{align}
    -\mathsf{F}'' + \frac{d_x^2 - 4d_x + 3}{4r^2}\mathsf{F} + \xi^2r^{2\gamma} \mathsf{F} + \frac{\nu^2 - \mathsf{H}}{r^2}\mathsf{F} = \lambda \mathsf{F},
\end{align}
in $L^2(0,L)$, with a Dirichlet boundary conditions at $r = 0$ and $r=L$, and $\mathsf{F}$ has at least the same regularity as $F$ since $d_x \geq 3$. In particular, we must have $\mathsf{F} \in H_0^1(0,L)$.\\

Developing the Hardy constant $\mathsf{H} = (d_x - 2)^2/4$, we obtain
\begin{align}\label{eqn: eigenproblem radial}
    -\mathsf{F}'' + \xi^2r^{2\gamma} \mathsf{F} + \frac{\nu^2 -1/4}{r^2}\mathsf{F}  = \lambda \mathsf{F},
\end{align}
in $L^2(0,L)$. It follows that the eigenfunctions of the above problem \eqref{eqn: eigenproblem radial} are exactly those of the case $d_x = 1$ described in the Section \ref{section: spectral analysis classical gamma = 1}, and thus the eigenvalues coincide. Indeed, the eigenfunctions given in Proposition \ref{prop: exact form eigenfunctions singular + cond eigenv} are the only solution of the eigenproblem that are $H^1$ near zero (see \cite[Section 2]{romankummer2025}). We therefore have the following. 

\begin{proposition}\label{prop: first eigenvalue in the ball}
    Assume that $d_x \geq 3$, and $\Omega_x = B(0,L)$ for some $L > 0$. Then, the first eigenvalue of $G_\xi$ satisfies, as $\xi \rightarrow + \infty$,
    \begin{align}
        \lambda_{\xi,0} \sim 2\xi(1+\nu), \quad \text{if } \gamma = 1,
    \end{align}
    and there exists two constants $c,C > 0$ such that 
    \begin{align}
        c \xi^{2/(\gamma + 1)} \leq \lambda_{\xi,0} \leq C \xi^{2/(\gamma + 1)}, \quad \text{if } \gamma > 1.
    \end{align}
\end{proposition}

\subsubsection{The case of a general smooth bounded domain}

We now assume that $\Omega_x$ is a smooth bounded domain that satisfies assumption \ref{assumption: d > 3}, \textit{i.e.} $0 \in \Omega_x$, and $d_x \geq 3$.

\begin{proposition}\label{prop: first eigenvalue in high dimension}
    Assume that $d_x \geq 3$, $\Omega_x$ has smooth boundary, and $0 \in \operatorname{Int}\left(\Omega_x \right)$. Then, the first eigenvalue of $G_\xi$ satisfies, as $\xi \rightarrow + \infty$,
    \begin{align}
        \lambda_{\xi,0} \sim 2\xi(1+\nu), \quad \text{if } \gamma = 1,
    \end{align}
    and there exists two constants $c,C > 0$ such that 
    \begin{align}
        c \xi^{2/(\gamma + 1)} \leq \lambda_{\xi,0} \leq C \xi^{2/(\gamma + 1)}, \quad \text{if } \gamma > 1.
    \end{align}
\end{proposition}

\begin{proof}
    By assumption \ref{assumption: d > 3}, there exist $0 < L < L'$ such that 
    \begin{align*}
        B(0,L) \subset \Omega_x \subset B(0,L').
    \end{align*}
    By smoothness of $\partial\Omega_x$, we have
    \begin{align}
        H_0^1(B(0,L)) \subset H_0^1(\Omega_x) \subset H_0^1(B(0,L')).
    \end{align}
    Let us denote the first eigenvalue of $G_\xi$ defined on $L^2(B(0,L))$ by $\lambda_\xi(L)$, and similarly for $L'$. From the min-max theorem, we have the monotonicity property
    \begin{align}
        \lambda_\xi(L') \leq \lambda_{\xi,0} \leq \lambda_\xi(L).
    \end{align}
    This concludes the proof thanks to Proposition \ref{prop: first eigenvalue in the ball}. 
\end{proof}

\subsection{Analysis of the generalized operator}\label{section: spectral analysis generalized}

In this subsection, we extract a sequence of eigenvalues of the generalized operator in Fourier for which we have good estimates. We assume that $d_x = 1$, thus $\Omega_x = (0,L)$, $\mathsf{H} = 1/4$. Here and only here, we will have to distinguish the generalized operator, and its eigenvalues, with the classical one. Thus, we introduce some notations for the sake of clarity. We consider the generalized operator 
\begin{align}
    \mathsf{G}_\xi := -\partial_x^2 + \xi^2q(x)^2 + \frac{\nu^2 - 1/4}{x^2},
\end{align}
with domain as prescribed in \eqref{eqn: definition G_xi bis}, and with $q$ that satisfies \ref{assumption: q} with $\gamma = 1$. Moreover, we denote its eigenvalues by $\Lambda_{\xi,k}$, $k \geq 0$. Those of the classical one ($q(x) = x)$ are denoted by $\lambda_{\xi,k}$, associated to the eigenfunctions $u_{\xi,k}$ introduced in Proposition \ref{prop: exact form eigenfunctions singular + cond eigenv}. We stress that the $u_{\xi,k}$ here are not normalized in norm. \\

We perform a classical perturbative argument, taking advantage of the fact that the eigenfunctions of the classical operator concentrate exponentially fast near zero, where, thanks to assumption \ref{assumption: q},
\begin{align*}
    q(x) \sim q'(0)x.
\end{align*}
We use the following well-known Lemma for self-adjoint operators. 

\begin{lemma}\label{lemma: spectral distance}
    Let $A$ be a self-adjoint operator on a Hilbert space $H$, with domain $D(A)$, and $\lambda \in \mathbb{R}$. Then,
    \begin{align}
        \operatorname{dist}(\lambda, \sigma(A)) \leq \frac{\|(A-\lambda)u\|}{\|u\|}, \quad \text{for every } u \in D(A),
    \end{align}
where $\sigma(A)$ denotes the spectrum of $A$.
\end{lemma}

We will also need the following Lemma.

\begin{lemma}\label{lemma: monotony kummer function}
    Let $\nu > 0$. There exists $\xi_0 > 0$, such that for every $\xi \geq \xi_0$, the function 
    \begin{align}\label{eqn: definition of kummer in lemma monotonicity}
        x \in [0,+\infty) \mapsto M\left(-\frac{\lambda_{\xi,0} - 2\xi(1+\nu)}{4\xi}, 1+\nu,\xi x^2 \right),
    \end{align}
    where $M(\cdot,\cdot,\cdot)$ is the Kummer function introduced in \eqref{eqn: definition kummer function}, is decreasing and strictly concave.
\end{lemma}

\begin{proof}
    Let us denote the function introduced in \eqref{eqn: definition of kummer in lemma monotonicity} by
    \begin{align}
        x \in [0,+\infty) \mapsto K(a_{\xi,0},1+\nu,x),
    \end{align}
    where we set 
    \begin{align}
        a_{\xi,0} := -\frac{\lambda_{\xi,0} - 2\xi(1+\nu)}{4\xi}.
    \end{align}
    We prove that $K'\left(a_{\xi,0}, 1 + \nu, x \right) < 0$ and $K''\left(a_{\xi,0}, 1 + \nu, x \right) < 0$ for every $x \in (0,+\infty)$, where the prime notation stands for the differentiation with respect to $x$. \\

    We have, by definition of the Kummer function (or see \cite[Section 13.3(ii) p. 325]{OlverHandbook2010}), 
    \begin{align}
       K'\left(a_{\xi,0}, 1 + \nu, x \right)= 2\xi x\frac{ a_{\xi,0}}{1+\nu}M\left(a_{\xi,0} + 1, 2 + \nu, \xi x^2 \right).
    \end{align}
    Now, by Proposition \ref{prop: localization eigenvalues classical gamma = 1 d=1}, there exists $\xi_0 > 0$, such that for every $\xi \geq \xi_0$, $a_{\xi,0} < 0$ is exponentially close to $0$, or in particular $a_{\xi,0} > - 1/2$. It follows that for every $\xi \geq \xi_0$, we have $a_{\xi,0} + 1 > 0$. By definition of the Kummer function, we observe that $M(a,b,z)$ does not have any real zeros if $a,b > 0$, and that  $M(a,b,z) > 0$. It follows that $M\left(a_{\xi,0} + 1, 2 + \nu, \xi x^2 \right) > 0$, and thus, $K'\left(a_{\xi,0}, 1 + \nu, x \right) < 0$ for every $x \in (0, + \infty)$ since $a_{\xi,0} < 0$, and $M'\left(a_{\xi,0}, 1 + \nu, 0 \right) = 0$. \\
    
    Reusing once again these arguments we obtain, for $\xi \geq \xi_0$, 
    \begin{align*}
        K''\left(a_{\xi,0}, 1 + \nu, x \right) &= 2\xi\frac{ a_{\xi,0}}{1+\nu}M\left(a_{\xi,0} + 1, 2 + \nu, \xi x^2 \right) +  4\xi^2 x^2\frac{a_{\xi,0}(a_{\xi,0}+1)}{(1+\nu)(2+\nu)}M \left(a_{\xi,0} + 2, 3 + \nu, \xi x^2 \right) \\
        &< 0,
    \end{align*}
    for every $x \in [0,+\infty)$, since the first term is negative because $a_{\xi,0} < 0$, and since the second term is negative for the same reason, noting that $a_{\xi,0} + 1 > 0$. This concludes the proof.
\end{proof}

\begin{remark}
    Lemma \ref{lemma: monotony kummer function}, together with Proposition \ref{prop: localization eigenvalues classical gamma = 1 d=1}, is sufficient to obtain the lower bound of Theorem \ref{thm: main theorem control classical grushin rectangle vertical} when $d_x = 1$, $\Omega_x = (0,L)$ for some $L > 0$, without relying on Agmon estimates. 
\end{remark}

\begin{proposition}\label{prop: eigenvalue generalized operator}
    Assume that $q$ satisfies \ref{assumption: q} with $\gamma = 1$. There exists $\xi_0 >0$ and a constant $C > 0$ such that for every $\xi \geq \xi_0$, $\mathsf{G}_\xi$ admits an eigenvalue $\Lambda_\xi$ that satisfies
    \begin{align}
        |\Lambda_\xi - 2q'(0)\xi(1+\nu)| \leq C\sqrt{\xi}.
    \end{align}
\end{proposition}

\begin{proof}
    Write the generalized operator as
    \begin{align}
        -\partial_x^2 + q(x)^2\xi^2 + \frac{\nu^2 - 1/4}{x^2} = -\partial_x^2 + q'(0)^2\xi^2 x^2 + \frac{\nu^2 - 1/4}{x^2} + \xi^2R(x),
    \end{align}
    where 
    \begin{align}\label{eqn: decomposition R(x)}
        R(x) =  q(x)^2 - q'(0)^2 x^2 = O(x^{3}),
    \end{align} 
    and the big $O$ notation holds near $x = 0$, thanks to assumption \ref{assumption: q}. Thus, 
    \begin{align}
        \mathsf{G}_\xi = G_{q'(0)\xi} + \xi^2R(x),
    \end{align}
    where $G_{q'(0)\xi}$ stands for the classical operator. We have 
    \begin{align}\label{eqn: G_xi - lambda_xi in proof of generalized eigenvalues}
        ( \mathsf{G}_\xi - \lambda_{q'(0)\xi,0})u_{q'(0)\xi,0} = \xi^2 R(x)u_{q'(0)\xi,0}.
    \end{align}
    To simplify the presentation, let us set
    \begin{align}
        \lambda := \lambda_{q'(0)\xi,0}, \quad \text{and} \quad u :=u_{q'(0)\xi,0}. 
    \end{align}
    Fix some $\delta \in (0,1) \cap (0,L)$ such that we have \eqref{eqn: decomposition R(x)} on $(0,\delta)$. We have
    \begin{align*}
        \int_0^L R(x)^2u(x)^2 \ dx = \int_0^\delta R(x)^2u(x)^2 \ dx + \int_\delta^L R(x)^2u(x)^2 \ dx = I_1 + I_2.
    \end{align*}
    Let us first focus on $I_1$. Recall the explicit form of $u$ from Proposition \ref{prop: exact form eigenfunctions singular + cond eigenv} with $A = 1$,
    \begin{align}\label{eqn: def eigenfunction classical q'(0)xi}
        u(x) = e^{- q'(0)\xi x^2/2}x^{\frac{1}{2} + \nu}M \left(-\frac{\lambda - 2q'(0)\xi (1 + \nu)}{4q'(0)\xi}, 1 + \nu, q'(0)\xi x^2 \right).
    \end{align}
   Since the Kummer function in \eqref{eqn: def eigenfunction classical q'(0)xi} values $1$ at $x=0$, and $0$ at $x=L$, it follows from Lemma \ref{lemma: monotony kummer function} that for any $\xi$ sufficiently large,
   \begin{align}\label{eqn: pointwise bound first kummer function}
       1 - \frac{x}{L} \leq M \left(-\frac{\lambda - 2q'(0)\xi (1 + \nu)}{4q'(0)\xi}, 1 + \nu, q'(0)\xi x^2 \right) \leq 1, \quad x \in [0,L].
   \end{align}
   Therefore, 
   \begin{align}\label{eqn: pointwise bound first eigenfunction classical}
      \left(1 - \frac{x}{L} \right)^2 e^{- q'(0)\xi x^2}x^{1 + 2\nu} \leq u(x)^2 \leq e^{- q'(0)\xi x^2}x^{1 + 2\nu}.
   \end{align}
   Assuming $\xi \geq 1/\delta^2$, we decompose $I_1$ as a sum of integrals on $(1,1/\sqrt{\xi})$ and $(1/\sqrt{\xi},\delta)$. We have, on one hand, using \eqref{eqn: decomposition R(x)} and \eqref{eqn: pointwise bound first eigenfunction classical},
   \begin{align*}
       \int_0^{1/\sqrt{\xi}} R(x)^2u(x)^2 \ dx &\leq \int_0^{1/\sqrt{\xi}} x^6 x^{1 + 2\nu} e^{- q'(0)\xi x^2}  \ dx \\
       &\leq \left( \frac{1}{ \sqrt{\xi}} \right)^{7 + 2\nu} \int_0^{1/\sqrt{\xi}}  e^{- q'(0)\xi x^2}  \ dx \\
       &= \left( \frac{1}{ \sqrt{\xi}} \right)^{8 + 2\nu} \int_0^1  e^{- q'(0) x^2}  \ dx.
   \end{align*}
   On the other hand, making the change of variable $z = \sqrt{\xi}x$, using again \eqref{eqn: decomposition R(x)} and \eqref{eqn: pointwise bound first eigenfunction classical},
   \begin{align*}
       \int_{1/\sqrt{\xi}}^\delta R(x)^2u(x)^2 \ dx &\leq \int_{1/\sqrt{\xi}}^\delta x^6 x^{1 + 2\nu} e^{- q'(0)\xi x^2}  \ dx \\
       &= \left( \frac{1}{\sqrt{\xi}}\right)^{8+2\nu} \int_1^{\delta\sqrt{\xi}} z^6 z^{1 + 2\nu} e^{- q'(0)z^2}  \ dz \\
       &\leq \left( \frac{1}{\sqrt{\xi}}\right)^{8+2\nu} \int_1^\infty z^6 z^{1 + 2\nu} e^{- q'(0) z^2}  \ dz.
   \end{align*}
   Hence, thanks to \eqref{eqn: G_xi - lambda_xi in proof of generalized eigenvalues}, we obtained that, for some $C > 0$, and any $\xi > 0$ sufficiently large, 
   \begin{align}
       \frac{\|(\mathsf{G}_\xi - \lambda)u\|^2}{\|u\|^2} &\leq C \frac{\xi^4}{\|u\|^2} \left( \frac{1}{\sqrt{\xi}}\right)^{8+ 2\nu} + \frac{\xi^4}{\|u\|^2}\int_\delta^L R(x)^2u(x)^2 \ dx.
   \end{align}
   We have, by boundedness of $R$, for some $ C > 0$ different from the preceding one,
   \begin{align*}
       \frac{\xi^4}{\|u\|^2}\int_\delta^L R(x)^2u(x)^2 \ dx \leq C\xi^4 \int_\delta^L \frac{u(x)^2}{\|u\|^2} \ dx.
   \end{align*}
   
   We can apply Proposition \ref{prop: agmon decay grushin general} to the above inequality, as $\lambda_{q'(0)\xi} = o(\xi^2)$ thanks to Proposition \ref{prop: localization eigenvalues classical gamma = 1 d=1}, and $u/\|u\|$ is a normalized in norm eigenfunction. It is therefore exponentially small on $(\delta,L)$, and we obtain, for some new constant $C>0$, and for every $\xi$ sufficiently large, \\
   \begin{align}\label{eqn: unfinished estimate spectral distance}
       \frac{\|(\mathsf{G}_\xi - \lambda)u\|^2}{\|u\|^2} &\leq C \frac{\xi^4}{\|u\|^2} \left( \frac{1}{\sqrt{\xi}}\right)^{8+ 2\nu}.
   \end{align}
   
   Let us bound from below the norm of $u$, in order to estimate from above \eqref{eqn: unfinished estimate spectral distance}, and conclude the proof. Using \eqref{eqn: pointwise bound first eigenfunction classical}, we have that there exists some $C > 0$ different from the preceding ones, such that for every $\xi$ large enough, 
   \begin{align*}
       \int_0^L u(x)^2 \ dx &\geq \int_0^L x^{1+2\nu}\left(1 - \frac{x}{L} \right)^2 e^{-q'(0)\xi x^2} \ dx \\
       &\geq \int_0^{1/\sqrt{\xi}} x^{1+2\nu}\left(1 - \frac{x}{L} \right)^2 e^{-q'(0)\xi x^2} \ dx \\
       &= \left( \frac{1}{\sqrt{\xi}}  \right)^{2+2\nu} \int_0^1 z^{1+2\nu} \left(1 - \frac{z}{\sqrt{\xi}L} \right)^2  e^{-q'(0)z^2} \ dz \\
       &\geq \left( \frac{1}{\sqrt{\xi}}  \right)^{2+2\nu}\left(1 - \frac{1}{\sqrt{\xi}L} \right)^2  \int_0^1 z^{1+2\nu}e^{-q'(0)z^2} \ dz \\
       &\geq C\left( \frac{1}{\sqrt{\xi}}  \right)^{2+2\nu}.
   \end{align*}
   Hence, applying Lemma \ref{lemma: spectral distance}, from what we computed above combined with \eqref{eqn: unfinished estimate spectral distance}, we finally obtain, for some $C > 0$, for every $\xi > 0$ sufficiently large, 
   \begin{align*}
       \operatorname{dist}(\lambda, \sigma(\mathsf{G}_\xi)) &\leq \frac{\|(\mathsf{G}_\xi-\lambda)u\|}{\|u\|} \\
       &\leq C \xi^2 \left( \frac{1}{\sqrt{\xi}}\right)^{4+ \nu}\left(\sqrt{\xi}  \right)^{1+\nu}\\
       &= C\sqrt{\xi},
   \end{align*}
    which concludes the proof, since, by Proposition \ref{prop: localization eigenvalues classical gamma = 1 d=1}, $\lambda := \lambda_{q'(0)\xi,0} \sim 2q'(0)\xi(1+\nu)$ as $\xi \rightarrow + \infty$. 
\end{proof}

\section{Proofs of the theorems}\label{section: proofs}

Recall from Section \ref{section : sketch of proof carleman} that we need to study whether or not uniform observability \eqref{eqn: observability inequality uniform fourier} holds for system \eqref{stm : adjoint system generalized fourier}. We fix the observation set $\omega_x \subset \Omega_x$ such that $0 \notin \overline{\omega_x}$.

\subsection{Upper bound for the minimal time of observability in Theorem \ref{thm: main theorem control classical grushin rectangle vertical}}\label{section: proof positive results classical equation}\label{section: proof positive result}

Here we consider system \eqref{stm : adjoint system generalized fourier} with $\Omega_x$ that satisfies \ref{assumption: d = 1} or \ref{assumption: d > 3}, and $q(x) = |x|$. We assume that there exists a bounded subdomain $\mathcal{O} \subset \subset \Omega_x$ such that $\partial \mathcal{O} \subset \partial\omega_x$ and $\omega_x \subset \Omega_x \setminus \mathcal{O}$ (see Figure \ref{fig: geometric_assumption}). Since $\omega_x$ is open, there exists $\epsilon_0 > 0$ such that for every $\epsilon \in (0,\epsilon_0)$, the set $\mathcal{O}_\epsilon := \{x \in \Omega_x, \operatorname{dist}(x,\mathcal{O}) < \epsilon\}$ satisfies $\partial \mathcal{O}_\epsilon \subset \omega_x$. Moreover, since $\mathcal{O}$ is smooth, up to choosing $\epsilon_0$ smaller, the set $\mathcal{O}_\epsilon$ is also smooth. We choose such a $\epsilon \in (0,\epsilon_0)$.\\

Letting $n \in \mathbb{N}^*$, and $f_{0,n} \in L^2(\Omega_x)$, we denote the associated solution of system \eqref{stm : adjoint system generalized fourier} by $f_n$. \\

As explained in Section \ref{section : sketch of proof carleman}, we first prove that systems \eqref{stm : adjoint system generalized fourier} are uniformly observable from $\partial\mathcal{O}_\epsilon$. We let $T_0 > 0$, and we apply Proposition \ref{prop: cost of boundary observability} with $\mathcal{O}_\epsilon$. We obtain, for $n$ large enough,  
    \begin{align}\label{eqn: proof positive result cost}
        \int_{\mathcal{O}_\epsilon} |f_n(T_0,x)|^2\ dx \leq C\xi_n e^{L^2/2T_0}e^{\xi_n L^2} \int_0^T \int_{\partial \mathcal{O}_\epsilon} |\partial_{\nu_x} f_n(t,x)|^2 \ dS \ dt, 
    \end{align}
where $L = \sup \{|x|, \ x \in \mathcal{O}_\epsilon\}$.\\

Now, let $T > T_0$. Thanks to Proposition \ref{prop: localization eigenvalues classical gamma = 1 d=1} if $d_x = 1$, and Proposition \ref{prop: first eigenvalue in high dimension} if $d_x \geq 3$, for any $n \geq 1$ we have
\begin{align}\label{eqn: proof positive result dissipation}
    \|f_n(T)\|_{L^2(\mathcal{O}_\epsilon)}^2 \leq e^{-2\lambda_{n,0}(T-T_0)}\|f_n(T_0)\|_{L^2(\mathcal{O}_\epsilon)}^2 \leq e^{-4\xi_n(1+\nu)(T-T_0)}\|f_n(T_0)\|_{L^2(\Omega_x)}^2.
\end{align}
Thus, combining \eqref{eqn: proof positive result cost} and \eqref{eqn: proof positive result dissipation}, 
\begin{align*}
    \int_{\mathcal{O}_\epsilon} |f_n(T,x)|^2 \ dx &\leq e^{-4\xi_n(1+\nu)(T-T_0)} C\xi_n e^{L^2/2T_0}e^{\xi_n L^2} \int_0^T \int_{\partial \mathcal{O}_\epsilon} |\partial_{\nu_x} f_n(t,x)|^2 \ dS \ dt.
\end{align*}

It follows that for any $n$ large enough, the coefficient in front of the integral in the right-hand side above is uniformly bounded by some constant $C'> 0$ as long as 
\begin{align}
    T > \frac{L^2}{4(1+\nu)} + T_0 = \frac{\sup \{|x|^2, \ x \in \mathcal{O}_\epsilon\}}{4(1+\nu)} + T_0.
\end{align}
We thus deduce a first upper bound on the minimal time of boundary observability for system \eqref{stm : adjoint system classical} posed on $\mathcal{O}_\epsilon$ with Dirichlet boundary conditions, since $T_0 > 0$ can be taken arbitrary small, 
\begin{align}\label{eqn: upper bound minimal time obs O_epsilon}
    T(\partial \mathcal{O}_\epsilon \times \Omega_y) \leq \frac{\sup \{|x|^2, \ x \in \mathcal{O}_\epsilon\}}{4(1+\nu)}.
\end{align}
This is equivalent to saying that the minimal time of boundary null-controllability of system \eqref{eqn: boundary control system posed on O_epsilon} from $\partial \mathcal{O}_\epsilon \times \Omega_y$ satisfies the same upper bound as in \eqref{eqn: upper bound minimal time obs O_epsilon}.\\

Using classical extension arguments, this yields that that the minimal time of internal null-controllability of system \eqref{eqn: internal control system from O_epsilon posed on Omega} from $(\mathcal{O}_\epsilon \setminus \mathcal{O}) \times \Omega_y$ also satisfies the upper bound 
\begin{align}
    T \left((\mathcal{O}_\epsilon \setminus \mathcal{O}) \times \Omega_y \right) \leq \frac{\sup \{|x|^2, \ x \in \mathcal{O}_\epsilon\}}{4(1+\nu)}.
\end{align}
Now, since $(\mathcal{O}_\epsilon \setminus \mathcal{O}) \times \Omega_y \subset \omega$, we have $T(\omega) \leq T((\mathcal{O}_\epsilon \setminus \mathcal{O}) \times \Omega_y)$, and since $\epsilon > 0$ can be taken arbitrary small, we deduce that the minimal time of observability of system \eqref{stm : adjoint system classical} from $\omega$ satisfies
\begin{align}
    T(\omega) \leq \frac{\sup \{|x|^2, \ x \in \mathcal{O}\}}{4(1+\nu)},
\end{align}
which concludes the proof. 

\subsection{Lower bound for the minimal time of observability in Theorems \ref{thm: main theorem control classical grushin rectangle vertical}, \ref{thm: main theorem control classical grushin rectangle vertical 2}  and \ref{thm: main theorem control generalized grushin rectangle vertical}}\label{section: proofs negative results}

Let us now obtain lower bounds on the minimal times of observability. That is, obtain a negative result for Theorem \ref{thm: main theorem control classical grushin rectangle vertical}, and prove Theorems \ref{thm: main theorem control classical grushin rectangle vertical 2} and \ref{thm: main theorem control generalized grushin rectangle vertical}. All these proofs are treated in a similar manner. 

\begin{proof}[Proof of the lower bound in Theorem \ref{thm: main theorem control classical grushin rectangle vertical}]

We consider system \eqref{stm : adjoint system generalized fourier} with $\Omega_x$ that satisfies \ref{assumption: d = 1} or \ref{assumption: d > 3}, and $q(x) = |x|^2$. Denote, for each $n$, the first eigenfunction normalized in norm by $u_{n,0}$, associated to $\lambda_{n,0}$. Assume that uniform observability \eqref{eqn: observability inequality uniform fourier} holds for some $C > 0$ and some horizon time $T > 0$. Testing \eqref{eqn: observability inequality uniform fourier} against the sequence of solutions $e^{-\lambda_{n,0}t}u_{n,0}$, this implies that there exists some $C > 0$, such that, for every $n \geq 1$, 
\begin{align}\label{eqn: proof negative result implied obs ineq}
    1 \leq C T e^{2\lambda_{n,0}T} \int_{\omega_x} |u_{n,0}(x)|^2 \ dx.
\end{align}

Thanks to Proposition \ref{prop: localization eigenvalues classical gamma = 1 d=1} if $d_x = 1$, and Proposition \ref{prop: first eigenvalue in high dimension} if $d_x \geq 3$, we obtain, for any $\epsilon > 0$ arbitrary small, and for any $n \geq 1$ sufficiently large depending on $\epsilon > 0$, 
\begin{align}\label{eqn: proof negative result classical obs ineq}
    1 \leq C T e^{4\xi_n(1+\epsilon)(1+\nu)T} \int_{\omega_x} |u_{n,0}(x)|^2 \ dx.
\end{align}
Again thanks to Proposition \ref{prop: localization eigenvalues classical gamma = 1 d=1} if $d_x = 1$, and Proposition \ref{prop: first eigenvalue in high dimension} if $d_x \geq 3$, we can apply the Agmon estimates of Proposition \ref{prop: agmon decay grushin general}. Thus, we obtain, for some $C > 0$ that may differ with the preceding one, and any $n$ large enough, 
\begin{align}\label{eqn: proof negative result gamma = 1 classical obs ineq}
    1 \leq C T e^{4\xi_n(1+\epsilon)(1+\nu)T} e^{-2\xi_n(1-\delta)d_{\operatorname{agm},\delta}(\omega_x,F_\delta)},
\end{align}
where we recall that $F_\delta := \{x \in \Omega_x, |x|^2 \leq \delta \}$ is a connected neighborhood of zero in $\Omega_x$, with $\delta$ sufficiently and arbitrary small. Observe that \eqref{eqn: proof negative result gamma = 1 classical obs ineq} cannot hold in the limit $n \rightarrow +\infty$ if 
\begin{align*}
    T < \frac{1-\delta}{1+\epsilon}\frac{d_{\operatorname{agm},\delta}(\omega_x,F_\delta)}{2(1+\nu)}.
\end{align*}
Thus, this gives a first lower bound 
\begin{align}
    T(\omega_x \times \Omega_y) > \frac{1-\delta}{1+\epsilon}\frac{d_{\operatorname{agm},\delta}(\omega_x,F_\delta)}{2(1+\nu)}.
\end{align}
Since the parameters can be chosen arbitrary small, we deduce,
\begin{align}
    T(\omega_x \times \Omega_y) \geq \frac{d_{\operatorname{agm}}(\omega_x)}{2(1+\nu)},
\end{align}
where we recall that $d_{\operatorname{agm}}(\omega_x) = d_{\operatorname{agm},0}(\omega_x,0)$. Since, from the definition \eqref{eqn: definition Agmon distance} of the Agmon distance with $q(x) = |x|$ we have
\begin{align*}
    d_{\operatorname{agm}}(\omega_x) = \int_0^1 s|x_m|^2 \ ds,
\end{align*}
where $x_m \in \overline{\omega}_x$ is such that $|x_m| = \operatorname{dist}(\omega_x,0_{\mathbb{R}^{d_x}})$, we thus get the sought result
\begin{align}
    T(\omega_x \times \Omega_y) \geq \frac{\operatorname{dist}(\omega_x,0_{\mathbb{R}^{d_x}})^2}{4(1+\nu)}.
\end{align}
This concludes the proof of Theorem \ref{thm: main theorem control classical grushin rectangle vertical}.
\end{proof}

\begin{proof}[Proof of Theorem \ref{thm: main theorem control classical grushin rectangle vertical 2}]
    We consider system \eqref{stm : adjoint system generalized fourier} with $\Omega_x$ that satisfies \ref{assumption: d = 1} or \ref{assumption: d > 3}, and $q(x) = |x|^{2\gamma}$, with $\gamma > 1$. We follow the exact same lines as above, but now using Proposition \ref{prop: eigenvalue classical gamma geq 1} if $d_x = 1$, and Proposition \ref{prop: first eigenvalue in high dimension} if $d_x \geq 3$. We are still in position to apply the Agmon estimate of Proposition \ref{prop: agmon decay grushin general}, and we obtain from \eqref{eqn: proof negative result implied obs ineq},
    \begin{align}
        1 \leq CTe^{2c\xi_n^{2/(\gamma + 1)}T}e^{-2\xi_n(1-\delta)d_{\operatorname{agm},\delta}(\omega_x,F_\delta)}.
    \end{align}
    Observe that this can never hold for any time $T > 0$ as $n \rightarrow + \infty$ since $\gamma >1.$ Thus,
    \begin{align}
        T(\omega_x \times \Omega_y) = + \infty,
    \end{align}
    which concludes the proof.
\end{proof}

\begin{proof}[Proof of theorem \ref{thm: main theorem control generalized grushin rectangle vertical}]
    We consider system \eqref{stm : adjoint system generalized fourier} with $\Omega_x = (0,L)$ for some $L > 0$, $q$ satisfies \ref{assumption: q}, and $\omega_x = (a,b)$ for some $0 < a < b \leq L$. Now using Proposition \ref{prop: eigenvalue generalized operator}, there exists for every $n$ large enough, an eigenvalue $\Lambda_{_n}$ that satisfies
    \begin{align}
        |\Lambda_n - 2q'(0)\xi_n(1+\nu)| \leq C\sqrt{\xi_n}.
    \end{align}
    We deduce that for every $\epsilon > 0$, for any $n$ large enough,
    \begin{align}
        \Lambda_n \leq 2(1+\epsilon)q'(0)(1+\nu).
    \end{align}
    Denote by $u_n$ the associated normalized eigenfunction. Thus, we get, as for \eqref{eqn: proof negative result gamma = 1 classical obs ineq}, by applying the Agmon estimates of Proposition \ref{prop: agmon decay grushin general},
    \begin{align}
        1 \leq CT\xi_n^2e^{4(1+\epsilon)\xi_nq'(0)(1+\nu)T}e^{-2\xi_nd_{\operatorname{agm},\delta}(\omega_x,F_\delta)}.
    \end{align}
    Thus, similarly to the classical case, we obtain
    \begin{align}
        T(\omega_x \times \Omega_y) \geq \frac{d_{\operatorname{agm}}(\omega_x)}{2q'(0)(1+\nu)} = \frac{1}{2q'(0)(1+\nu)}\int_0^a q(s) \ ds,
    \end{align}
    which concludes the proof. 
\end{proof}

\section{Extension to almost-Riemannian manifolds and open problems}\label{section: additional comments}

\subsubsection*{Extension to almost-Riemannian manifolds}

Now that we have seen all the tools, and proved the Theorems, let us develop the comments presented in Section \ref{section: comments results and literature}. \\

The study of the observability properties of the systems presented in the present paper, as of those of the non-singular Grushin equation, and related equations, studied in the references given in Section \ref{section: comments results and literature}, is part of a line of research that aims to obtain a geometric interpretation of the observability properties of sub-elliptic heat equations. Formally, given a manifold $\mathcal{M}$ of dimension $n$, a family of vector fields $X_1,...X_m$ on $\mathcal{M}$, and a measure $\mu$, one naturally associates to these a second order differential operator 
\begin{align*}
    \Delta = \sum_{i=1}^m X_i^*X_i,
\end{align*}
where $X_i^*$ is the formal adjoint of $X_i$ in $L^2(\mathcal{M},\mu)$. The operator $\Delta$ is called a sub-Laplacian. Under some particular geometric conditions (see \cite[Chapter 20]{agrachev2019comprehensive}) on the set where the vector fields become linearly dependent, one can define a canonical volume form on $\mathcal{M}$ (\textit{i.e.} using only the sub-Riemannian structure on $\mathcal{M}$), and $\Delta$ is called the Laplace-Beltrami operator. \\

It is of interest to study the observability properties of the heat equation associated to $\Delta$ under the assumption that the vector fields $X_1,...,X_m$ do not span the tangent space to $\mathcal{M}$ at every point of it, but their iterated Lie brackets do. In this case, some vector fields $X_j$, $j \in \{1,...,m\}$, may become linearly dependent. If $\operatorname{Span}(X_1,...,X_m) = T_p\mathcal{M}$ for every $p \in \mathcal{M}$, we are simply looking at the heat equation on Riemannian manifolds which is now well-understood.\\ 

As we mentioned in Section \ref{section: introduction classical operator}, our singular Grushin operator in system \eqref{stm : adjoint system classical} with $d_x  = 1$, and $\nu^2 = \left( \gamma d_y + 1 \right)^2/4$, is exactly the Laplace-Beltrami operator associated with the Riemannian volume form $dx^2 + x^{-2\gamma} \sigma$, being singular on the set $\{0\} \times \Omega_y$, where $\sigma$ is a Riemannian metric on $\Omega_y$. The non-singular operators will correspond to the sub-Laplacian, for the choice of a smooth measure up to the boundary of $\Omega$. \\

If one wants to study sub-elliptic heat equations in all generalities, it is natural that in order to obtain results, one would look at the spectral properties of the underlying sub-Laplacian. What the present paper shows, is that the sub-elliptic equation, in presence of an inverse square potential, sees its observability properties depend on the strength of the singular term, despite the fact that the localization estimates of its eigenfunctions, and its Weyl law, are independent of it. This was made clear in our settings by the use of Fourier decomposition, and study of the associated dissipation speeds, but we do not know how to extract this in all generalities, losing this tool. Also, one may observe from our results that the observability properties for the Laplace-Beltrami operator depend on the dimension of $\Omega_y$, while in the non-singular case, that is for a sub-Laplacian, it does not. \\ 

Following \cite{vanlaere2025grushinlike}, one can still however obtain observability results under suitable assumptions. Assume that $\mathcal{M}$ is a compact manifold of dimension $n+1$, and with boundary $\partial \mathcal{M}$ of dimension $n$. We endow $\mathcal{M}$ with a Riemannian metric that explodes on $\partial\mathcal{M}$, but such that the associated sub-Riemannian structure is of step $1$ on $\mathcal{M}$, and $\gamma + 1$ on $\partial\mathcal{M}$ (that is, we are Riemannian on $\mathcal{M}$, while we need $\gamma$ iterations of Lie brackets to span the tangent space at points of the boundary $\partial\mathcal{M}$, and thus the sub-Riemannian distance to $\partial \mathcal{M}$ is finite). Generically, there exists a neighborhood $\mathcal{U}$ of $\partial \mathcal{M}$, diffeomorphic to $[0,L) \times \partial \mathcal{M}$, in which the metric writes
\begin{align}
    g = dx^2 + x^{-2\gamma} \sigma(x),
\end{align}
with $\gamma \in \mathbb{N}^*$ and $\sigma(x)$ a continuous family of Riemannian metrics on $\partial\mathcal{M}$. The boundary $\partial \mathcal{M}$ is identified with $\{0\} \times \partial \mathcal{M}$. Assuming that in this neighborhood $\sigma(x) = \sigma$ does not depend on $x$, we can obtain in the spirit of \cite{vanlaere2025grushinlike} some results, observing that in $\mathcal{U}$ the Laplace-Beltrami operator $\Delta$ is exactly our Grushin operator introduced in system \eqref{stm : adjoint system classical} with $\nu^2 = (\gamma d_y + 1)^2/4$.\\

In particular, using the process of \cite[Section 5]{vanlaere2025grushinlike}, sketched in \cite[Section 2]{vanlaere2025grushinlike}, and the results presented in this paper, one can show obstruction to small time observability when $\gamma = 1$, and in any time when $\gamma > 1$, from any observation set such that $\overline{\omega} \cap \partial\mathcal{M} = \emptyset$. In the case $\gamma = 1$, one obtains from \eqref{eqn: minimal time laplace beltrami} (which is derived from Theorem \ref{thm: main theorem control classical grushin rectangle vertical}), that 
\begin{align}\label{eqn: lower bound manifold}
    T(\omega) \geq \frac{\min(L,\tilde{L})^2}{6 + 2n},
\end{align}
where $\tilde{L}$ denotes the sub-Riemannian distance from $\omega$ to $\partial\mathcal{M}$. \\

Similarly, if $\gamma > 1$, then from Theorem \ref{thm: main theorem control classical grushin rectangle vertical 2} one obtains that $T(\omega) = + \infty$.\\

In the case $\gamma = 1$, positive results in some particular setting for the observation set to intersect correctly the neighborhood $\mathcal{U}$ can also be obtained, using Theorem \ref{thm: main theorem control classical grushin rectangle vertical}, cutoff arguments, and the fact that $\Delta$ in uniformly elliptic away from $\partial\mathcal{M}$. In particular, for a certain class of observation sets, which at least satisfy $\tilde{L} \leq L$, the lower bound \eqref{eqn: lower bound manifold} is sharp, and
\begin{align}
    T(\omega) = \frac{\tilde{L}^2}{6 + 2n}.
\end{align}

In particular, and it has its importance, for the Laplace-Beltrami operator, the minimal time of observability of the heat equation posed on almost-Riemannian manifolds depends on the dimension of $\partial \mathcal{M}$, that is, on the dimension of the singularity. Thus, we conjecture that the minimal time of observability for the heat equation posed on an almost-Riemannian manifold $\mathcal{M}$ depends on the choice of a measure on $\mathcal{M}$. \\

We mention that in the setting described above, Weyl's laws were studied in \cite{de2024weyl,chitour2024weyl,boscain2016spectral}, and concentration properties of eigenfunctions in \cite{dietze2024concentration, dietze2025critical}.

\subsubsection*{Open problems}

The observation sets considered in Theorems \ref{thm: main theorem control classical grushin rectangle vertical} and \ref{thm: main theorem control classical grushin rectangle vertical 2} are vertical strips. It is therefore natural to wonder what to expect when we observe in the complementary of an horizontal strip. Namely, when $\omega = \Omega_x \times \omega_y$, with $\overline{\omega_y} \subsetneq \Omega_y$. When $d_x = d_y = 1$, and $\nu^2 = \mathsf{H} = 1/4$, we recover the classical Grushin equation on rectangular domains of two dimensions, and this analysis has been performed in the series of papers \cite{duprez2020control,darde2023null,CRMATH_2017__355_12_1215_0} (they also treat the generalized Grushin operator), and if $\Omega = \mathbb{R}^2$ in \cite{LISSYprolate2025}. We expect the approach of \cite{duprez2020control,darde2023null,CRMATH_2017__355_12_1215_0} to work in our setting in two dimensions. However, the proof is very long and technical, and relies on complex analysis arguments. But the reader may be convinced that we may expect analogous results, as the work \cite{LISSYprolate2025} naturally extends to the singular case. Directly following from the proof of \cite[Theorem 1.2]{LISSYprolate2025}, the following holds. 

\begin{theorem*}\label{thm: main theorem control classical grushin 2d cylinder horizontal}
    Consider system \eqref{stm : adjoint system classical} with $\gamma = 1$, $\nu > 0$, and $\Omega = (0,+\infty) \times \mathbb{R}$. For some $\epsilon > 0$, assume that $\omega = (0,+\infty) \times (\mathbb{R} \setminus[-\epsilon,\epsilon])$. Then we have 
    \begin{align}
        T(\omega) = + \infty.
    \end{align}
\end{theorem*}

This is due to the fact the the first eigenfunction of the singular operator $\mathcal{G}_\xi$ on the half-line is of the form 
\begin{align*}
    \Phi_{\xi,0} = x^{\frac{1}{2} + \nu}e^{-\xi x^2/2},
\end{align*}
and thus similar to the one of the Harmonic oscillator (see Section \ref{section: heat kernel computation}). Thus, the computations in \cite{LISSYprolate2025} still hold. \\

We did not give the minimal time for the generalized operator, and Theorem \ref{thm: main theorem control generalized grushin rectangle vertical} is incomplete, but we expect the lower bound to be optimal. In this case, one cannot use the Carleman weight of \cite[Section 3]{beauchard2020minimal}, as it is not well-suited to deal with the singular term. However, we expect a weight of the form 
\begin{align*}
    \theta(t) \int_x^L q(s) \ ds
\end{align*}
to work. One shall probably assume that $q$ satisfies \ref{assumption: q} with $\gamma = 1$, and additionally to that, $q'(x) > 0$ for every $x \in [0,L]$ for technical issues in the computations. Observe that if $q(x) = x$, this is of the form of the classical Carleman weight \eqref{eqn: carleman weight grushin classic} and those presented in Appendix \ref{section: appendix The cost of small time observability from interior subsets}. Without the monotony assumption $q'(x) > 0$, it is still an open problem to compute the minimal time of observability. \\

For the generalized equation, we expect $T(\omega) = + \infty$ for system \eqref{stm : adjoint system generalized} if $q$ satisfies \ref{assumption: q} with $\gamma > 1$, as in \cite[Theorem 2.2]{vanlaere2025grushinlike}. The only missing ingredient is for Proposition \ref{prop: eigenvalue generalized operator} to hold in this case (replacing $\sqrt{\xi}$ by $\xi^{1/(\gamma + 1)}$), which would require a precise understanding of the eigenfunctions of the classical operators. Namely, how do they behave in the sets $(0,\xi^{1/(\gamma+1)})$. \\

It would also be interesting to see if $T(\omega_x \times \Omega_y) = 0$ for system \eqref{stm : adjoint system classical} when $\gamma \in (0,1)$, and deduce from it that $T(\omega) = 0$ for any open set $\omega$, analogously to the non-singular case \cite{beauchard2014null}. \\

Finally, consider system \eqref{stm : adjoint system classical} with $d_x = d_y = 1$, and $\Omega_x = (-L_-,L_+)$, $L_-,L_+ > 0$. It would be interesting to obtain some results in this case too. Approximate controllability was proved in \cite{morancey2015approximate}, but observability (or null-controllability) is still an open problem. We do not expect this to be easy to solve. One shall start by proving that for each $\xi_n \geq 0$, the one-dimensional heat \eqref{stm : adjoint system generalized fourier} is observable. When $\xi_n = 0$, this has been proved in the recent paper \cite{lissy2025null}.

\section*{Acknowledgments}

The author would like to thank his PhD advisors P. Lissy and D. Prandi for their corrections, suggestions, and careful reading of this work. The author would also like to thank S. Ervedoza for discussions concerning the use of the heat kernel in Carleman estimates, and M. Trabut for discussions about her upcoming work \cite{DardeTrabut2025}. 

\appendix

\section{The cost of small time observability from interior subsets}\label{section: appendix The cost of small time observability from interior subsets}

We recall that the following section is motivated by the discussion of Section \ref{section : sketch of proof carleman} to get rid of the geometric restriction on $\omega_x$ in Theorem \ref{thm: main theorem control classical grushin rectangle vertical}. It also provides an alternative proof of the upper bound in Theorem \ref{thm: main theorem control classical grushin rectangle vertical}, at least in the case $d_x = 1$, without having to rely on the equivalent control systems. However, we still rely on boundary Carleman estimates and use of cutoffs, as we will explain later. \\

Following the strategy of proof of boundary observability used in the proof of Theorem \ref{thm: main theorem control classical grushin rectangle vertical} in Section \ref{section: proof positive result}, Proposition \ref{prop: appendix cost of small time observability} below yields a sharp upper bound on the minimal time of internal observability in Theorem \ref{thm: main theorem control classical grushin rectangle vertical} in the case $d_x = 1$. \\

We thus study the cost of small time internal observability of 
\begin{align}\label{stm : appendix adjoint system classical d=1 fourier}
    \left\{ \begin{array}{lcll}
     \partial_t f -\partial_x^2f + \xi^2x^2 f + \displaystyle\frac{\nu^2 - 1/4}{x^2} f & = & 0, & (t,x) \in (0,T) \times (0,L),\\[6pt]
     f(t,x) & = & 0, & t \in (0,T), \ x \in \{0,L\}, \\[6pt]
     f(0,x) & = & f_0(x), & x \in (0,L),
    \end{array} \right.
\end{align}
with $f_0 \in L^2(0,L)$, in the limit $\xi \rightarrow + \infty$. \\

We have the following result, analogous to Proposition \ref{prop: cost of boundary observability} for internal observation sets.

\begin{proposition}\label{prop: appendix cost of small time observability}
    Let $0 < a < b \leq L$. For every $T > 0$, for every $\epsilon > 0$ and $\delta > 0$ small enough, there exist $\xi_0 > 0$ and a constant $C>0$, such that for every $\xi \geq \xi_0$, for every solution $f$ of system \eqref{stm : appendix adjoint system classical d=1 fourier}, we have
    \begin{align}\label{eqn: appendix cost of small time observability}
        \int_0^L |f(T,x)|^2  \ dx \leq C \left( e^{\frac{\xi}{\sqrt{1-3\delta/2}}(a+\epsilon/2)^2} + e^{\epsilon\xi}\right) \int_0^T \int_a^{b} |f(t,x)|^2 \ dx \ dt.
    \end{align}
\end{proposition}

Let us fix some time $T > 0$ and an observation set $(a,b)$, with $0<a<b\leq L$. The proof of Proposition \ref{prop: appendix cost of small time observability} relies on the following two inequalities that we prove later: for every $T > 0$, for every $\epsilon > 0$ and $\delta > 0$, there exists $\xi_0 > 0$ and a constant $C>0$, such that for every $\xi \geq \xi_0$, for every solution $f$ of system \eqref{stm : appendix adjoint system classical d=1 fourier}, we have
    \begin{align}
          \int_{2T/5}^{3T/5} \int_0^{a+\epsilon/2} |f(t,x)|^2  \ dx \ dt \quad &\leq \quad C e^{\frac{\xi}{\sqrt{1-3\delta/2}}(a+\epsilon/2)^2}\int_0^T \int_a^{b} |f(t,x)|^2 \ dx \ dt, \label{eqn: appendix cutoff inequality near sing}    \\[8pt]
           \int_{T/3}^{2T/3} \int_{a+\epsilon/2}^L  |f(t,x)|^2 \ dx \ dt \quad &\leq \quad Ce^{\epsilon\xi} \int_0^T \int_a^b |f(t,x)|^2 \ dx \ dt. \label{eqn: appendix cutoff inequality away sing}
    \end{align}

\begin{remark}
    Inequality \eqref{eqn: appendix cutoff inequality away sing} is obviously non-optimal when $b = L$, and in this case only \eqref{eqn: appendix cutoff inequality near sing} needs to be proved. Indeed, one simply replaces \eqref{eqn: appendix cutoff inequality away sing} by the trivial inequality
    \begin{align*}
        \int_{2T/5}^{3T/5} \int_a^L  |f|^2 \ dx \ dt \quad &\leq \quad \int_0^T \int_a^L |f|^2 \ dx \ dt.
    \end{align*}
\end{remark}

Let us sketch the strategy to prove \eqref{eqn: appendix cutoff inequality near sing} and \eqref{eqn: appendix cutoff inequality away sing}. Let $\epsilon > 0$ be such that $a+\epsilon < b$. Set $\chi_1$ and $\chi_2$ to be  two smooth cutoffs such that (see Figure \ref{fig: graph_cutoff})
\begin{enumerate}[label = (\roman*)]
    \item $\chi_1 = 1$ on $(0,a+\epsilon/2)$, $\chi_1 = 0$ on $(a+3\epsilon/4,L)$, and $0 \leq \chi_1 \leq 1$, 
    \item $\chi_2 = 1$ on $(a+\epsilon/2,L)$, $\chi_2 = 0$ on $(0,a+\epsilon/4)$, and $0 \leq \chi_2 \leq 1$. 
\end{enumerate}
 
\begin{figure}[H]
    \centering
    \includegraphics[width=0.5\linewidth]{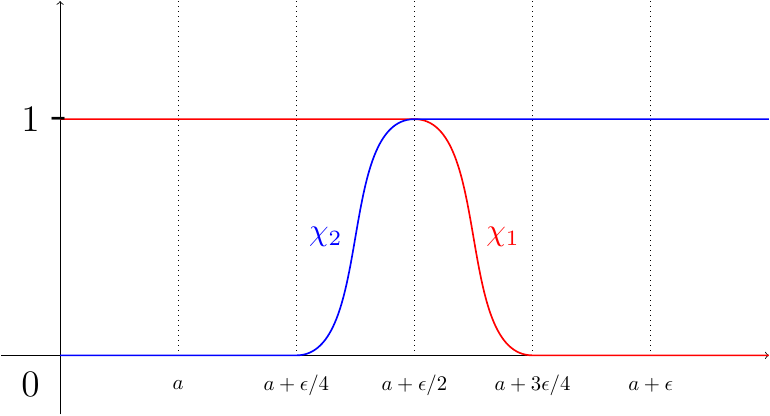}
    \caption{Exhaustive graph of the cutoffs $\chi_1$ and $\chi_2$.}
    \label{fig: graph_cutoff}
\end{figure}

For a solution $f$ of system \eqref{stm : appendix adjoint system classical d=1 fourier}, we  set $f_i = \chi_i f$, $i  = 1,2$. \\

We will apply to $f_1$ a Carleman estimate (see Proposition \ref{prop: appendix carleman near sing}) near the singularity, and inferring \eqref{eqn: appendix cutoff inequality near sing}, and another one to $f_2$ (see Proposition \ref{prop: appendix carleman beauchard}), away from the singularity and yielding \eqref{eqn: appendix cutoff inequality away sing}. Gluing these estimates together proves Proposition \ref{prop: appendix cost of small time observability} as follows. Our strategy is classical (see, e.g., \cite[Section 3.2]{beauchard2020minimal} and \cite[Section III]{vancostenoble2008null}).

\begin{proof}[Proof of Proposition \ref{prop: appendix cost of small time observability}]
    Let $\epsilon > 0$ and $\delta > 0$ be arbitrary small. Let also $\xi$ be sufficiently large to use \eqref{eqn: appendix cutoff inequality near sing} and \eqref{eqn: appendix cutoff inequality away sing}. Summing \eqref{eqn: appendix cutoff inequality near sing} with \eqref{eqn: appendix cutoff inequality away sing}, we obtain, for any solution $f$ of system \eqref{stm : appendix adjoint system classical d=1 fourier},
    \begin{align}\label{eqn: appendix cost integral proof main prop}
        \int_{2T/5}^{3T/5} \int_0^L |f(t,x)|^2  \ dx \ dt \quad &\leq \quad C \left( e^{\frac{\xi}{\sqrt{1-3\delta/2}}(a+\epsilon/2)^2} + e^{\epsilon\xi}\right) \int_0^T \int_a^{b} |f(t,x)|^2 \ dx \ dt.
    \end{align}
    Using the dissipativity of the solutions of system \eqref{stm : appendix adjoint system classical d=1 fourier}, we have
    \begin{align}\label{eqn: appendix dissipativity proof main prop}
        \|f(T)\|^2_{L^2(0,L)} \leq \|f(3T/5)\|^2_{L^2(0,L)} \leq \frac{5}{T} \int_{2T/5}^{3T/5} \int_0^L |f(t,x)|^2  \ dx \ dt.
    \end{align}
    Combining \eqref{eqn: appendix cost integral proof main prop} and \eqref{eqn: appendix dissipativity proof main prop} proves the Proposition.
\end{proof}

The Carleman estimates in the rest of this section are typically used to establish boundary observability, meaning they produce observation terms at the boundary rather than in the interior. In our approach, internal observation terms arise from the cutoffs, which also explains the absence of boundary terms. We rely on both estimates near and away from the singularity because we do not know how to construct, for an arbitrary observation set $\omega_x$, a Carleman weight on the whole domain $\Omega_x$ that produces an interior observation in $\omega_x$ with optimal cost $d$ in \eqref{eqn: cost of observability sketch of proof}. When generalized in higher dimensions, this cutoff-based strategy also justifies the geometric restriction on $\omega_x$ in Theorem \ref{thm: main theorem control classical grushin rectangle vertical}. Although when $\omega_x = \{ r < |x| \leq R \}$ the generalization of our computations are not complicated, in particular when $\Omega_x = B(0,L)$, for arbitrary geometries for $\omega_x$ and/or $\Omega_x$ it can become very tedious.

\subsubsection*{Analysis near the singularity: proof of \eqref{eqn: appendix cutoff inequality near sing}}

We consider system \eqref{stm : appendix adjoint system classical d=1 fourier} with a source term. Let $T > 0$,
\begin{align}\label{stm : appendix adjoint system classical d=1 fourier source}
    \left\{ \begin{array}{lcll}
     \partial_t f -\partial_x^2f + \xi^2x^2 f + \displaystyle\frac{\nu^2 - 1/4}{x^2} f & = & f_\xi, & (t,x) \in (0,T) \times (0,\tilde{L}),\\[6pt]
     f(t,x) & = & 0, & t \in (0,T), \ x \in \{0,\tilde{L}\}, \\[6pt]
     f(0,x) & = & f_0(x), & x \in (0,\tilde{L}).
    \end{array} \right.
\end{align}

System \eqref{stm : appendix adjoint system classical d=1 fourier source} is aimed at being considered with $\tilde{L} = a + \epsilon$.\\

We first introduce a Carleman estimate for the solutions of this system. Define the Carleman weight
\begin{align}\label{eqn: appendix def carleman weight near sing}
    \psi(t,x) = \frac{M\theta(t)}{2}(L^2 - x^2),
\end{align}
where $M > 0$ is aimed at being taken large, and $\theta(t)$ is defined below in \eqref{eqn: appendix def theta near sing}. 

\begin{remark}
    Our Carleman weight \eqref{eqn: appendix def carleman weight near sing} is very similar to the Carleman weight \eqref{eqn: carleman weight grushin classic}. However, we cannot use $\theta(t) = \coth(2\xi t)$ since this would fail to produce an estimate with an integral term on the left-hand side. The function $\theta(t)$ has a strong influence on the optimality of the cost of observability that we can derive from the Carleman estimate of Proposition \ref{prop: appendix carleman near sing} below. This is why in the construction of $\theta$ below, we will require it to be equal to one on some sub-interval of $(0,T)$, and we will need to track the constants in the computations. In particular, obtaining an optimal cost requires taking $M$ as close as possible to $\xi$. To do that, we have to pay the price that our Carleman estimate does not cover all the modes $\xi$ (although it is not an issue to deduce Theorem \ref{thm: main theorem control classical grushin rectangle vertical} from it). 
\end{remark}

Let us construct $\theta$. Let $\eta_i \in C^\infty([0,1];[0,1])$, $i = 1,2,3$, be such that (see Figure \ref{fig:partition_unity}) $\eta_1(s) + \eta_2(s) + \eta_3(s) = 1$, and $\eta_1(s) = 1$ on $(0,1/5)$, $\eta_2(s) = 1$ on $(2/5,3/5)$, $\eta_3(s)=1$ on $(4/5,1)$. We set, for $k > 1$ to be chosen later, 
\begin{align}\label{eqn: appendix def theta near sing}
    \theta(t) = \tilde{\theta}(t/T), \quad \tilde{\theta}(s) = \frac{\eta_1(s)}{s^k} + \eta_2(s) + \frac{\eta_3(s)}{(1-s)^k}. 
\end{align}

\begin{figure}[H]
    \centering
    \includegraphics[width=0.5\linewidth]{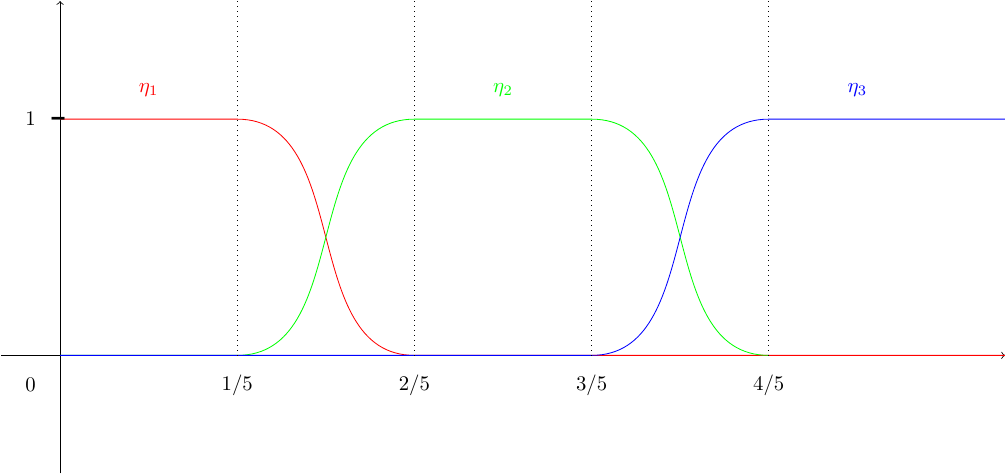}
    \caption{Exhaustive graphs of the functions $\eta_i$.}
    \label{fig:partition_unity}
\end{figure}

We obtain the following Carleman estimate, that we prove later in Section \ref{section: appendix proof of proposition}.

\begin{proposition}\label{prop: appendix carleman near sing}
    For every $T > 0$, for every $\delta > 0$ small enough, there exists a constant $\xi_0 > 0$ depending on $T$ and $\delta$, such that for every $\xi \geq \xi_0$, for every solution $f$ of system \eqref{stm : appendix adjoint system classical d=1 fourier source}, setting $g = fe^{-\psi}$, with $\psi$ defined in \eqref{eqn: appendix def carleman weight near sing}, we have for every $M \geq \xi/\sqrt{1-3\delta/2}$,
    \begin{align}
        \int_Q &\left[ M\theta \left( |\partial_xg|^2 + \frac{c(\nu)}{x^2}|g|^2\right) + \frac{M\theta}{2} \frac{|g|^2}{x^\eta} + \delta M^3\theta^3x^2g^2 \right] \ dQ \notag \\
        &\leq \int_Q |f_\xi|^2e^{-2\psi} \ dQ + \int_0^T ML\theta(t)|g_x(t,L)|^2 \ dt,
    \end{align}
    where we set $Q := (0,T) \times (0,L)$, $dQ := dxdt$.
\end{proposition}

Observe that our Carleman estimate is very close to the one given in \cite[Proposition 5]{cannarsa2014null}, as we shall follow their proof. The main difference is the choice of $\theta(t)$. We also emphasize that this weight is also similar, naturally, to the one of \cite[Theorem 3.2]{vancostenoble2008null}. However, we cannot use the Carleman estimate of \cite[Theorem 3.2]{vancostenoble2008null} as a black box, in which case we should treat $\xi^2 x^2f$ as a source term, and we would lose the optimality of the cost. We need to consider the potential $\xi^2 x^2$ in the principal part of the operator. \\

Let us now prove \eqref{eqn: appendix cutoff inequality near sing} using Proposition \ref{prop: appendix carleman near sing}. Observe that since $f$ solves system \eqref{stm : appendix adjoint system classical d=1 fourier}, then $f_1 = \chi_1 f$ solves system \eqref{stm : appendix adjoint system classical d=1 fourier source} with $\tilde{L} = a+\epsilon$ and with a source term 
\begin{align*}
    f_\xi = - \partial_x^2\chi_1f  - 2 \partial_x\chi_1\partial_xf.
\end{align*}
Choose some $\delta > 0$ small. We apply Proposition \ref{prop: appendix carleman near sing} to $f_1$, for every $\xi \geq \xi_0$, with $M = \xi/\sqrt{1-3\delta/2}$, and by denoting $g_1 = f_1e^{-\psi}$. Since the boundary term at $x = a+\epsilon$ vanishes, we obtain
    \begin{align*}
        \int_Q \frac{M\theta}{2} \frac{g_1^2}{x^\eta}  \ dQ \leq \int_Q |-\partial_x^2\chi_1f  - 2 \partial_x\chi_1\partial_xf|^2 e^{-2\psi} \ dQ.
    \end{align*}
By definition of $\chi_1$, by using the fact that $\theta(t) = 1$ on $(2T/5,3T/5)$, and by triangular inequality, this implies 
    \begin{align*}
        \frac{M(a+\epsilon/2)^\eta}{2} \int_{2T/5}^{3T/5} \int_0^{a+\epsilon/2} f^2e^{-M\left((a+\epsilon/2)^2 - x^2 \right)}  \ dx \ dt \leq C\int_0^T \int_{a+\epsilon/2}^{a+3\epsilon/4} |f|^2e^{-2\psi} + |\partial_x f|^2 e^{-2\psi} \ dx \ dt.
    \end{align*}
The integral involving the first order term on the right-hand side can be bounded from above using Cacciopoli's inequality (see \textit{e.g.} \cite[Lemma 2.9]{ervedoza2008control}) since $\psi > 0$ diverges as $t \rightarrow 0,T$,
\begin{align*}
    \int_0^T \int_{a+\epsilon/2}^{a+3\epsilon/4} |\partial_x f|^2 e^{-2\psi}  \ dx \ dt \leq c(\epsilon) \int_0^T \int_a^{a+\epsilon} |f|^2 \ dx \ dt.
\end{align*}

Since the weight $\psi$ is positive, and using that
\begin{align*}
    e^{-M\left((a+\epsilon/2)^2 - x^2 \right)}  \geq e^{-\frac{\xi}{\sqrt{1-3\delta/2}}(a+\epsilon/2)^2},
\end{align*}
we deduce the sought inequality \eqref{eqn: appendix cutoff inequality near sing}.

\subsubsection*{Analysis away from the singularity: proof of \eqref{eqn: appendix cutoff inequality away sing}}

We now turn to the proof of \eqref{eqn: appendix cutoff inequality away sing}. We use the following boundary Carleman estimate from \cite[Proposition 3.1]{beauchard2020minimal}.

\begin{proposition}\label{prop: appendix carleman beauchard}
    Let $T > 0$, $\alpha, \beta \in \mathbb{R}$, $\alpha < \beta$, $q \in C^1 \left([\alpha,\beta],\mathbb{R} \right)$, $\xi > 0$ and $\psi \in C^2((0,T);C^4([\alpha,\beta]))$ a function satisfying for every $x \in (\alpha,\beta)$
    \begin{align}
        \begin{array}{lllll}
            \displaystyle \underset{t \rightarrow 0^+}{\lim} \underset{x \in [\alpha,\beta]}{\inf} \psi(t,x) &=& \displaystyle \underset{t \rightarrow T^-}{\lim} \underset{x \in [\alpha,\beta]}{\inf}\psi(t,x)   & = & \infty,  \\[12pt]
            \displaystyle \underset{t \rightarrow 0^+}{\lim} \partial_x \psi(t,x) e^{-\psi(t,x)} &=& \underset{t \rightarrow T^-}{\lim} \partial_x \psi(t,x) e^{-\psi(t,x)}  & = & 0.
        \end{array}
    \end{align}
    Then, for any solution $f$ of system 
    \begin{align}\label{stm : appendix adjoint system non singular classical d=1 fourier}
    \left\{ \begin{array}{lcll}
     \partial_t f -\partial_x^2f + \xi^2q(x)^2 f & = & f_\xi, & (t,x) \in (0,T) \times (\alpha,\beta),\\[6pt]
     f(t,x) & = & 0, & t \in (0,T), \ x \in \{\alpha,\beta\}, \\[6pt]
     f(0,x) & = & f_0(x), & x \in (\alpha,\beta),
    \end{array} \right.
    \end{align}
    with $f_0 \in H_0^1(\alpha,\beta)$, $f_\xi \in L^2((0,T) \times (\alpha,\beta))$, the function $g = f e^{-\psi}$ satisfies
    \begin{align}\label{eqn: appendix carleman estimate beauch}
        2\int_0^T \left[ |\partial_x g|^2 \partial_x\psi\right]_{x = \alpha}^{x = \beta} \ dt + \int_0^T \int_\alpha^\beta \left( -4\partial_x^2 \psi |\partial_x g|^2 + |g|^2 G_\psi \right) \ dx \ dt \leq \int_0^T \int_\alpha^\beta |f_\xi e^{-\psi}|^2 \ dx \ dt, 
    \end{align}
    where
    \begin{align}
        G_\psi(t,x) = 2\partial_x \psi \partial_x F_\psi - \partial_t F_\psi + \partial_x^4 \psi,
    \end{align}
    with
    \begin{align}
        F_\psi(t,x) = \partial_t\psi - |\partial_x \psi|^2 + \xi^2q(x)^2.
    \end{align}
\end{proposition}

We choose as a Carleman weight in Proposition \ref{prop: appendix carleman beauchard}, the function 
\begin{align}
    \psi_\xi(t,x) = A\xi\theta(t) + \theta(t) - \sqrt{\xi}\theta(t)\left( \frac{(x-a)^2}{2} - 2L(x-a)\right),
\end{align}
where $A > 0$ is arbitrary, and $\theta \in C^\infty(0,T)$, $T < 4$, satisfies 
\begin{align}\label{eqn: appendix def theta}
    \theta(t) = \left\{
    \begin{array}{ll}
        1/t, & t \in (0,T/4),  \\
        1 & t \in (T/3,2T/3),\\
        1/(T-t) & t \in t \in (3T/4,T),\\
        \geq 1, & t \in (0,T).
    \end{array}
    \right.
\end{align}

This choice is inspired by the one of \cite[Proposition 3.5]{beauchard2020minimal}, and is very much well-suited for us. Indeed, on $(a,L)$ the function $\psi_\xi$ is positive, increasing and concave. Thus, it produces a boundary term that is significant only at $x = a$, which is precisely where the cutoff $\chi_2$ vanishes. Moreover, it will produce for large $\xi$ a cost of the form $e^{\xi A}$, with $A$ arbitrary small. \\

Observe that $f_2 = \chi_2 f$, with $f$ solution of system \eqref{stm : appendix adjoint system classical d=1 fourier}, solves system \eqref{stm : appendix adjoint system non singular classical d=1 fourier} with $q(x) = x$ on $(a,L)$, and with a source term
\begin{align}\label{eqn: appendix second membre}
    f_\xi = \frac{1/4 - \nu^2}{x^2}f_2 - \partial_x^2\chi_2 f - 2 \partial_x \chi_2 \partial_x f. 
\end{align}

Let us make some preliminary computations before applying Proposition \ref{prop: appendix carleman beauchard}. We have
\begin{align}\label{eqn: appendix computation psi}
    \begin{array}{llll}
        \partial_x \psi_\xi (t,x) &=& -\sqrt{\xi}\theta(t) \left( x - a - 2L\right)&  > 0  \text{ on } (a,L), \\[8pt]
        \partial_x^2 \psi_\xi (t,x) &=& -\sqrt{\xi}\theta(t) & < 0 \text{ on } (a,L),\\[8pt]
        \partial_x \psi_\xi (t,a) &=& 2L\sqrt{\xi}\theta(t) & > 0 \\[8pt]
        \partial_x \psi_\xi (t,L) &=& \sqrt{\xi}\theta(t)(L+a) & > 0,  \\[8pt]
        \partial_t \psi_\xi(t,x) &=& A \xi \theta'(t) + \theta'(t) - \sqrt{\xi}\theta'(t)\left( \displaystyle\frac{(x-a)^2}{2} - 2L(x-a)\right).    
    \end{array}
\end{align}

It follows that
\begin{align*}
    \begin{array}{lll}
            F_{\psi_\xi} &=& A \xi \theta'(t) + \theta'(t) - \sqrt{\xi}\theta'(t)\left( \displaystyle\frac{(x-a)^2}{2} - 2L(x-a)\right) - \xi \theta(t)^2(x-a-2L)^2 + \xi^2 x^2,\\[8pt]
            \partial_x F_{\psi_\xi} &=& (a+2L-x) \left(\sqrt{\xi} \theta'(t) + 2\xi \theta(t)^2 \right) + 2\xi^2 x, \\[8pt]
            \partial_t F_{\psi_\xi} &=& A\xi \theta''(t) + \theta''(t) - \sqrt{\xi}\theta''(t)\left( \displaystyle\frac{(x-a)^2}{2} - 2L(x-a)\right) - 2\xi \theta'(t)\theta(t)(x-a-2L)^2,
    \end{array}
\end{align*}
and finally
\begin{align*}
    G_{\psi_\xi} &= (2L+a-x)^2\left(4\xi \theta'(t)\theta(t) + 4\xi^{3/2}\theta(t)^3 \right) + 4\xi^{5/2}\theta(t)x(2L+a-x) - A\xi \theta''(t) - \theta''(t) \notag \\
    &+ \sqrt{\xi}\theta''(t)\left( \displaystyle\frac{(x-a)^2}{2} - 2L(x-a)\right).
\end{align*}

Using that there exists $C > 0$ such that for every $t \in (0,T)$,
\begin{align*}
    |\theta'(t)| \leq C\theta(t)^2, \quad \quad |\theta''(t)| \leq C\theta(t)^3,
\end{align*}
we have, for $\xi$ sufficiently large, 
\begin{align}\label{eqn: appendix estimate G i}
    \begin{array}{lll}
      \left| (2L+a-x)^2 4\xi \theta'(t)\theta(t) - A\xi \theta''(t) - \theta''(t) 
    + \sqrt{\xi}\theta''(t)\left( \displaystyle\frac{(x-a)^2}{2} - 2L(x-a)\right) \right|   &\leq&  C\xi \theta(t)^3.
    \end{array}
\end{align}

Observe that we also have
\begin{align*}
    \begin{array}{lll}
       4(2L+a-x)^2 \xi^{3/2}\theta(t)^3    & \geq & 4(L+a)^2\xi^{3/2}\theta(t)^3, \\[8pt]
      4\xi^{5/2}\theta(t)x(2L+a-x) &\geq& 0,
    \end{array}
\end{align*}
from which we obtain, combined with \eqref{eqn: appendix estimate G i}, that for every $\xi$ sufficiently large, for every $(t,x) \in (0,T) \times (a,L)$,
\begin{align}\label{eqn: appendix lower bound G psi}
    G_{\psi_\xi}(t,x) \geq 2(L+a)^2\xi^{3/2}\theta(t)^3.
\end{align}

We can now apply Proposition \ref{prop: appendix carleman beauchard} to $f_2 = \chi_2f$. Observe that the boundary term at $x = a$ in \eqref{eqn: appendix carleman estimate beauch} vanishes due to the fact that $\chi_2 = 0$ on $(a,a+\epsilon/4)$, and the remaining boundary term at $x=L$ is positive due to \eqref{eqn: appendix computation psi}. Thus, we obtain from Proposition \ref{prop: appendix carleman beauchard}, using \eqref{eqn: appendix lower bound G psi} together with  \eqref{eqn: appendix second membre} and \eqref{eqn: appendix computation psi}, with $g = f_2e^{-\psi_\xi}$,
\begin{align*}
    \int_0^T \int_a^L 4\sqrt{\xi}\theta |\partial_x g|^2 &+ 2(L+a)^2\xi^{3/2}\theta^3|g|^2   \ dx \ dt \\
    &\leq \int_0^T \int_a^L \left| \frac{1/4 - \nu^2}{x^2}f_2 - \partial_x^2\chi_2 f - 2 \partial_x \chi_2 \partial_x f \right|^2 e^{-2\psi_\xi}\ dx \ dt. 
\end{align*}

The term in $1/x^2$ is bounded on $(a,L)$, so that by triangular and young inequalities we obtain, using that $\chi_2'$ and $\chi_2''$ are supported in $(a+\epsilon/4, a+\epsilon/2)$, 
\begin{align}\label{eqn: appendix almost final estimate away}
    \int_0^T \int_a^L 4\sqrt{\xi}\theta |\partial_x g|^2 &+ 2(L+a)^2\xi^{3/2}\theta^3|g|^2   \ dx \ dt \notag \\
    &\leq C \int_0^T \int_a^L |g|^2 + C \int_0^T \int_{a+\epsilon/4}^{a+\epsilon/2} \left( |f|^2 + |\partial_x f|^2 \right)e^{-2\psi_\xi}  \ dx \ dt, 
\end{align}
for some constant $C> 0$ depending on the cutoff and $\epsilon$.\\

The first term  of the right-hand side above can be absorbed by the second term of left-hand side for large $\xi$, observing that $\theta \geq 1$. The first order term in the right-hand side above can be bounded from above by Cacciopoli's inequality (see \textit{e.g.} \cite[Lemma 2.9]{ervedoza2008control}) since $\psi_\xi$ diverges as $t \rightarrow 0,T$,
\begin{align*}
    \int_0^T \int_{a+\epsilon/4}^{a+\epsilon/2} |\partial_x f|^2 e^{-2\psi_\xi}  \ dx \ dt \leq c(\epsilon) \int_0^T \int_{a+\epsilon/8}^{a+5\epsilon/8} |f|^2 \ dx \ dt.
\end{align*}

We thus deduce from \eqref{eqn: appendix almost final estimate away}, replacing $g$ by its definition, and using that it depends on the cutoff $\chi_2$ that values $1$ on $(a+\epsilon/2,L)$,
\begin{align}\label{eqn: appendix one of final estimate}
    2(L+a)^2\xi^{3/2} \int_0^T \int_{a+\epsilon/2}^L  |f|^2e^{-2\psi_\xi}   \ dx \ dt \leq c(\epsilon) \int_0^T \int_a^b |f|^2 \ dx \ dt,
\end{align}
which holds for every $\xi$ large enough. Now, for any $c> 1$, for every $\xi$ sufficiently large, for every $(t,x) \in (T/3,2T/3) \times (a,L)$, we have, since $\theta(t)=1$ on $(T/3,2T/3)$,
\begin{align*}
    \psi_\xi(t,x) &= A\xi + 1 - \sqrt{\xi}\left( \frac{(x-a)^2}{2} - 2L(x-a)\right) \\
    &\leq c(A\xi + 1).
\end{align*}
We thus obtain from \eqref{eqn: appendix one of final estimate} for some new constant $C'(\epsilon)$, and for every $\xi$ large enough,
\begin{align}
    \int_{T/3}^{2T/3} \int_{a+\epsilon/2}^L  |f|^2 \ dx \ dt \leq C'(\epsilon)e^{2cA\xi} \int_0^T \int_a^b |f|^2 \ dx \ dt.
\end{align}

This is exactly the sought estimate \eqref{eqn: appendix cutoff inequality away sing} by taking $A = \epsilon/2c$. \\

\subsection{Proofs of Proposition \ref{prop: appendix carleman near sing}}\label{section: appendix proof of proposition}

In the present proof, to simplify the reading, we shall denote by $Q := (0,T) \times (0,L)$, and use the notations $u_s = \partial_s u$ and $c(\nu) := \nu^2 - 1/4$. We follow the computations of \cite[Proposition 5]{cannarsa2014null}. \\

Recall our Carleman weight $\psi$, introduced in \eqref{eqn: appendix def carleman weight near sing},
\begin{align*}
    \psi(t,x) = \frac{M\theta(t)}{2}(L^2 - x^2),
\end{align*}
with $\theta(t)$ defined in \eqref{eqn: appendix def theta near sing}. \\

Let $f \in C^0([0,T];L^2(0,L)) \cap L^2((0,T);H_0^1(0,L))$ be a solution of system \eqref{stm : appendix adjoint system classical d=1 fourier source}, and set
\begin{align*}
    g(t,x) = f(t,x) e^{\psi(t,x)}.
\end{align*}

Setting $P_\psi = e^{-\psi}P_\xi(e^{\psi} \cdot)$, with
\begin{align*}
    P_\xi = \partial_t -\partial_x^2 + \xi^2x^2  + \displaystyle\frac{c(\nu)}{x^2},
\end{align*}

we obtain that 
\begin{align*}
    P_\psi g = \mathsf{P}_\xi^+g + \mathsf{P}_\xi^- g,
\end{align*}
with 
\begin{align*}
    \mathsf{P}_\xi^+g &= (\psi_t - \psi_x^2)g - g_{xx} + \left[ \xi^2x^2 + \frac{c(\nu)}{x^2} \right]g, \\[8pt]
    \mathsf{P}_\xi^-g &= g_t - 2\psi_xg_x - \psi_{xx}g.
\end{align*}
Thus, 
\begin{align}\label{eqn: appendinx bound cross product near sing}
    \langle \mathsf{P}_\xi^+g,\mathsf{P}_\xi^-g \rangle \leq \frac{1}{2}\int_Q e^{-2\psi}|f_\xi|^2 \ dQ.
\end{align}

Due to the properties of our Carleman weight, we have $\langle \mathsf{P}_\xi^+g,\mathsf{P}_\xi^-g \rangle  = D + B$, with the distributed part given by 
\begin{align}\label{eqn: appendix def distributed part}
    D &= -2\int_Q \psi_{xx}g_x^2 \ dQ - \int_Q \psi_{xxx}gg_x \ dQ - \frac{1}{2}\int_Q(\psi_{tt} - 2 \psi_x\psi_{xt})g^2 \ dQ \notag \\
    &+ \int_Q (\psi_t - \psi_x^2)_x \psi_x g^2 \ dQ + \int_Q \left[ \xi^2x^2 + \frac{c(\nu)}{x^2}\right]_x \psi_x g^2 \ dQ,
\end{align}
and the boundary terms given by (we don't disclose the time-dependent boundary terms since they vanish due to our choice of $\theta$ \eqref{eqn: appendix def theta near sing}; see \textit{e.g.} \cite[Proposition 5]{cannarsa2014null}) 
\begin{align}
    B = \left[ \int_0^T \psi_x g_x^2 \ dt \right]_{x=0}^{x=L}.
\end{align}

We need to bound from below the distributed part $D$. Before doing so, let us observe that by definition \eqref{eqn: appendix def theta near sing} of $\theta$, there exists $c(T) > 0$ such that
\begin{align}\label{eqn: appendix properties theta near sing}
    |\theta'(t)| \leq c(T)\theta(t)^{1+1/k}, \quad |\theta''(t)| \leq c(T)\theta(t)^{1+2/k}.
\end{align}

Moreover, we compute
\begin{align}\label{eqn: appendix derivatives weight}
    \psi_x = -M\theta(x)x, \quad \psi_{xx} = -M\theta(t), \quad \psi_t = \frac{M}{2} \theta_t(L^2-x^2).
\end{align}

Using \eqref{eqn: appendix derivatives weight}, we obtain from \eqref{eqn: appendix def distributed part}
\begin{align}
    D = \int_Q \left[ 2M\theta \left( g_x^2 + \frac{c(\nu)}{x^2}g^2\right) + 2M^3\theta^3x^2g^2 + 2M^2\theta_t\theta x^2 g^2 - \frac{M}{4}\theta_{tt}(L^2-x^2)g^2 - 2\xi^2x^2M\theta g  \right] \ dQ.
\end{align}

Let us now choose $\delta > 0$ arbitrary small. Then, by \eqref{eqn: appendix properties theta near sing}, there exists a constant $c(\delta,T) =2c(T)/\delta> 0$ such that, for every $M \geq c(\delta,T)$,
\begin{align*}
    \int_Q \left[ 2M^3\theta^3x^2g^2 + 2M^2\theta_t\theta x^2 g^2  \right] \ dQ \geq \int_Q (2-\delta) M^3\theta^3x^2g^2 \ dQ.
\end{align*}

Thus, for every $M \geq c(\delta,T)$, 
\begin{align}\label{eqn: appendix first lower bound distributed part}
    D \geq \int_Q \left[ 2M\theta \left( g_x^2 + \frac{c(\nu)}{x^2}g^2\right) + (2-\delta)M^3\theta^3x^2g^2 - \frac{M}{4}\theta_{tt}(L^2-x^2)g^2 - 2\xi^2x^2M\theta g^2  \right] \ dQ.
\end{align}

Recall the Hardy-Poincaré inequality. For every $m > 0$ and $0 <\eta < 2$, there exists a constant $C_0$ such that
\begin{align*}
    \int_0^L g_x^2 - \frac{g^2}{4x^2} \ dx \geq m\int_0^L \frac{g^2}{x^\eta} \ dx - C_0 \int_0^L g^2 \ dx.
\end{align*}
We thus deduce from \eqref{eqn: appendix first lower bound distributed part} and the above inequality with $m = 1$ that, since $c(\nu) > -1/4$,
\begin{align}\label{eqn: appendix lower bound D i}
    D \geq \int_Q &\left[ M\theta \left( g_x^2 + \frac{c(\nu)}{x^2}g^2\right) + M\theta \frac{g^2}{x^\eta} - C_0M\theta g^2 \right. \\
    &\left. + (2-\delta)M^3\theta^3x^2g^2 - \frac{M}{4}\theta_{tt}(L^2-x^2)g^2 - 2\xi^2x^2M\theta g^2  \right] \ dQ. \notag
\end{align}

We now fix $k = 1 + 2/\eta$ in the definition \eqref{eqn: appendix def theta near sing} of $\theta$. We also set $q = k$, and $q' = k/(k-1)$ its conjugate exponent. We observe that 
\begin{align}\label{eqn: appendix def exponent}
    q(1+2/k - 1/q') = 3 \quad \text{and} \quad \eta q/q' = 2.
\end{align}
Using \eqref{eqn: appendix properties theta near sing}, we have for some $\epsilon > 0$,
\begin{align*}
    &\left| \int_Q \left[ - C_0M\theta - \frac{M}{4}\theta_{tt}(L^2-x^2) \right]g^2 \ dQ \right| \leq c(T)L^2M \int_Q \theta^{1+2/k} g^2 \ dQ \\[8pt]
    &= c(T)L^2M \int_Q \left( \frac{1}{\epsilon}\theta^{1+2/k - 1/q'} x^{\eta/q'} g^{2/q} \right) \left( \epsilon \theta^{1/q'}x^{-\eta/q'}g^{2/q'} \right) \ dQ \\
    &\leq \frac{c(\epsilon)c(T)L^2M}{\epsilon^q} \int_Q \theta^{q(1+2/k-1/q')}x^{\eta q/q'}g^2 \ dQ + \epsilon^{q'}c(T)L^2M \int_Q \theta \frac{g^2}{x^\eta} \ dQ \\
    &=  \frac{c(\epsilon)c(T)L^2M}{\epsilon^q} \int_Q \theta^3x^2g^2 \ dQ + \epsilon^{q'}c(T)L^2M \int_Q \theta \frac{g^2}{x^\eta} \ dQ, \quad \text{by \eqref{eqn: appendix def exponent}.}
\end{align*}

We plug the above estimate into \eqref{eqn: appendix lower bound D i}, choosing $\epsilon$ such that $1-\epsilon^{q'}c(T)L^2 = 1/2$, to obtain 
\begin{align}\label{eqn: appendix lower bound D ii}
    D \geq \int_Q &\left[ M\theta \left( g_x^2 + \frac{c(\nu)}{x^2}g^2\right) + \frac{M\theta}{2} \frac{g^2}{x^\eta} + (2-\delta)M^3\theta^3x^2g^2 - 2\xi^2x^2M\theta g^2 -  \frac{c(\epsilon)c(T)L^2M}{\epsilon^q} \theta^3x^2g^2 \right] \ dQ. 
\end{align}
    
Now, there exists a new constant $c'(\delta,T,\epsilon)$,
\begin{align*}
    c'(\delta,T,\epsilon) := \left( \frac{c(\epsilon)c(T)L^2}{\delta \epsilon^q}\right)^{1/2},
\end{align*}
such that for every $M \geq c'(\delta,T,\epsilon)$,
\begin{align*}
    (2-\delta)M^3 - \frac{c(\epsilon)c(T)L^2M}{\epsilon^q} \geq 2(1-\delta)M^3.
\end{align*}

We thus deduce, for every $M \geq \max\{c(\delta,T),c'(\delta,T,\epsilon) \}$,
\begin{align}
    D \geq \int_Q &\left[ M\theta \left( g_x^2 + \frac{c(\nu)}{x^2}g^2\right) + \frac{M\theta}{2} \frac{g^2}{x^\eta} + 2(1-\delta)M^3\theta^3x^2g^2 - 2\xi^2x^2M\theta g^2\right] \ dQ. 
\end{align}
To conclude on the distributed part, we are left with dealing with the terms in $x^2$. First, by definition \eqref{eqn: appendix def theta near sing} of $\theta$, we have $\theta(t) \leq \theta(t)^3$ for every $t \in (0,T)$. Indeed, $\theta(t)^2 \geq 1$. We thus deduce that, 
\begin{align}\label{eqn: appendix lower bound D iii}
    D \geq \int_Q &\left[ M\theta \left( g_x^2 + \frac{c(\nu)}{x^2}g^2\right) + \frac{M\theta}{2} \frac{g^2}{x^\eta} + 2(1-\delta)M^3\theta^3x^2g^2 - 2\xi^2x^2M\theta^3 g^2\right] \ dQ. 
\end{align}
Next, we ask that 
\begin{align*}
    2(1-\delta)M^3 - 2\xi^2M \geq \delta M^3.
\end{align*}
This holds as long as 
\begin{align*}
    M \geq \frac{\xi}{\sqrt{1-3\delta/2}}.
\end{align*}
But recall that we must have $M \geq \max\{c(\delta,T),c'(\delta,T,\epsilon) \}$. Thus, for some $\xi_0 > 0$ sufficiently large so that 
\begin{align*}
    \frac{\xi_0}{\sqrt{1-3\delta/2}} \geq \max\{c(\delta,T),c'(\delta,T,\epsilon) \},
\end{align*}
for every $\xi \geq \xi_0$, for every $M \geq \xi/\sqrt{1-3\delta/2}$, we obtain from \eqref{eqn: appendix lower bound D iii}
\begin{align}\label{eqn: appendix lower bound D iv}
    D \geq \int_Q &\left[ M\theta \left( g_x^2 + \frac{c(\nu)}{x^2}g^2\right) + \frac{M\theta}{2} \frac{g^2}{x^\eta} + \delta M^3\theta^3x^2g^2 \right] \ dQ,
\end{align}
which concludes the treatment of the distributed part $D$
We now briefly treat the boundary term $B$,  
\begin{align}
    B = \left[ \int_0^T \psi_x g_x^2 \ dt \right]_{x=0}^{x=L}.
\end{align}
By \eqref{eqn: appendix derivatives weight}, we have $\psi_x \leq 0$. Hence,  
\begin{align}\label{eqn: appendix bound boundary term near singularity}
    B \geq \int_0^T \psi_x(t,L) g_x(t,L)^2 \ dt  = - \int_0^T ML\theta(t)g_x(t,L)^2 \ dt
\end{align}

Combining \eqref{eqn: appendix lower bound D iv} and \eqref{eqn: appendix bound boundary term near singularity} with \eqref{eqn: appendinx bound cross product near sing} concludes the proof.

\printbibliography

@inproceedings{morancey2015approximate,
  title={Approximate controllability for a 2D Grushin equation with potential having an internal singularity},
  author={Morancey, Morgan},
  booktitle={Annales de l'Institut Fourier},
  volume={65},
  number={4},
  pages={1525--1556},
  year={2015}
}

@book{OlverHandbook2010,
author = {Olver, Frank W. and Lozier, Daniel W. and Boisvert, Ronald F. and Clark, Charles W.},
title = {NIST Handbook of Mathematical Functions},
year = {2010},
isbn = {0521140633},
publisher = {Cambridge University Press},
address = {USA},
edition = {1st},
}

@article{allonsius2021analysis,
  title={Analysis of the null controllability of degenerate parabolic systems of Grushin type via the moments method},
  author={Allonsius, Damien and Boyer, Franck and Morancey, Morgan},
  journal={Journal of Evolution Equations},
  volume={21},
  number={4},
  pages={4799--4843},
  year={2021},
  publisher={Springer}
}

@article{LaurentLeauthaud23,
     author = {Camille Laurent and Matthieu L\'eautaud},
     title = {On uniform controllability of {1D} transport equations in the vanishing viscosity limit},
     journal = {Comptes Rendus. Math\'ematique},
     pages = {265--312},
     publisher = {Acad\'emie des sciences, Paris},
     volume = {361},
     year = {2023},
     doi = {10.5802/crmath.405},
     language = {en},
}

@book{coron2007control,
  title={Control and nonlinearity},
  author={Coron, Jean-Michel},
  number={136},
  year={2007},
  publisher={American Mathematical Soc.}
}

@article{bateman1953higher,
  title={Higher transcendental functions, volume II},
  author={Bateman, Harry and Erd{\'e}lyi, Arthur},
  journal={Bateman Manuscript Project) Mc Graw-Hill Book Company},
  volume={410},
  year={1953}
}

@inproceedings{beauchard2020minimal,
  title={Minimal time issues for the observability of Grushin-type equations},
  author={Beauchard, Karine and Dard{\'e}, J{\'e}r{\'e}mi and Ervedoza, Sylvain},
  booktitle={Annales de l'Institut Fourier},
  volume={70},
  number={1},
  pages={247--312},
  year={2020}
}

@article{romankummer2025,
  title={Low-lying eigenvalues in the semiclassical limit of a Schrödinger operator with an inverse square potential, and non-asymptotic a-zeros of Kummer functions},
  author={Vanlaere, Roman},
  journal={arXiv preprint arXiv:2511.20025},
  year={2025}
}

@article{cannarsa2014null,
  title={Null controllability in large time for the parabolic Grushin operator with singular potential},
  author={Cannarsa, Piermarco and Guglielmi, Roberto},
  journal={Geometric control theory and sub-Riemannian geometry},
  pages={87--102},
  year={2014},
  publisher={Springer}
}

@article{beauchard20152d,
  title={2D Grushin-type equations: minimal time and null controllable data},
  author={Beauchard, Karine and Miller, Luc and Morancey, Morgan},
  journal={Journal of Differential Equations},
  volume={259},
  number={11},
  pages={5813--5845},
  year={2015},
  publisher={Elsevier}
}

@article{vanlaere2025grushinlike,
      title={Non Null-Controllability Properties of the Grushin-Like Heat Equation on 2D-Manifolds},
  author={Vanlaere, Roman},
  journal={arXiv preprint arXiv:2503.00997},
  year={2025}
}

@article{beauchard2014null,
  title={Null controllability of Grushin-type operators in dimension two.},
  author={Beauchard, Karine and Cannarsa, Piermarco and Guglielmi, Roberto},
  journal={Journal of the European Mathematical Society (EMS Publishing)},
  volume={16},
  number={1},
  year={2014}
}

@article{ball1977strongly,
  title={Strongly continuous semigroups, weak solutions, and the variation of constants formula},
  author={Ball, John M},
  journal={Proceedings of the American Mathematical Society},
  volume={63},
  number={2},
  pages={370--373},
  year={1977}
}

@article{duprez2020control,
  title={Control of the Grushin equation: non-rectangular control region and minimal time},
  author={Duprez, Michel and Koenig, Armand},
  journal={ESAIM: Control, Optimisation and Calculus of Variations},
  volume={26},
  pages={3},
  year={2020},
  publisher={EDP Sciences}
}

@article{darde2023null,
  title={Null-controllability properties of the generalized two-dimensional Baouendi-Grushin equation with non-rectangular control sets},
  author={Dard{\'e}, J{\'e}r{\'e}mi and Koenig, Armand and Royer, Julien},
  journal={Annales Henri Lebesgue},
  volume={6},
  pages={1479--1522},
  year={2023}
}

@article{CRMATH_2017__355_12_1215_0,
     author = {Koenig, Armand},
     title = {Non-null-controllability of the {Grushin} operator in {2D}},
     journal = {Comptes Rendus. Math\'ematique},
     pages = {1215--1235},
     publisher = {Elsevier},
     volume = {355},
     number = {12},
     year = {2017},
     doi = {10.1016/j.crma.2017.10.021},
     language = {en},
     url = {http://www.numdam.org/articles/10.1016/j.crma.2017.10.021/}
}

@article{LISSYprolate2025,
title = {A non-controllability result for the half-heat equation on the whole line based on the prolate spheroidal wave functions and its application to the Grushin equation},
journal = {Journal of Differential Equations},
volume = {433},
pages = {113306},
year = {2025},
issn = {0022-0396},
doi = {https://doi.org/10.1016/j.jde.2025.113306},
url = {https://www.sciencedirect.com/science/article/pii/S002203962500333X},
author = {Pierre Lissy}
}

@article{vazquez2000hardy,
  title={The Hardy inequality and the asymptotic behaviour of the heat equation with an inverse-square potential},
  author={Vazquez, Juan Luis and Zuazua, Enrike},
  journal={Journal of Functional Analysis},
  volume={173},
  number={1},
  pages={103--153},
  year={2000},
  publisher={Elsevier}
}

@article{anh2016null,
  title={Null controllability in large time of a parabolic equation involving the Grushin operator with an inverse-square potential},
  author={Anh, Cung The and Toi, Vu Manh},
  journal={Nonlinear Differential Equations and Applications NoDEA},
  volume={23},
  number={2},
  pages={20},
  year={2016},
  publisher={Springer}
}

@article{beauchard2014inverse,
  title={Inverse source problem and null controllability for multidimensional parabolic operators of Grushin type},
  author={Beauchard, Karine and Cannarsa, Piermarco and Yamamoto, Masahiro},
  journal={Inverse Problems},
  volume={30},
  number={2},
  pages={025006},
  year={2014},
  publisher={IOP Publishing}
}

@article{tamekue2022null,
  title={Null controllability of the parabolic spherical Grushin equation},
  author={Tamekue, Cyprien},
  journal={ESAIM: Control, Optimisation and Calculus of Variations},
  volume={28},
  pages={70},
  year={2022},
  publisher={EDP Sciences}
}

@article{vancostenoble2008null,
  title={Null controllability for the heat equation with singular inverse-square potentials},
  author={Vancostenoble, Judith and Zuazua, Enrique},
  journal={Journal of Functional Analysis},
  volume={254},
  number={7},
  pages={1864--1902},
  year={2008},
  publisher={Elsevier}
}

@article{ervedoza2008control,
  title={Control and stabilization properties for a singular heat equation with an inverse-square potential},
  author={Ervedoza, Sylvain},
  journal={Communications in Partial Differential Equations},
  volume={33},
  number={11},
  pages={1996--2019},
  year={2008},
  publisher={Taylor \& Francis}
}

@article{lissy2025null,
  title={Null controllability of the 1D heat equation with interior inverse square potential},
  author={Lissy, Pierre and Lourme, Tanguy},
  journal={arXiv preprint arXiv:2505.07302},
  year={2025}
}

@article{beauchard2017heat,
  title={Heat equation on the Heisenberg group: observability and applications},
  author={Beauchard, Karine and Cannarsa, Piermarco},
  journal={Journal of Differential Equations},
  volume={262},
  number={8},
  pages={4475--4521},
  year={2017},
  publisher={Elsevier}
}

@article{LEROUSSEAU20163193,
title = {Null-controllability of the Kolmogorov equation in the whole phase space},
journal = {Journal of Differential Equations},
volume = {260},
number = {4},
pages = {3193-3233},
year = {2016},
issn = {0022-0396},
doi = {https://doi.org/10.1016/j.jde.2015.09.062},
url = {https://www.sciencedirect.com/science/article/pii/S0022039615005367},
author = {Jérôme {Le Rousseau} and Iván Moyano},
}

@article{beauchardKolmog2015,
	author = {{Beauchard, Karine} and {Helffer, Bernard} and {Henry, Raphael} and {Robbiano, Luc}},
	title = {Degenerate parabolic operators of Kolmogorov type with a
          geometric control condition∗},
	DOI= "10.1051/cocv/2014035",
	url= "https://doi.org/10.1051/cocv/2014035",
	journal = {ESAIM: COCV},
	year = 2015,
	volume = 21,
	number = 2,
	pages = "487-512",
	month = "",
}

@article{koenigfractionalheat,
author = {Koenig, Armand},
title = {Lack of Null-Controllability for the Fractional Heat Equation and Related Equations},
journal = {SIAM Journal on Control and Optimization},
volume = {58},
number = {6},
pages = {3130-3160},
year = {2020},
doi = {10.1137/19M1256610},

URL = { https://doi.org/10.1137/19M1256610},
eprint = {},
}

@article{beauchardkolmogorov,
  title={Null controllability of Kolmogorov-type equations},
  author={Beauchard, Karine},
  journal={Mathematics of Control, Signals, and Systems},
  volume={26},
  number={1},
  pages={145--176},
  year={2014},
  publisher={Springer}
}

@article{DardeRoyerKolmogorov,
     author = {Jeremi Darde and Julien Royer},
     title = {Critical time for the observability of {Kolmogorov-type} equations},
     journal = {Journal de l'Ecole polytechnique Mathematiques},
     pages = {859--894},
     publisher = {Ecole polytechnique},
     volume = {8},
     year = {2021},
     doi = {10.5802/jep.160},
     language = {en},
     url = {https://jep.centre-mersenne.org/articles/10.5802/jep.160/}
}

@article{prandi2018quantum,
  title={Quantum confinement on non-complete Riemannian manifolds},
   volume={8},
   ISSN={1664-0403},
   url={http://dx.doi.org/10.4171/JST/226},
   DOI={10.4171/jst/226},
   number={4},
   journal={Journal of Spectral Theory},
   publisher={European Mathematical Society - EMS - Publishing House GmbH},
   author={Prandi, Dario and Rizzi, Luca and Seri, Marcello},
   year={2018},
   month=jul, pages={1221–1280} }

@incollection{berger2006spectre,
  title={Le spectre d'une vari{\'e}t{\'e} riemannienne},
  author={Berger, Marcel and Gauduchon, Paul and Mazet, Edmond},
  booktitle={Le spectre d’une vari{\'e}t{\'e} Riemannienne},
  pages={141--241},
  year={2006},
  publisher={Springer}
}

@article{cazacu2014controllability,
  title={Controllability of the heat equation with an inverse-square potential localized on the boundary},
  author={Cazacu, Cristian},
  journal={SIAM Journal on Control and Optimization},
  volume={52},
  number={4},
  pages={2055--2089},
  year={2014},
  publisher={SIAM}
}

@article{DardeTrabut2025,
  title={On the critical time of observability of the multi-dimensional Baouendi-Grushin equation},
  author={Dard{\'e}, J{\'e}r{\'e}mi and Trabut, Mathilda},
  journal={arXiv preprint arXiv:2604.01419},
  year={2026}
}

@article{miller2010direct,
  title={A direct Lebeau-Robbiano strategy for the observability of heat-like semigroups},
  author={Miller, Luc},
  journal={Discrete and Continuous Dynamical Systems-Series B},
  volume={14},
  number={4},
  pages={1465--1485},
  year={2010}
}

@article{jerison1999nodal,
  title={Nodal Sets of Sums of},
  author={Jerison, David},
  journal={Harmonic Analysis and Partial Differential Equations: Essays in Honor of Alberto P. Calderon},
  pages={223},
  year={1999},
  publisher={University of Chicago Press}
}

@book{agrachev2019comprehensive,
  title={A comprehensive introduction to sub-Riemannian geometry},
  author={Agrachev, Andrei and Barilari, Davide and Boscain, Ugo},
  volume={181},
  year={2019},
  publisher={Cambridge University Press}
}

@book{hardy1952inequalities,
  title={Inequalities},
  author={Hardy, Godfrey Harold and Littlewood, John Edensor and P{\'o}lya, George},
  year={1952},
  publisher={Cambridge university press}
}

@article{de2024weyl,
  title={Weyl formulae for some singular metrics with application to acoustic modes in gas giants},
  author={de Verd{\`\i}{\`e}re, Yves Colin and Dietze, Charlotte and de Hoop, Maarten V and Tr{\'e}lat, Emmanuel},
  journal={arXiv preprint arXiv:2406.19734},
  year={2024}
}

@article{dietze2024concentration,
  title={Concentration of eigenfunctions on singular Riemannian manifolds},
  author={Dietze, Charlotte and Read, Larry},
  journal={arXiv preprint arXiv:2410.20563},
  year={2024}
}

@article{dietze2025critical,
  title={The critical case for the concentration of eigenfunctions on singular Riemannian manifolds},
  author={Dietze, Charlotte},
  journal={arXiv preprint arXiv:2510.23520},
  year={2025}
}

@book{zabczyk2020mathematical,
  title={Mathematical control theory},
  author={Zabczyk, Jerzy},
  year={2020},
  publisher={Springer}
}

@book{berezin2012schrodinger,
  title={The Schr{\"o}dinger Equation},
  author={Berezin, Feliks Aleksandrovich and Shubin, Mikhail},
  volume={66},
  year={2012},
  publisher={Springer Science \& Business Media}
}

@article{chitour2024weyl,
  title={Weyl's law for singular Riemannian manifolds},
  author={Chitour, Yacine and Prandi, Dario and Rizzi, Luca},
  journal={Journal de Math{\'e}matiques Pures et Appliqu{\'e}es},
  volume={181},
  pages={113--151},
  year={2024},
  publisher={Elsevier}
}

@article{boscain2016spectral,
  title={Spectral analysis and the Aharonov-Bohm effect on certain almost-Riemannian manifolds},
  author={Boscain, Ugo and Prandi, Dario and Seri, Marcello},
  journal={Communications in Partial Differential Equations},
  volume={41},
  number={1},
  pages={32--50},
  year={2016},
  publisher={Taylor \& Francis}
}

\end{document}